\documentclass[reqno,12pt]{amsart}
\usepackage[utf8]{inputenc}
\usepackage[left=1.2in, right=1.2in, top=1in]{geometry}

\usepackage{amsmath,amssymb,amsthm,mathrsfs,color,times,textcomp,verbatim}
\usepackage{xcolor}
\usepackage[colorlinks=true]{hyperref}
\hypersetup{urlcolor=blue, citecolor=red, linkcolor=blue}
\usepackage[square,numbers]{natbib}

\numberwithin{equation}{section}
\theoremstyle{plain}
\newtheorem{theorem}{Theorem}[section]
\newtheorem{proposition}[theorem]{Proposition}
\newtheorem{lemma}[theorem]{Lemma}

\newtheorem{corollary}[theorem]{Corollary}
\newtheorem{claim}{Claim}

\usepackage{graphicx}

\theoremstyle{definition}
\newtheorem{definition}[theorem]{Definition}

\newtheorem{remark}[theorem]{Remark}

\title[Sharp quantitative estimates]{Sharp quantitative estimates of Struwe's Decomposition}
\author{Bin Deng}
\address{School of Mathematical Sciences, University of Science and Technology of China,
Hefei, Anhui Province, P.R. China, 230026}
\email{bingomat@mail.ustc.edu.cn}
\author{Liming Sun}
\address{Department of Mathematics, University of British Columbia, Vancouver, BC, V6T 1Z2, CA.}
\email{lsun@math.ubc.ca}
\author{Jun-cheng Wei}
\address{Department of Mathematics, University of British Columbia, Vancouver, BC, V6T 1Z2, CA.}
\email{jcwei@math.ubc.ca}
\date{\today \,(Last Typeset)}

\def\l{\lambda}
\def\Rn{\mathbb{R}^n}
\def\al{\alpha}

\begin{document}

\maketitle
\begin{abstract}
Suppose $u\in \dot{H}^1(\mathbb{R}^n)$. In a seminal work,  Struwe  proved that if $u\geq 0$ and $\|\Delta u+u^{\frac{n+2}{n-2}}\|_{H^{-1}}:=\Gamma(u)\to 0$ then $dist(u,\mathcal{T})\to 0$, where $dist(u,\mathcal{T})$ denotes the $\dot{H}^1(\mathbb{R}^n)$-distance of $u$ from the manifold of sums of Talenti bubbles.  Ciraolo, Figalli  and Maggi obtained the first quantitative version of Struwe's decomposition with one bubble in all dimensions, namely $\delta (u) \leq C \Gamma (u)$. For Struwe's decomposition with two or more bubbles, Figalli and Glaudo  showed a striking  dimensional dependent quantitative estimate, namely  $\delta(u)\leq C \Gamma(u)$ when $3\leq n\leq 5$ while this is false for $ n\geq 6$.  In this paper, we show that
    \[dist (u,\mathcal{T})\leq C\begin{cases} \Gamma(u)\left|\log \Gamma(u)\right|^{\frac{1}{2}}\quad&\text{if }n=6,\\
    |\Gamma(u)|^{\frac{n+2}{2(n-2)}}\quad&\text{if }n\geq 7.\end{cases}\]
   Furthermore, we show that this inequality is sharp.
\end{abstract}
\section{Introduction}
\subsection{Motivation and main results} The Sobolev inequality with exponent 2 states that, for any $n \geq 3$ and any $u \in \dot{H}^{1}\left(\mathbb{R}^{n}\right):=D^{1,2}\left(\mathbb{R}^{n}\right) $, it holds that
\begin{align}\label{eq:1.1}
S\|u\|_{L^{2^{*}}} \leq\|\nabla u\|_{L^{2}},
\end{align}
where $2^*=\frac{2n}{n-2}$ and $S = S(n)$ is a dimensional constant.

It is well-known that the Euler-Lagrange equation of \eqref{eq:1.1}, i.e., critical points of Sobolev inequality, up to scaling, is given by
\begin{align}\label{E-L}
    \Delta u+|u|^{p-1}u=0\ \ \ \ \ \ \mbox{in} \ \mathbb{R}^n.
\end{align}
Throughout this paper, we denote $p=\frac{n+2}{n-2}$.
By \citet{caffarelli1989} and \citet{gidas1979symmetry}, it is known that all the positive solutions are \textit{Talenti bubbles} \cite{talenti1976}, i.e.
\begin{align}\label{bubble}
    U[z, \lambda](x):=(n(n-2))^{\frac{n-2}{4}}  \left(\frac{\lambda}{1+\lambda^{2}|x-z|^{2}}\right)^{\frac{n-2}{2}}.
\end{align}
These are all the minimizers of the Sobolev inequality, up to scaling. There are many interests regarding to the stability of \eqref{eq:1.1}. From the perspective of discrepancy in the Sobolev inequality,  \citet{bianchi1991note} gave a quantitative estimate near the minimizers, that is
\begin{align}
\label{bianchi}
    \inf _{z \in \mathbb{R}^{n}, \lambda>0,\alpha\in \mathbb{R}}\|\nabla(u-\alpha U[z, \lambda])\|_{L^{2}}^{2} \leq C(n)\left(\|\nabla u\|_{L^{2}}^{2}-S^{2}\|u\|^2_{L^{2^*}}\right).
\end{align}

A natural  and more challenging perspective is through Euler-Lagrange equation, that is, whether a function $u$ that almost solves \eqref{E-L} must be quantitatively close to Talenti bubbles. There are many obstacles to address this question. First, \eqref{E-L} has many other sign-changing solutions \cite{delPino2011,weiyue1986conformally}. Second, $u$ could be the sum of many weakly interacting Talenti bubbles even if we restrict to the non-negative functions. In fact, a seminal work of Struwe \cite{struwe1984global} showed this is always the case, at least for non-negative functions.

\begin{theorem}[\citet{struwe1984global}]\label{thm:str}
Let $n \geq 3$ and $\nu \geq 1$ be positive integers. Let $\left(u_{k}\right)_{k \in \mathbb{N}} \subseteq \dot{H}^{1}\left(\mathbb{R}^{n}\right)$ be a sequence of non-negative functions such that $\left(\nu-\frac{1}{2}\right) S^{n} \leq \int_{\mathbb{R}^{n}}\left|\nabla u_{k}\right|^{2} \leq\left(\nu+\frac{1}{2}\right) S^{n}$ with $S=S(n)$ as in \eqref{eq:1.1}, and assume that
\[\left\|\Delta u_{k}+u_{k}^{2^{*}-1}\right\|_{H^{-1}} \rightarrow 0 \quad \text { as } k \rightarrow \infty.\]
Then there exist a sequence $(z_{1}^{(k)}, \ldots, z_{\nu}^{(k)})_{k \in \mathbb{N}}$ of $\nu$-tuples of points in $\mathbb{R}^{n}$ and a sequence $(\lambda_{1}^{(k)}, \ldots, \lambda_{\nu}^{(k)})_{k \in \mathbb{N}}$
of $\nu$-tuples of positive real numbers such that
\[
\left\|\nabla\left(u_{k}-\sum_{i=1}^{\nu} U[z_{i}^{(k)}, \lambda_{i}^{(k)}]\right)\right\|_{L^{2}} \rightarrow 0 \quad \text { as } k \rightarrow \infty.
\]
\end{theorem}

One can also show that the family of $U[z_{i}^{(k)}, \lambda_{i}^{(k)}]$ has the so-called weak interaction.
\begin{definition}[Interaction of Talenti bubbles]
Let $U[z_i,\l_i]$ and $U[z_j,\l_j]$ be two bubbles. Define the interaction of them by
\begin{align}\label{inter-bubble}
    q(z_i,z_j,\lambda_i,\lambda_j)=\left(\frac{\l_i}{\l_j}+\frac{\l_j}{\l_i}+\l_i\l_j|z_i-z_j|^2\right)^{-\frac{n-2}{2}}.
\end{align}
We shall write $q_{ij}=q(z_i,z_j,\lambda_i,\lambda_j)$ if there is no confusion. Let $\left(U_{i}=U\left[z_{i},\lambda_{i}\right]\right)_{1\leq i\leq\nu}$ be a family of Talenti bubbles. We say that the family is $\delta$-interacting if \begin{align}\label{delta-interacting}
    Q:=\max\{q_{ij}:i,j=1,\cdots,\nu\}<\delta.
\end{align}
\end{definition}
Despite the difficulty of non-negativity issue, one can still investigate the problem locally. That is, if $u$ is already near to a sum of weakly-interacting Talenti bubbles in $\dot H^1$-norm, then $\|\Delta u+u|u|^{p-1}\|_{H^{-1}}$ should control the distance from $\mathcal{T}$, the manifold of sums of weakly-interacting Talenti bubbles. Along this direction, 
\citet{ciraolo2018quantitative} obtained the first quantitative estimate $dist(u,\mathcal{T})\leq C\|\Delta u+u|u|^{p-1}\|_{L^{\frac{2n}{n+2}}}$ for all $n\geq 3$ when $\nu=1$, i.e., when only one bubble is present. Later, \citet{figalli2020sharp} established the following theorem for any finite number of bubbles.
\begin{theorem}[\citet{figalli2020sharp}]
For any dimension $3 \leq n \leq 5$ and $\nu \in \mathbb{N},$ there exist a small constant $\delta=\delta(n, \nu)>0$ and a large constant $C=C(n, \nu)>0$ such that the following statement holds. Let $u \in \dot{H}^{1}\left(\mathbb{R}^{n}\right)$ be a function such that
\[
\left\|\nabla u-\sum_{i=1}^{\nu} \nabla \tilde{U}_{i}\right\|_{L^{2}} \leq \delta,
\]
where $(\tilde{U}_{i})_{1 \leq i \leq \nu}$ is a $\delta$-interacting family of Talenti bubbles. Then there exist $\nu$ Talenti bubbles $U_{1}, U_{2}, \ldots, U_{\nu}$ such that
\begin{align}\label{35-r}
\left\|\nabla u-\sum_{i=1}^{\nu} \nabla U_{i}\right\|_{L^{2}} \leq C\left\|\Delta u+u|u|^{p-1}\right\|_{H^{-1}}.
\end{align}
\end{theorem}

 However,  Figalli and Glaudo constructed some counter examples that \eqref{35-r} does not work in $n\geq6$ if $\nu>1$. They conjectured that one needs to modify the RHS of \eqref{35-r} to $\Gamma |\log \Gamma|$ when $n=6$ and to $|\Gamma|^{\gamma}$ for some $\gamma<1$ when $n\geq7$, where $\Gamma=\left\|\Delta u+u|u|^{p-1}\right\|_{H^{-1}}$. However, the exact value of $\gamma$ is not known.

In this paper, we give affirmative answers to both of them.

\begin{theorem}\label{thm:main}
Suppose $n\geq 6$. There exist a small enough $\delta =\delta(n,\nu)>0$ and a large constant $C=C(n, \nu)>0$ such that the following statement holds. Let $u \in \dot{H}^{1}\left(\mathbb{R}^{n}\right)$ be a function such that
\begin{align}\label{u-cond}
\left\|\nabla u-\sum_{i=1}^{\nu} \nabla \tilde{U}_{i}\right\|_{L^{2}} \leq \delta,
\end{align}
where $(\tilde{U}_{i})_{1 \leq i \leq \nu}$ is a $\delta$-interacting family of Talenti bubbles. Then there exist $\nu$ Talenti bubbles $U_{1}, U_{2}, \ldots, U_{\nu}$ such that
\begin{align}\label{f-ieq}
\left\|\nabla u-\sum_{i=1}^{\nu} \nabla U_{i}\right\|_{L^{2}}  \leq C\begin{cases} \Gamma |\log \Gamma|^{\frac12}&\quad\text{if }\ n=6,\\ \Gamma^{\frac{p}{2}}&\quad\text{if }\ n\geq 7,\end{cases}
\end{align}
for $\Gamma=\left\|\Delta u+u|u|^{p-1}\right\|_{H^{-1}}$.
\end{theorem}
Note that our theorem completely solves the remaining cases in higher dimensions $n\geq 6$. Moreover, we improve the conjecture of \cite{figalli2020sharp} when $n=6$. After finding out this intriguing power $\frac{p}{2}$, we went back to check the counter examples in \cite{figalli2020sharp}. Their examples show that there exists $u\in \dot H^1(\mathbb{R}^n)$ when $n=7$ such that
\[\inf_{\substack{z_1,\cdots,z_\nu\in \mathbb{R}^n\\ \l_1,\cdots,\l_\nu>0}}\|\nabla u-\sum_{i}\nabla U_i\|_{L^2}\gtrsim\Gamma^{\frac{9}{10}}.\]
Notice the fact that $\frac{9}{10}=\frac{p}{2}$ when $n=7$ exactly implies that \eqref{f-ieq} is sharp in this case. Indeed, we can prove that our result \eqref{f-ieq} is sharp for all $n\geq 6$.
\begin{theorem}\label{thm:7counter}
For sufficiently large $R>0$, there exists some $\rho$ such that if  $u=U[-Re_1,1]+U[Re_1,1]+\rho$, then
\begin{align}
   \inf_{\substack{z_1,z_2\in \mathbb{R}^n\\ \l_1,\l_2>0}}\|\nabla u-\sum_{i=1}^2\nabla U[z_i,\l_i]\|_{L^2}\gtrsim \begin{cases} \Gamma |\log \Gamma|^{\frac12}&\quad\text{if }\ n=6,\\ \Gamma^{\frac{p}{2}}&\quad\text{if }\ n\geq 7. \end{cases}
\end{align}
\end{theorem}

As a consequence of Theorem \ref{thm:main}, we obtain the following sharp quantitative estimates of Struwe's decomposition.
\begin{corollary}\label{cor:struwe}
Suppose $n\geq 6$. There exist a small enough $\delta =\delta(n,\nu)>0$ and a large constant $C=C(n, \nu)>0$ such that the following statement holds. For any non-negative function $u \in \dot H^{1}\left(\mathbb{R}^{n}\right)$ such that
\[
\left(\nu-\frac{1}{2}\right) S^{n} \leq \int_{\mathbb{R}^{n}}|\nabla u|^{2} \leq\left(\nu+\frac{1}{2}\right) S^{n},
\]
there exist $\nu$ Talenti bubbles $U_{1}, U_{2}, \ldots, U_{\nu}$ such that
\begin{align}
\left\|\nabla u-\sum_{i=1}^{\nu} \nabla U_{i}\right\|_{L^{2}}  \leq C\begin{cases} \Gamma |\log \Gamma|^{\frac12}&\quad\text{if }\ n=6,\\ \Gamma^{\frac{p}{2}}&\quad\text{if }\ n\geq 7,\end{cases}
\end{align}
for $\Gamma=\left\|\Delta u+u|u|^{p-1}\right\|_{H^{-1}}$. Furthermore, for any $i \neq j$, the interaction between the bubbles can be estimated as
\[
\int_{\mathbb{R}^{n}} U_{i}^{p} U_{j} \leq C\left\|\Delta u+u^{p}\right\|_{H^{-1}}.
\]

\end{corollary}
\bigskip

Finally we remark that recently there has been a growing interest in understanding quantitative stability for functional and geometric inequalities,  due to important applications to problems in the calculus of variations and PDEs. For extension of (\ref{bianchi}) to Sobolev inequality  with general exponents we refer to \citet{figalli2019} and the references therein. Stability results on Sobolev inequality  can be used to obtain quantitative rates of convergence for fast diffusion equations. We refer to \citet{bonforte2020, DelPinoSaez2001} and the references therein.
There is also  a rich literature on the study of  quantitative versions of the isoperimetric inequality and other geometric inequalities  analogous to the Sobolev inequality. (A nice description of comparison  between Sobolev inequality and isoperimetric inequality can be found in \citet{figalli2020sharp}.)  We refer to \citet{BPV2015, maggi2019, maggi2018, figalli2020sharp, fusco2008, maggi2008} and the references therein.

\subsection{Sketch of the proof} We briefly explain the ideas of our proof. Throughout this paper, we shall write that $a \lesssim b$ (resp. $a \gtrsim b)$ if $a \leq C b$ (resp. $C a \geq b)$ where $C$ is a constant depending only on the dimension $n$ and on the number of bubbles $\nu$. The constant $C$ may change line by line.  Also, we say that $a \approx b$ if $a \lesssim b$ and $a \gtrsim b$. The integral $\int$ always means $\int_{\mathbb{R}^n}$ unless specified. We always denote with $o(1)$ any quantity  that goes to 0 when $\delta$ goes to 0. The common notion $o(Q)$ means $o(Q)/Q$ goes to 0 when $Q$ goes to 0.

 Suppose $u$ satisfies \eqref{u-cond} with a family of $\delta$-interacting bubbles. Consider the following variational problem
\[\delta(u):=\inf_{\substack{z_1,\cdots,z_\nu\in \mathbb{R}^n\\\l_1,\cdots,\l_\nu>0}}\left\|\nabla u-\nabla\left(\sum_{i=1}^{\nu} U\left[\bar{z}_{i}, \bar{\lambda}_{i}\right]\right)\right\|_{L^{2}}.
\]
It is easy to know that (for instance, see \cite[Appendix A]{bahri1988nonlinear}) if $\delta$ is small enough then such an infimum is achieved by some
\begin{equation}
    \label{sigmadef}
\sigma:=\sum_{i=1}^{\nu} U\left[z_{i}, \lambda_{i}\right].
\end{equation}
Let us denote $U_{i}:=U\left[z_{i}, \lambda_{i}\right]$. Since the bubbles $\tilde{U}_{i}$ are $\delta$-interacting, the family $\left( U_{i}\right)_{1 \leq i \leq \nu}$ is $\delta^{\prime}$-interacting for some $\delta^{\prime}$ that goes to 0 as $\delta$ goes to 0.

Let $\rho:=u-\sigma$ be the difference between the original function and the best approximation. Then $\rho$ satisfies the equation
\begin{align}\label{rho-fir}
    \Delta \rho+[(\sigma+\rho)^p-\sigma^p]+\sigma^p-\sum_{i=1}^\nu U_i^p+f=0,
\end{align}
where $f=-\Delta u-u|u|^{p-1}$.
Moreover, $\rho$ also satisfies the following orthogonal conditions
\begin{align}\label{eq:rho-ortho}
\begin{split}
\int_{\mathbb{R}^{n}} \nabla \rho\cdot \nabla Z_{i}^a =0  ~\text { for any } 1 \leq i \leq \nu, 1\leq a\leq n+1,
\end{split}
\end{align}
where $Z_{i}^a$ are the (rescaled) derivatives of $U[z_i,\l_i]$ with respect to  $a$-th component of $z_i$ and $\l_i$ (defined in \eqref{eq:Z}).

The linearized operator of \eqref{rho-1} at $0$ is $\Delta+p\sigma^{p-1}$, which will have kernels when $\sigma$ is the sum of a family of weakly-interacting bubbles. The non-homogeneous term $\sigma^p-\sum_{i}U_i^p$ is the main data which encodes the interaction of bubbles. Since the linearized operator have kernels, $f$ should be interpreted as some Lagrange multiplier.
The key idea of this paper is to obtain a precise behavior of the first approximation of $\rho$.

To illustrate the main idea, we start with the easiest case. Assume $ U_i= U[z_i,1]$ is a family of $\delta$-interacting bubbles with the same height. Since $\delta$ is very small, $z_i$ are far from each other.  
Define  $R=\min\{\frac12|z_i-z_j|:i\neq j\}$ and then $Q\thickapprox R^{2-n}$.

By some standard finite-dimensional reduction method (see for example  \cite{delPino2003,wei2010infinitely}), given a family of $\left( U_{i}\right)_{1 \leq i \leq \nu}$ which is $\delta^{\prime}$-interacting, we can find a function $ \rho_0$ (in an appropriate space) and a family of scalars $ \left( c_a^i\right)$ such that we can solve
\begin{align}\label{rho0}
\begin{cases}
     \Delta \rho_{0}+\left[\left(\sigma+\rho_{0}\right)^{p}-\sigma^{p}\right]=-(\sigma^{p}-\sum_{j} U_{j}^{p})+\sum\limits_{i=1}^\nu\sum\limits_{a=1}^{n+1}c_{a}^i U_i^{p-1}Z_{i}^a,\\
    \int \nabla\rho_0 \cdot\nabla Z_{i}^{a}=0,\quad i=1,\cdots, \nu;\ a=1,\cdots, n+1.
\end{cases}
\end{align}
In the core of each bubble $U_j$, i.e. $\{x:|x-z_j|<R\}$, we have $U_j\gg U_k$ for $k\neq j$. Then
\begin{align}\label{intro:sUk}
    \left|\sigma^p-\sum_{k=1}^\nu U_k^p\right|\lesssim \sum_{k=1,k\neq j}^{\nu}U_j^{p-1}U_k\lesssim \frac{R^{2-n}}{1+|x-z_j|^4}=R^{2-n}\langle x-z_i\rangle^{-4}.
\end{align}
Here $\langle y\rangle=\sqrt{1+|y|^2}$. In the outer region $\Omega=\{x:|x-z_j|>R, \forall \, j=1,\cdots,\nu\}$,  one has
\[\left|\sigma^{p}-\sum_{k=1}^\nu U_{k}^{p}\right|(x) \lesssim \sum_{j=1}^\nu U_{j}^{p}(x) \lesssim \sum_{j=1}^\nu\frac{1}{1+|x-z_j|^{n+2}} \lesssim  \sum_{j=1}^\nu R^{-4}\langle x-z_j\rangle^{2-n}.\]
Thus the solution $\rho_0$ to \eqref{rho0} should have the following control
\begin{align}\label{intro:rho0bd}
    |\rho_{0}(x)| \lesssim \sum_{j=1}^\nu R^{2-n}\langle x-z_j\rangle ^{-2} \chi_{\left\{|x-z_j|<R\right\}}+R^{-4}\langle x-z_j\rangle^{4-n} \chi_{\Omega}.
\end{align}
Here $\chi_\Omega$ is the characteristic function for a set $\Omega$. This type of point-wise estimate is the key to our proof. Notice that $\rho_0$ decay very slow in the core of each bubble. Next, we shall multiply $\rho_0$ to \eqref{rho0}\footnote{One attempts to differentiate the RHS of \eqref{intro:rho0bd} and integrate to get \eqref{rhoest}. This could give a quick check of $n\geq 7$, but not $n=6$ because the integral on the outer region is divergent. Check Remark \ref{rmk:on-weight} for how we overcome this. },  and integrate  it to find the following estimate
\begin{equation}
    \label{rhoest}
\|\nabla\rho_0\|_{L^2}\lesssim\begin{cases}
R^{-4}|\log R|^{\frac12}\approx Q|\log Q|^{\frac12},\quad & n=6,\\
R^{2-n}R^{\frac{n-6}{2}}\approx Q^{\frac{p}{2}},\quad &n\geq 7.\end{cases}
\end{equation}

Here the dimension of the space plays an important role in the integration. Estimates (\ref{intro:rho0bd})-(\ref{rhoest}) show that the contributions to $dist(u,\mathcal{T})$ also come from the far away behavior of $u$. This is one of the main difficulties in obtaining the sharp quantitative estimates for Struwe's decomposition.

Now let $\rho=\rho_0+\rho_1$. Then \eqref{rho-fir} and \eqref{rho0} implies $\rho_1$ will satisfy
\begin{align}\label{eq:rho1}
\begin{cases}
 \Delta \rho_{1}+\left[\left(\sigma+\rho_{0}+\rho_{1}\right)^{p}-\left(\sigma+\rho_{0}\right)^{p}\right]+\sum\limits_{i=1}^\nu\sum\limits_{a=1}^{n+1} c_{a}^j U_{j}^{p-1} Z_{j}^a+f=0,\\
    \int \nabla\rho_1\cdot \nabla Z_j^a=0, \quad i=1,\cdots, \nu;\ a=1,\cdots, n+1.
\end{cases}
\end{align}
Observe the equation of $\rho_1$ does not contain the interaction term $\sigma^p-\sum_{i=1}^\nu U_i^p$. Therefore, $\rho_1$ should depend on a higher order of $Q$. Indeed, Proposition \ref{prop:rho1-gradient} proves that $\|\nabla\rho_1\|\lesssim Q^2+\|f\|_{H^{-1}}$.
Combining with previous estimate of $\nabla\rho_0$, we get
\[\|\nabla \rho\|_{L^2}\leq \|\nabla\rho_0\|_{L^2}+\|\nabla\rho_1\|_{L^2}\lesssim \|f\|_{H^{-1}}+\begin{cases}Q^{\frac{p}{2}},& n\geq 7,\\ Q|\log Q|^{\frac12},&n=6.\end{cases}\]
On the other hand, we shall multiply \eqref{rho-fir} by some appropriate $Z_{k}^{n+1}$ and integrate it to arrive (cf. Lemma \ref{lem:I1Uk})
\[Q\lesssim \|f\|_{H^{-1}}+\left|\int \sigma^{p-1}\rho U_k\right|+\int |\rho|^p U_k.\]
To establish the above estimates, unlike \cite{figalli2020sharp}, we did not use cut-off functions. Using the point-wise estimate \eqref{intro:rho0bd} of $\rho_0$, we can show that the last two terms are higher order in $Q$ and then  $Q\lesssim \|f\|_{H^{-1}}$. Consequently, $\|\nabla\rho\|_{L^2}\leq \|f\|_{H^{-1}}^{\frac{p}{2}}$ if $n\geq 7$ and $\|\nabla\rho\|_{L^2}\leq \|f\|_{H^{-1}}|\log \|f\|_{H^{-1}}|^{\frac12}$ if $n=6$. Thus one can establish Theorem \ref{thm:main} in this setting.

\bigskip

For a general family of bubbles $\sigma=\sum_i U[z_i,\l_i]$, things are much more complicate. Since we can not a priori assume that $ \lambda_i \sim \lambda_j$ or $  |z_i-z_j|<< \min (\lambda_i^{-1}, \lambda_j^{-1})$,  we may have bubble towers, bubble clusters, and these two may be  mixed. (We remark that there are many papers in the literature concerning the construction of  bubbling clusters or bubbling towers solutions. For bubbling towers we refer to \citet{DDM2003,MP2010,PJ2014} and the references therein. For bubbling clusters we refer to \citet{WY2007,wei2010infinitely} and the references therein.)  For instance, if $\sigma$ contains bubble towers (that is $\max\{\lambda_i/\lambda_j,\lambda_j/\lambda_i\}\gg \lambda_i\lambda_j|z_i-z_j|^2$), one neither knows the existence of $\rho_0$ satisfying \eqref{rho0}, nor has a simple control of the interaction as \eqref{intro:sUk}.
These two difficulties are deeply related. The strategy is to design a ``good" space for the interaction term $|\sigma^p-\sum_iU_i^p|$ so that \eqref{rho0} has a solution $\rho_0$ with the desired control. Choosing the right norm is a very delicate process. We start with just two bubbles and examine the magnitude of the interaction term $(U_i+U_j)^p-U_i^p-U_j^p$ on different regions of $\Rn$. Fortunately, we obtain a uniform norm $\|\cdot\|_{**}$ (cf. \eqref{h-t-all}) to handle the bubble tower and bubble cluster at the same time, which reduces the amount of work significantly. Then $\|\sigma^p-\sum_k U_k^p\|_{**}\leq C(n,\nu)$ follows from the bounds of all pairs by a simple inequality.

The existence of $\rho_0$ satisfying $\eqref{rho0}$ is based on some a prior estimates (cf. Lemma \ref{lem:prior-est}). To establish such estimates, we used blow-up argument and divide $\Rn$ into three types of region: inner, neck, and  exterior ones. We manage to prove the leading parts of $W(x)$ is a supersolution on the neck and exterior regions. This is the most crucial and technical part of the proof. After establishing the a prior estimate, we get the existence of $\rho_0$ from standard contraction mapping theorem. Consequently $\|\rho_0\|_*\leq C(n,\nu)$, which plays the corresponding role as \eqref{intro:rho0bd}. Then the rest of proof is almost the same.


We also construct an example which shows the exponents in  \eqref{f-ieq} are sharp. Suppose $\sigma=\tilde U_1+\tilde U_2$ where $\tilde U_1:= U[-Re_1,1]$ and $\tilde U_2=U[Re_1,1]$. By reduction theorem (see \cite{wei2010infinitely}), one can prove the existence of $\tilde \rho$ such that
\begin{align}
\begin{cases}
\Delta \tilde \rho+[(\sigma+\tilde \rho)^p-\sigma^p]+\sigma^p-\sum_{i=1,2}\tilde U_i^p+\sum_{i,a}\tilde c_a^i\tilde U_i^{p-1}\tilde Z_i^{a}=0,\\
    \int \tilde U_i^{p-1}\tilde Z_i^a \tilde \rho=0,  \quad i=1,\cdots, \nu;\ a=1,\cdots, n+1.
\end{cases}
\end{align}
Choose $u=\sigma+\tilde \rho$. Then $ \Delta u + |u|^{p-1} u=-\sum_{i,a}\tilde c_a^i\tilde U_i^{p-1}\tilde Z_i^{a}$. We manage to find a good approximation of $\tilde \rho$ in the interior and exterior region which shows that $\|\nabla \tilde \rho\|_{L^2}$ in the interior part is already no less than $Q^{\frac{p}{2}}$ when $n=7$ and $Q|\log Q|^{\frac12}$ when $n=6$.
\bigskip

The organization of the paper is as follows. In Section 2, we prove the main result assuming several crucial estimates on $\rho$ and $\nabla\rho$. Section 3 contains two subsections. The first one proves the existence of the first expansion of $\rho$ and its point-wise estimates. We use it to establish the estimates used in the proof of main result in the second subsection. In Section 4, we devote to constructing some example to show that Theorem \ref{thm:main} is sharp. The appendix consists of various integral estimates between bubbles and their derivatives.

\section{Proof of the main theorem}
In this section, we will prove the main Theorem \ref{thm:main} based on some crucial estimates, whose proofs are deferred to the next section.

Suppose $u=\sigma+\rho$ where $\sigma=\sum_{i=1}^{\nu}U_i$ is the best approximation (see \eqref{sigmadef}). Then
\begin{align}\label{eq:u-split}
\begin{split}
\Delta u+u|u|^{p-1}
=\Delta(\sigma+\rho)+(\sigma+\rho)|\sigma+\rho|^{p-1}
=\Delta \rho+p\sigma^{p-1}\rho+I_1+I_2,
\end{split}
\end{align}
where
\begin{align}\label{I1I2}
I_1=&\sigma^p-\sum_{i=1}^\nu U_i^p,\quad I_2=(\sigma+\rho)|\sigma+\rho|^{p-1}-\sigma^p-p\sigma^{p-1}\rho.
\end{align}
Let $f=-\Delta u-u|u|^{p-1}$. Then \eqref{eq:u-split} can be reorganized to
\begin{align}\label{rho-1}
\Delta \rho+p\sigma^{p-1}\rho+I_1+I_2+f=0.
\end{align}
If $n\geq 6$, then $p\in (1,2]$. We have the following elementary inequality
\[\left|(\sigma+\rho)| \sigma+\left.\rho\right|^{p-1}-\sigma^{p}-\left.p \sigma^{p-1} \rho\right|\lesssim| \rho\right|^{p}.\]
Thus
\[|I_2|\leq |\rho|^p.\]
Multiplying $Z^{n+1}_k$ (defined in \eqref{eq:Z}) to \eqref{rho-1} and integrating over $\mathbb{R}^n$ give
\begin{align*}
&\left|\int I_1 Z^{n+1}_k\right|\leq \int|fZ^{n+1}_k|+\left|\int p\sigma^{p-1} \rho Z^{n+1}_k\right|+\int |\rho|^p|Z^{n+1}_k|.
\end{align*}
Here we have used the orthogonal conditions \eqref{eq:rho-ortho}. Using $|Z_k^{n+1}|\lesssim U_k$ and $||U_k||_{H^{-1}}\leq C(n)$, we can apply H\"older inequality
\begin{align}\label{eq:xi-main-7}
\begin{split}
    &\left|\int I_1 Z^{n+1}_k\right|\lesssim \|f\|_{H^{-1}}+\left|\int \sigma^{p-1} \rho Z_k^{n+1}\right|+\int |\rho|^p |Z_k^{n+1}|.
\end{split}
\end{align}
\begin{lemma}\label{lem:I1Uk}
Suppose $u$ satisfies \eqref{u-cond} with $\delta$ small enough. Then
\begin{align}\label{I1LUK}
    \int I_1 Z_{k}^{n+1}=\int I_1\l_k\partial_{\l_k}U_k=\sum_{i=1,i\neq k}^\nu \int U_i^p\l_k\partial_{\l_k}U_k+o(Q).
\end{align}
\end{lemma}
\begin{proof}
The first identity follows from the definition $Z_{k}^{n+1}=\l_k\partial_{\l_k}U_k$. Denote $\sigma=\sigma_k+ U_k$ where $\sigma_k=\sum_{i=1,i\neq k}^\nu U_i$. We make the following decomposition
\begin{align*}
    \int I_1\l_k\partial_{\lambda_k} U_k=&\int (\sigma^p-\sum_{i=1}^\nu  U_i^p)\l_k\partial_{\l_k}U_k=J_1+J_2+J_3+J_4,
\end{align*}
where
\begin{align*}
    J_1=&\int_{\{\nu U_k\geq\sigma_k\}}(p U_k^{p-1}\sigma_k-\sum_{i=1,i\neq k}^\nu U_i^p)\l_k\partial_{\l_k}U_k,\\
    J_2=&\int_{\{\nu U_k\geq \sigma_k\}}(\sigma^p-U_k^p-pU_k^{p-1}\sigma_k)\l_k\partial_{\l_k}U_k,\\
     J_3=&\int_{\{\sigma_k>\nu U_k\}}(p\sigma_k^{p-1}U_k+\sigma_k^p-\sum_{i=1}^\nu  U_i^p)\l_k\partial_{\l_k}U_k,\\
    J_4=&\int_{\{\sigma_k>\nu U_k\}}(\sigma^p-\sigma_k^p-p\sigma_k^{p-1}U_k)\l_k\partial_{\l_k}U_k.
\end{align*}
Notice $|\l_k\partial_{\l_k}U_k|\lesssim U_k$.
Based on the inequality
\begin{align}\label{el-1}
    |(a+b)^p-a^p-pa^{p-1}b|\lesssim a^{p-2}b^2 \quad \text{if } a\geq b>0,
\end{align}
we have
\begin{align}\label{e-trick}
    |J_2|\lesssim \int_{\{\nu U_k>\sigma_k\}} U_k^{p-1}\sigma_k^2\lesssim\int U_k^{p-\epsilon}\sigma_k^{1+\epsilon}\approx Q^{1+\epsilon}.
\end{align}
Here $\epsilon>0$ is very small, such that $1+\epsilon<p-\epsilon$, and in the last step we have used Lemma \ref{lem:2int-est}. Similarly $|J_4|\lesssim Q^{1+\epsilon}$. For $J_3$,
\begin{align*}
    |J_3|=&\left|\int_{\{\sigma_k\geq \nu U_k\}}(p\sigma_k^{p-1}U_k+\sigma_k^p-\sum_{i=1}^\nu  U_i^p)\l_k\partial_{\l_k}U_k\right|\\
    \lesssim & \int_{\{\sigma_k\geq \nu U_k\}}p\sigma_k^{p-\epsilon}U_k^{1+\epsilon}+\int \left|\sigma_k^p-\sum_{i=1,i\neq k}^\nu U_i^p\right|U_k+\int_{\{\sigma_k\geq \nu U_k\}} U_k^{p+1}.
\end{align*}
Using an elementary inequality
\[\left|\sigma_k^p-\sum_{i=1,i\neq k}^\nu U_i^p\right|\lesssim \sum_{\substack{1\leq i<j\leq \nu\\i\neq k,j\neq k}}U_i^{p-1}U_j\]
and the triple integral estimate in Lemma \ref{lem:3int-est}, we have
\[|J_3|\lesssim Q^{1+\epsilon}+\int_{\{\sigma_k\geq \nu U_k\}}U_k^{p+1}.\]
Consider the second term on the RHS. Lemma \ref{lem:B-inf} implies
\[\int_{\{U_i\geq  U_k\}}U_k^{p+1}\leq \int  U_k^{p-\epsilon}\inf (U_i^{1+\epsilon},U_k^{1+\epsilon})=O(q_{ik}^{\frac{n}{n-2}}|\log q_{ik}|)=o(Q),\]
therefore
\[\int_{\{\sigma_k\geq \nu U_k\}}U_k^{p+1}\leq\sum_{i=1,i\neq k}^\nu \int_{\{U_i>U_k\}}U_k^{p+1}=o(Q).\]
Therefore $|J_3|=o(Q)$.
Now consider $J_1$,
\begin{align*}
    &J_1-\sum_{i=1,i\neq k}^{\nu}p\int U_k^{p-1}U_i\l_k\partial_{\l_k}U_k\\
    &=\int_{\{\nu U_k<\sigma_k\}}p U_k^{p-1}\sigma_k\l_k\partial_k U_k-\int_{\{\nu U_k>\sigma_k\}}\sum_{i=1,i\neq k}^\nu U_i^p\l_k\partial_{\l_k}U_k=o(Q).
\end{align*}
Here we applied the same trick in \eqref{e-trick} to obtain $o(Q)$.
With the above estimates of $J_i$, $i=1,2,3,4$, we can get
\begin{align}
    \int I_1\l_k\partial_{\l_k}U_k=\sum_{i=1,i\neq k}^\nu p\int U_k^{p-1}U_i\l_k\partial_{\l_k}U_k+o(Q).
\end{align}
Simple integration by parts shows that
\[p\int U_k^{p-1}U_i\l_k\partial_{\l_k}U_k=\int U_i^p\l_k\partial_{\l_k}U_k.\]
Thus \eqref{I1LUK} holds.
\end{proof}

Now let us go back to \eqref{eq:xi-main-7}. In the next section, we will prove two important estimates
\begin{align}\label{pre-est}
    \left|\int \sigma^{p-1} \rho Z_k^{n+1}\right|=o(Q)+\|f\|_{H^{-1}},\quad \int |\rho|^{p} |Z_k^{n+1}|=o(Q)+\|f\|_{H^{-1}},
\end{align}
in Lemma \ref{lem:srU-est}.
\begin{remark}
These two terms have rough bounds easily by H\"older inequality and Sobolev inequality. Indeed, for instance, when $n\geq 7$, as did in \cite[(3.31) and footnote 1]{figalli2020sharp},
\[\left|\int \sigma^{p-1} \rho Z_k^{n+1}\right|\lesssim \|\nabla\rho\|_{L^2}Q^{p-1}\]
\[\int |\rho|^{p} |Z_k^{n+1}|\lesssim \|\nabla\rho\|_{L^2}^p.\]
By Lemma \ref{lem:I1Uk} and the above two estimates, one can achieve
\begin{align}\label{q-f-tmp}
Q\lesssim \|\nabla \rho\|_{L^2}Q^{p-1}+\|\nabla \rho\|_{L^2}^p+\|f\|_{H^{-1}}.
\end{align}
Multiplying \eqref{rho-1} by $\rho$, the approach in \cite{figalli2020sharp} would induce
$\|\nabla\rho\|_{L^2}\lesssim Q^{p-1}+\|f\|_{H^{-1}}$. Plugging in this fact to \eqref{q-f-tmp}, we obtain $Q\lesssim Q^{2(p-1)}+\|f\|_{H^{-1}}$, one fails to conclude anything when $2(p-1)\leq 1$ (equivalent to $n\geq 10$).
This obstacle actually motivates us to have a better control of $\rho$ instead of simply using H\"older inequality and Sobolev inequality.
\end{remark}

Using \eqref{pre-est}, we can prove the following important lemma.
\begin{lemma} Suppose $u$ satisfies \eqref{u-cond} with $\delta$ small enough, then
\begin{align}
    Q\lesssim \|f\|_{H^{-1}}.
\end{align}
\end{lemma}
\begin{proof}Lemma \ref{lem:2int-est} implies $\int U_i^pU_j\approx q_{ij}$ for $i\neq j$. Notice that if $q_{ij}\ll Q$, then
\[\left|\int U_i^p\l_j\partial_{\l_j}U_j\right|\lesssim \int U_i^pU_j\lesssim q_{ij}\ll Q.\]
We shall collect all the $q_{ij}$ such that $q_{ij}\approx Q$ and choose $k$ such that $\lambda_k$ is the largest one among them. Notice Lemma \ref{lem:U12-est} implies $\int U_i^p \l_k\partial_{\l_k}U_k\approx -q_{ik}\approx -Q$. Since they have the same sign, Lemma \ref{lem:I1Uk} yields
\begin{align}\label{I1Q}
    \left|\int I_1\l_k\partial_{\l_k}U_k\right|\approx Q.
\end{align}
Plugging in this fact  and \eqref{pre-est} to \eqref{eq:xi-main-7} to obtain
\[Q\lesssim \|f\|_{H^{-1}}.\]

\end{proof}

Now we can prove our main result Theorem \ref{thm:main}.

\begin{proof}[Proof of Theorem \ref{thm:main}]

In the next section, we shall prove $\rho=\rho_0+\rho_1$. By Proposition \ref{prop:L2-rho0}, we have
\begin{align}
\|\nabla\rho_0\|_{L^2}\lesssim \begin{cases}Q|\log Q|^{\frac12}&\text{if }n=6,\\ Q^{\frac{p}{2}}&\text{if }n\geq 7.\end{cases}
\end{align}
By Proposition \ref{prop:rho1-gradient}, we have
\[\|\nabla\rho_1\|\lesssim Q^2+\|f\|_{H^{-1}}.\]
Since we have shown $Q\lesssim \|f\|_{H^{-1}}$ in the previous lemma, then
\[\|\nabla \rho\|_{L^2}\leq \|\nabla\rho_0\|_{L^2}+\|\nabla\rho_1\|_{L^2}\lesssim \begin{cases} \|f\|_{H^{-1}}\left|\log \|f\|_{H^{-1}}\right|^{\frac12},& n=6,\\\|f\|_{H^{-1}}^{\frac{p}{2}},&n\geq 7.\end{cases}\]
Here we have used the fact that $x|\log x|^{\frac12}$ is increasing near 0. Therefore \eqref{f-ieq} holds.
\end{proof}

\begin{proof}[Proof of Corollary \ref{cor:struwe}]
The proof is identical to that of Corollary 3.4  in \cite{figalli2020sharp}.

\end{proof}

It remains to establish Lemma \ref{lem:srU-est}, Proposition \ref{prop:L2-rho0} and Proposition \ref{prop:rho1-gradient}. As we have discussed in the introduction, these depend crucially on a point-wise estimate of $\rho_0$.




\section{Expansion of the error}
In this section we will prove the estimates required in Section 2, the most important two parts are a point-wise estimate of $\rho_0$ and a global $L^2$ estimate of $\nabla\rho_0$. We divide the proof of them into two subsections.

\subsection{Existence of the first approximation.}
Let us consider the equation \eqref{rho-1} of $\rho$. The linearized operator is $\Delta+p\sigma^{p-1}$, $I_1+I_2+f$ is the non-homogeneous term and $\rho$ is the solution. $I_1$ is the main data which encodes the interaction of bubbles. $I_2$ is a higher order term in $\rho$ and negligible. Since the linearized operator have kernels, $f$ should be interpreted as some Lagrange multiplier. Therefore, an approximation of $\rho$ can be obtained from studying the linear equation $\Delta\rho_0+p\sigma^{p-1}\rho_0=I_1+I_2$.

For $i=1,\cdots,\nu$, define
\begin{align}\label{eq:Z}
    Z_{i}^a=\frac{1}{\lambda_i}\left.\frac{\partial U[z,\l_i]}{\partial z^a}\right|_{z=z_i},\quad Z_{i}^{n+1}=\l_i\left.\frac{\partial U[z_i,\l]}{\partial\l}\right|_{\lambda=\lambda_i},
\end{align}
here $z^a$ is $a$-th component of $z$ for $a=1,\cdots,n$. Notice $\{Z_i^a:i=1,\cdots,\nu\}$ is the kernel of  $\Delta+pU_i^{p-1} $. It is easy to verify $|Z_i^a|\lesssim U_i$ for any $i$ and $1\leq a\leq n+1$.

Consider the following linear equation
\begin{align}\label{w2.3}
\begin{cases}
    \Delta \phi+p\sigma^{p-1}\phi=h+\sum\limits_{i=1}^\nu\sum\limits_{a=1}^{n+1}c_{a}^i U_i^{p-1}Z_{i}^a,\\
    \int U_i^{p-1}\phi Z_{i}^{a}=0,\quad i=1,\cdots, \nu;\ a=1,\cdots, n+1,
\end{cases}
\end{align}
where $\sigma=\sum_{i=1}^\nu U[z_i,\l_i]$ is a family of $\delta$-interacting bubbles. We always assume $\delta$ is very small. We shall use finite-dimensional reduction to prove the solvability of $\phi$ given a reasonable $h$ in Lemma \ref{lem:prior-est}. To that end, we need to set up the norms and spaces.

Let
$I=\{1,2,\cdots,\nu\}$, throughout this paper we  denote $y_i=\lambda_i(x-z_i)$, $i\in I$, and
\begin{align}\label{define:Rij}
   R_{ij}=\max\left\{\sqrt{\lambda_i/\lambda_j}, \sqrt{\lambda_j/\lambda_i}, \sqrt{\lambda_i\lambda_j}|z_i-z_j|\right\}=\varepsilon_{ij}^{-1}\quad \text{if}\quad i\neq j\in I.
\end{align}
\begin{definition}
For any two bubbles $U_i, U_j$, if $R_{ij}=\sqrt{\l_i\l_j}|z_i-z_j|$, then we call them a \textit{bubble cluster}, otherwise call them a  \textit{bubble tower}.
\end{definition}
Let us define
\begin{align}
    R:=\frac{1}{2}\min_{i\neq j\in I} R_{ij}.
\end{align}
Now we can define $\|\cdot\|_{**}$ and $\|\cdot\|_{*}$ norms as
\begin{align}\label{def:**}
    \|h\|_{**}=\sup_{x\in\Rn} |h(x)|V^{-1}(x),\quad \|\phi\|_{*}=\sup_{x\in\Rn} |\phi(x)|W^{-1}(x)
\end{align}
with
\begin{align}\label{define-V}
    \begin{split}
        V(x)= &\
        \sum_{i=1}^\nu\left(\frac{\lambda_i^{\frac{n+2}{2}}R^{2-n}}{\langle y_i\rangle^{4}}\chi_{\{|y_i|\leq R\}}
        +\frac{\lambda_i^{\frac{n+2}{2}}R^{-4}}{\langle y_i\rangle^{n-2}}\chi_{\{|y_i|\geq R/2\}}\right),
    \end{split}
\end{align}
\begin{align}\label{define-W}
    \begin{split}
        W(x)= &\
        \sum_{i=1}^\nu\left(\frac{\lambda_i^{\frac{n-2}{2}}R^{2-n}}{\langle y_i\rangle^{2}}\chi_{\{|y_i|\leq R\}}
        +\frac{\lambda_i^{\frac{n-2}{2}}R^{-4}}{\langle y_i\rangle^{n-4}}\chi_{\{|y_i|\geq R/2\}}\right).
    \end{split}
\end{align}
\begin{remark}
The ad hoc weight $V$ is used to capture  the behavior of the error term $h=\sigma^p-\sum_{i=1}^\nu U_i^p$. Since $\Delta_x \phi(x) \sim h(x)$, it is natural to  define $W\sim \lambda^{-2}\langle y\rangle^{2}V$.
\end{remark}

\begin{proposition}\label{prop:on-weight}
There exists a small constant $\delta_0=\delta_0(n,\nu)$ and large constant $C(n,\nu)$ such that if $\delta<\delta_0$, then
\begin{align}
    \|\sigma^p-\sum_{i=1}^\nu U_i^p \|_{**}\leq C(n,\nu).
\end{align}
\end{proposition}
\begin{proof} To make the proof more transparent, we will start with two bubbles.
 Consider $U_1=U[z_1,\lambda_1]$, $U_2=U[z_2,\lambda_2]$, and $h=(U_1+U_2)^p-U_1^p-U_2^p$.
 Because of weak interaction, $U_1$ and $U_2$ must be a bubble tower or bubble cluster. Notice that $h$ is always positive.

\begin{figure}[ht]
\centering
\includegraphics[width=0.45\textwidth]{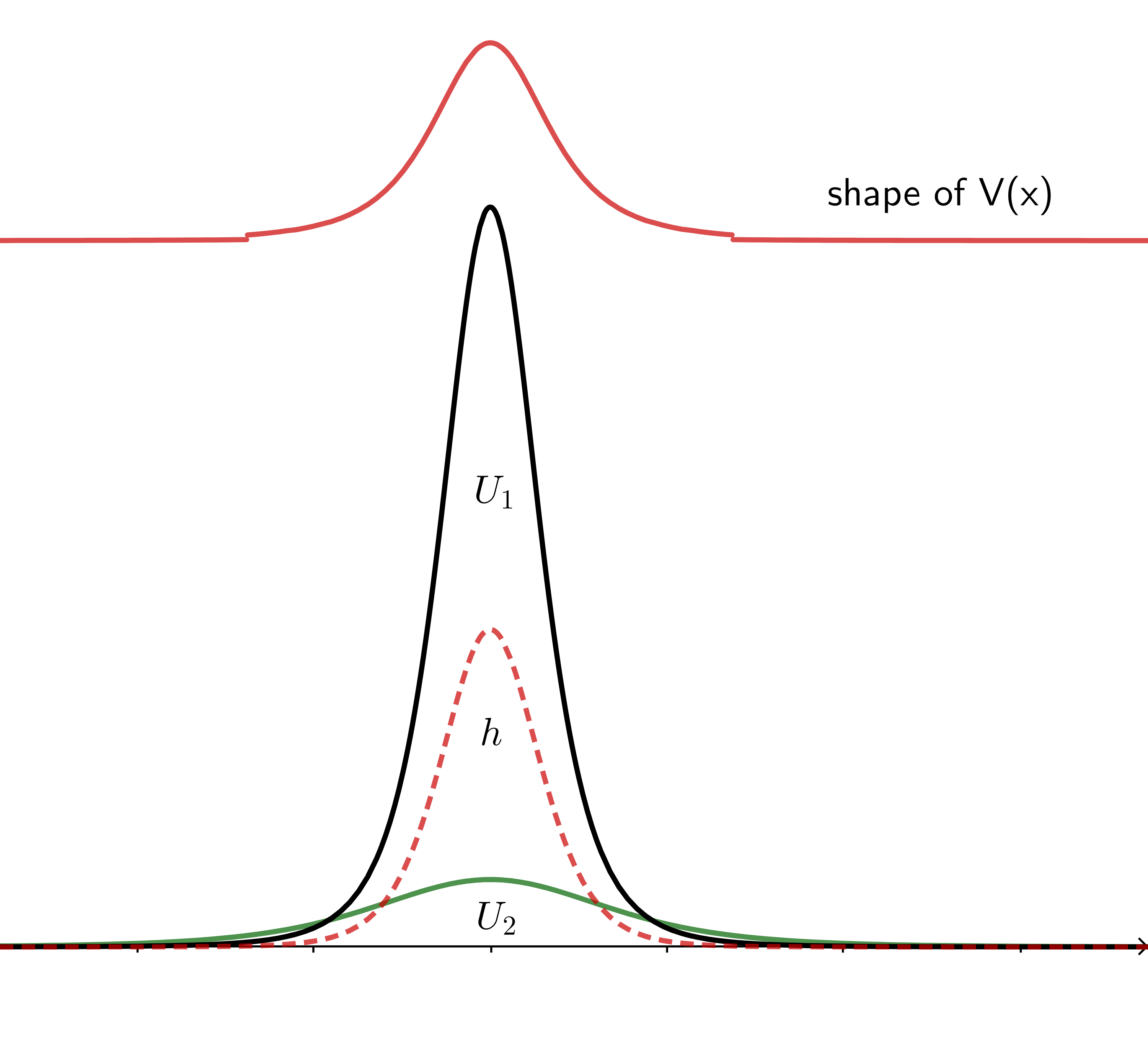}
\includegraphics[width=0.4\textwidth]{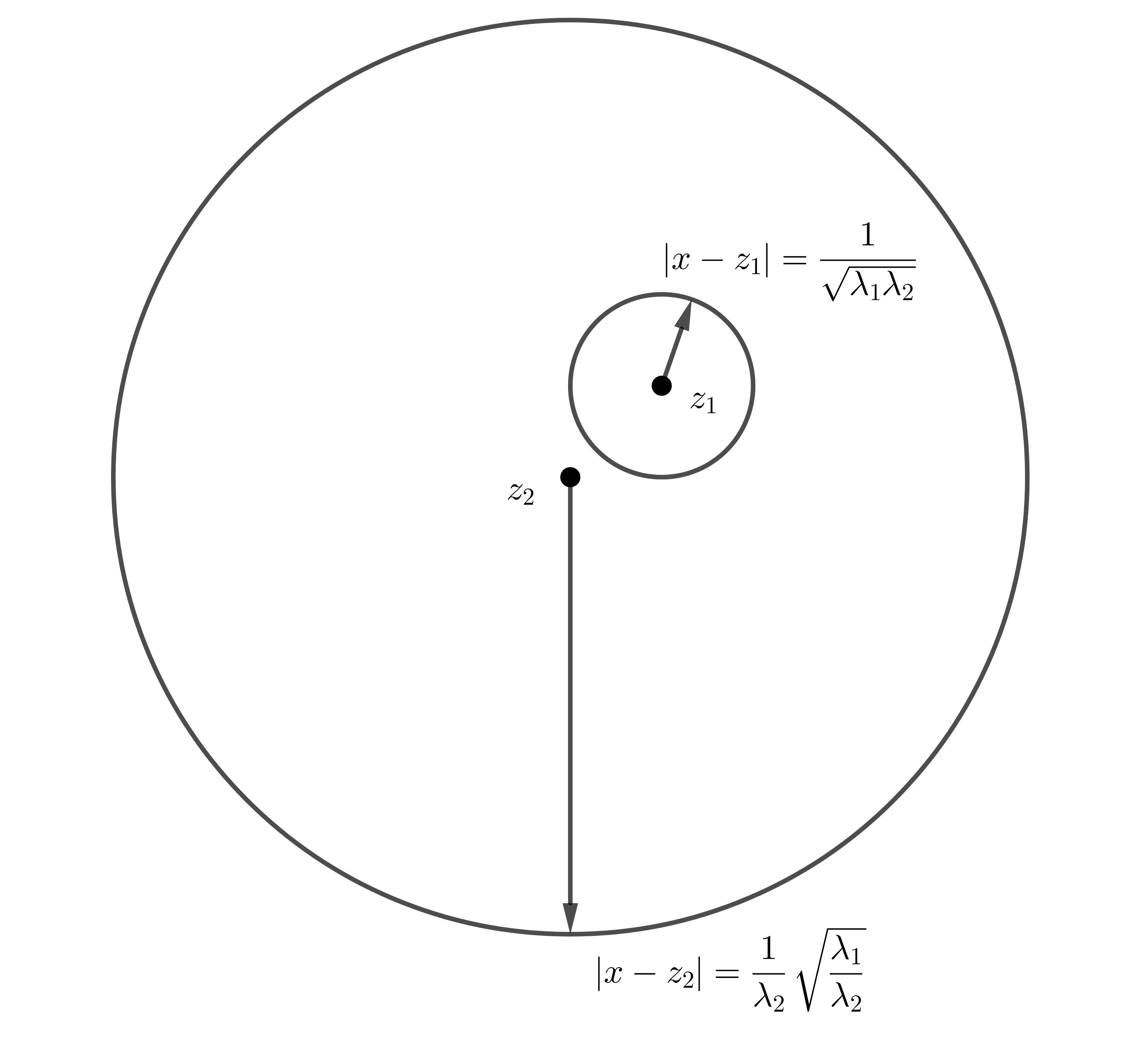}
\caption{$U_1$ and $U_2$ form a bubble tower with $\lambda_1\gg \lambda_2$. The dotted line denotes the $h$. The right picture shows that the core region of $U_2$ (i.e. $\{|y_2|\leq R_{12}\}$) contains that of $U_1$ (i.e. $\{|y_1|\leq R_{12}\}$).}
\label{fig:bubble-tower}
\end{figure}
\item[$\bullet$] Bubble tower (see Figure \ref{fig:bubble-tower}):  Without loss of generality (WLOG), we can assume $\lambda_1> \lambda_2$ and $R_{12}=\sqrt{\lambda_1/\lambda_2}=\varepsilon_{12}^{-1}\gg1$. In the rescaled  $z_1$-centered coordinate $y_1=\lambda_1(x-z_1)$, we see that $U_1(x)=\lambda_1^{(n-2)/2}U(y_1)$, where $U(y)=U[0,1](y)$, and
\begin{align} \label{t:U2atU1}
U_2(x)=\frac{\lambda_1^{(n-2)/2}R_{12}^{2-n}}{(1+\varepsilon_{12}^4|y_1-\xi_2|^2)^{(n-2)/2}},
 \end{align}
with $|\xi_2|=|\lambda_1(z_2-z_1)|\leq R_{12}^2$. If $|y_1|\leq R_{12}/2$, then $U_2\lesssim U_1$,
\begin{align}\label{h-t-1}
h\lesssim U_1^{p-1}U_2\thickapprox \frac{\lambda_1^{(n+2)/2}R_{12}^{2-n}}{\langle y_1\rangle^4}.
\end{align}
If $R_{12}/3\leq |y_1|\leq 2R_{12}^2$, then $U_1\thickapprox \lambda_1^{(n-2)/2}|y_1|^{2-n}$ and $U_2\thickapprox \lambda_1^{(n-2)/2}R_{12}^{2-n}$.  Setting $\hat{y}=y_1/R_{12}$, then
\begin{align}\label{h-t-2}
\begin{split}
    h\lesssim&\
    \lambda_1^{(n+2)/2}R_{12}^{-(n+2)}\left|\left(1+\frac{1}{|\hat{y}|^{n-2}}\right)^p-1-\frac{1}{|\hat{y}|^{p(n-2)}}\right|\\
    \lesssim&\ \frac{\lambda_1^{(n+2)/2}R_{12}^{-(n+2)}}{|\hat{y}|^{n-2}}\thickapprox\frac{\lambda_1^{(n+2)/2}R_{12}^{-4}}{\langle y_1\rangle^{n-2}}.
\end{split}
\end{align}
On the other hand, in the rescaling $z_2$-centered coordinate $y_2=\lambda_2(x-z_2)$, we see that $U_2(x)=\lambda_2^{(n-2)/2}U(y_2)$ and
\begin{align}
   U_1(x)\label{t:U1atU2}
    =\frac{\lambda_2^{(n-2)/2}R_{12}^{2-n}}{(\varepsilon_{12}^4+|y_2-\xi_1|^2)^{(n-2)/2}},
\end{align}
where $\xi_1=\lambda_2(z_1-z_2)$ with $|\xi_1|<1$. Keep in mind that $y_2-\xi_1=\varepsilon_{12}^2y_1$. Thus, if $1\leq |y_2-\xi_1|\leq R_{12}/2$, it means $R_{12}^2\leq |y_1|
\leq R_{12}^3/2$. In this region, we have $U_1\lesssim U_2$,
\begin{align}\label{h-t-3}
    h\lesssim U_2^{p-1}U_1\thickapprox\frac{\lambda_2^{(n+2)/2}R_{12}^{2-n}}{\langle y_2\rangle^4}.
\end{align}
In the outer region $|y_2-\xi_{2}|\geq R_{12}/3$, it is easy to see that
\begin{align}\label{h-t-4}
     h\lesssim U_2^p\thickapprox
\frac{\lambda_1^{(n+2)/2}R_{12}^{-2}}{\langle y_1\rangle^n}\lesssim \frac{\lambda_1^{(n+2)/2}R_{12}^{-4}}{\langle y_1\rangle^{n-2}}.
\end{align}
From \eqref{h-t-1}, \eqref{h-t-2}, \eqref{h-t-3} and \eqref{h-t-4}, we conclude
\begin{align}\label{h-t-all}
    \begin{split}
        h\lesssim &\
        \sum_{i=1}^2\left(\frac{\lambda_i^{\frac{n+2}{2}}R_{12}^{2-n}}{\langle y_i\rangle^{4}}\chi_{\{|y_i|\leq R_{12}/2\}}
        +\frac{\lambda_i^{\frac{n+2}{2}}R_{12}^{-4}}{\langle y_i\rangle^{n-2}}\chi_{\{|y_i|\geq R_{12}/3\}}\right)\\
        \leq&\ \sum_{i=1}^2\left(\frac{\lambda_i^{\frac{n+2}{2}}R^{2-n}}{\langle y_i\rangle^{4}}\chi_{\{|y_i|\leq R/2\}}
        +\frac{\lambda_i^{\frac{n+2}{2}}R^{-4}}{\langle y_i\rangle^{n-2}}\chi_{\{|y_i|\geq R/3\}}\right),
    \end{split}
\end{align}
for any $0<R\leq R_{12}$.
\begin{figure}[ht]
\centering
\includegraphics[width=0.45\textwidth]{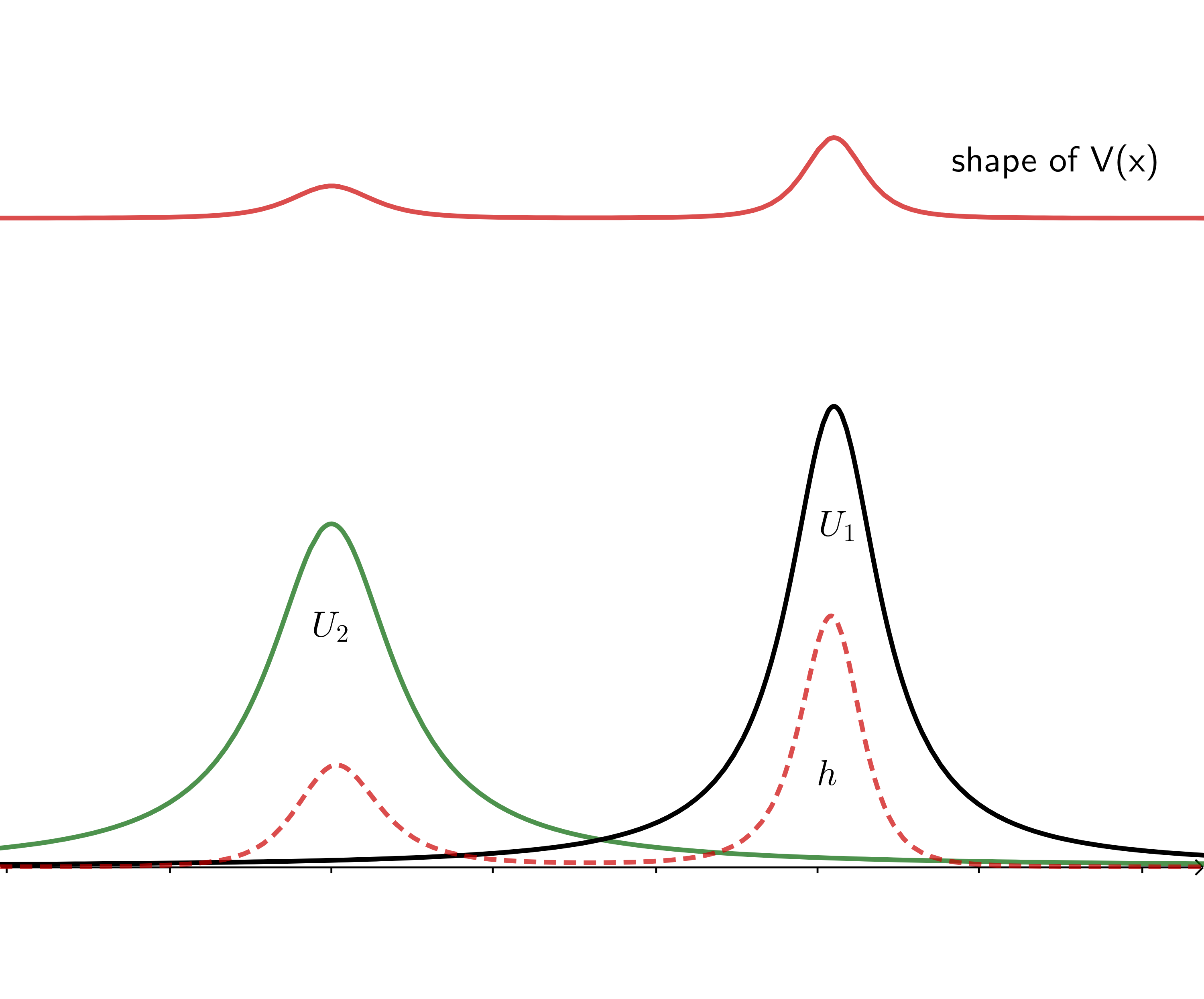}
\includegraphics[width=0.4\textwidth]{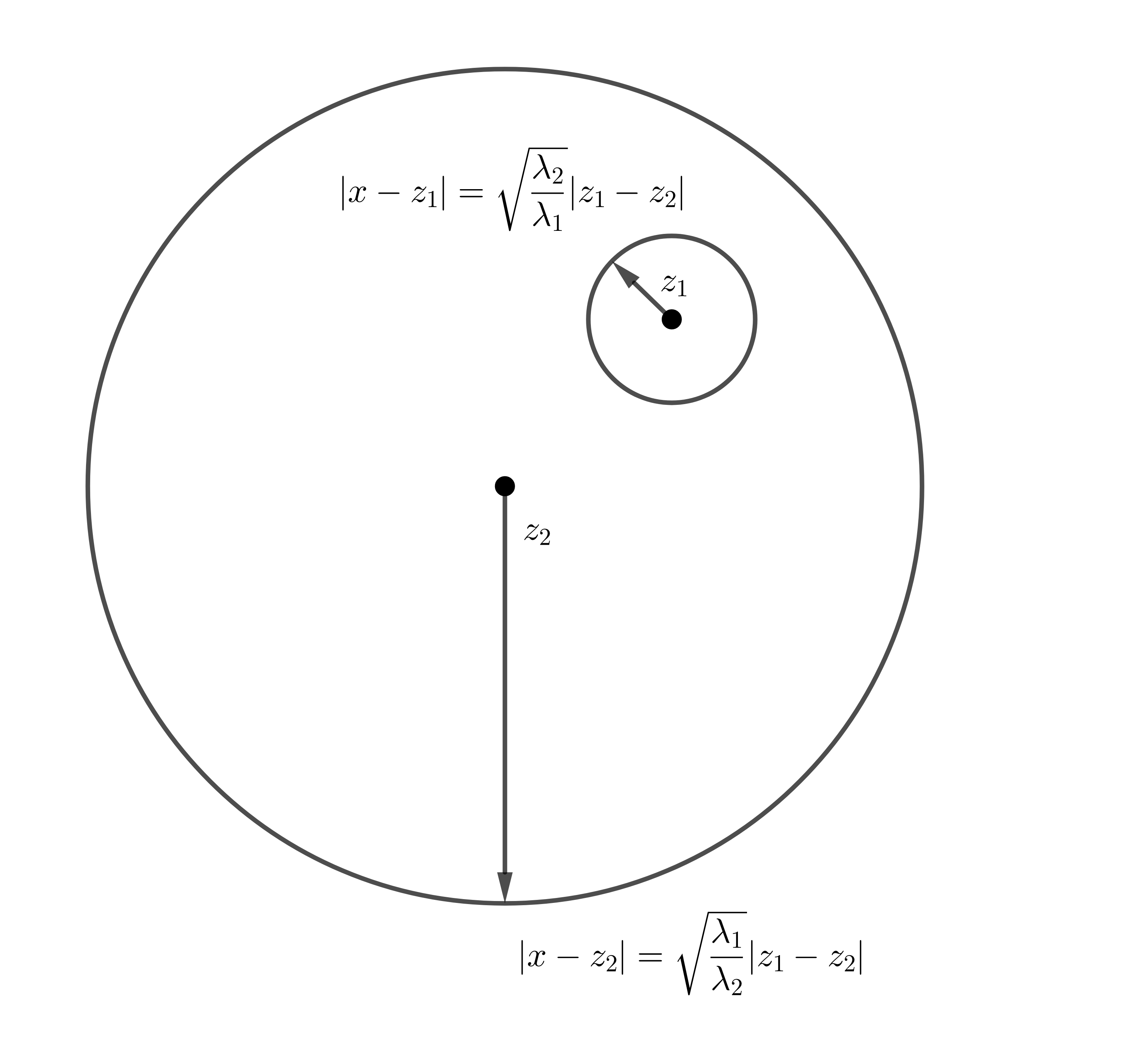}
\caption{$U_1$ and $U_2$ form a bubble cluster. The right picture shows that the influence region of $U_2$ contains that of $U_1$ like bubble tower when $\lambda_1\gg \lambda_2$. However, if $\l_1\approx\l_2$, the core region of them shall be disjoint.}
\label{fig:bubble-cluster}
\end{figure}
\item[$\bullet$] Bubble cluster (see Figure \ref{fig:bubble-cluster}): WLOG, we can assume $\lambda_1\geq\lambda_2$ and $R_{12}=\sqrt{\lambda_1\lambda_2}|z_1-z_2|=\varepsilon_{12}^{-1}\gg1$. In the rescaling $z_1$-centered coordinate $y_1=\lambda_1(x-z_1)$, $U_1(x)=\lambda_1^{(n-2)/2}U(y_1)$ and
\begin{align}\label{c:U2atU1}
  U_2(x)=\frac{\lambda_1^{(n-2)/2}R_{12}^{2-n}}{(\frac{1}{\lambda_2^2|z_1-z_2|^2}+|\frac{y_1}{\lambda_1|z_1-z_2|}-e|^2)^{(n-2)/2}},
\end{align}
with $e=(z_2-z_1)/|z_2-z_1|$. Since $R_{12}\leq \lambda_1|z_1-z_2|$, if  $|y_1|\leq R_{12}/2$,  we have $U_2\lesssim U_1$,
\begin{align}\label{h-c-1}
    h\lesssim U_1^{p-1}U_2\lesssim \frac{\lambda_{1}^{(n+2)/2}R_{12}^{2-n}}{\langle y_1\rangle^4}.
\end{align}
If $ R_{12}/3\leq |y_1|\leq \lambda_1|z_1-z_2|/2$, then $U_1\thickapprox \lambda_1^{(n-2)/2}|y_1|^{2-n}$ and $U_2\thickapprox \lambda_1^{(n-2)/2}R_{12}^{2-n}$. Setting $\hat{y}=y_1/R_{12}$, then
\begin{align}\label{h-c-2}
\begin{split}
    h\lesssim&\
    \lambda_1^{(n+2)/2}R_{12}^{-(n+2)}\left|\left(1+\frac{1}{|\hat{y}|^{n-2}}\right)^p-1-\frac{1}{|\hat{y}|^{p(n-2)}}\right|\\
    \lesssim&\ \frac{\lambda_1^{(n+2)/2}R_{12}^{-(n+2)}}{|\hat{y}|^{n-2}}\thickapprox\frac{\lambda_1^{(n+2)/2}R_{12}^{-4}}{\langle y_1\rangle^{n-2}}.
\end{split}
\end{align}
In the rescaling $z_2$-centered coordinate $y_2=\lambda_2(x-z_2)$, $U_2(x)=\lambda_2^{(n-2)/2}U(y_2)$ and
\begin{align}\label{c:U1atU2}
  U_1(x)=\frac{\lambda_2^{(n-2)/2}R_{12}^{2-n}}{(\frac{1}{\lambda_1^2|z_1-z_2|^2}+|\frac{y_2}{\lambda_2|z_1-z_2|}+e|^2)^{(n-2)/2}}.
\end{align}
 If $|y_2|\leq R_{12}/2$ and $|y_1|\geq \lambda_1|z_1-z_2|/3$, it means $|\frac{y_2}{\lambda_2|z_1-z_2|}+e|\geq1/3$, then
\begin{align}\label{h-c-3}
    h\lesssim U_2^{p-1}U_1\lesssim \frac{\lambda_{2}^{(n+2)/2}R_{12}^{2-n}}{\langle y_2\rangle^4}.
\end{align}
In the outer region $|y_2|\geq R_{12}/3$ and $|y_1|\geq \lambda_1|z_1-z_2|/3\geq R_{12}/3$, we have
\begin{align}\label{h-c-4}
    h\lesssim U_1^p+U_2^p\lesssim
\frac{\lambda_1^{(n+2)/2}R_{12}^{-4}}{\langle y_1\rangle^{n-2}}+\frac{\lambda_2^{(n+2)/2}R_{12}^{-4}}{\langle y_2\rangle^{n-2}}.
\end{align}
From \eqref{h-c-1}, \eqref{h-c-2}, \eqref{h-c-3} and \eqref{h-c-4}, we also have \eqref{h-t-all}.

\item[$\bullet $]For any finite number of bubbles, we shall use a simple inequality (see Lemma \ref{lem:aip})
\[h=\sigma^p-\sum_{i=1}^\nu U_i^p\leq \sum_{i\neq j} [(U_i+U_j)^p-U_i^p-U_j^p].\]
Each one on the RHS can be bounded by the above estimates of two bubbles, see \eqref{h-t-all}. Summing them up, one can obtain $h\leq C(n,\nu)V(x)$.
\end{proof}
\begin{remark}\label{rmk:on-weight} In order to have a simple form of $V$, we bound $h$ just by $\langle y_i\rangle^{2-n}$ in \eqref{h-t-4} and \eqref{h-c-4}. In fact, $h$ decays faster than $V$ at infinity. Such relaxation causes a minor problem for $n=6$ when estimating $\int_{|y_i|\geq R, |y_j|\geq R}VW$ in Proposition \ref{prop:L2-rho0}. Check the estimate \eqref{s-c-2} and \eqref{s-c-6} in Lemma \ref{lem:s-U-rho}. Thanks to the fact that if $n=6$ then $p=2$ and  $\sigma^p-\sum_{i=1}^\nu U_i^p=\sum_{i\neq j}U_iU_j$. We can get around this and directly estimate the integral $\int U_iU_j\rho_0$. Check the estimate \eqref{6-dim-1} in Lemma \ref{lem:6-dim}.
\end{remark}

\begin{lemma}\label{lem:prior-est}
There exist a positive $\delta_0$ and a constant $C$, independent of $\delta$, such that for all $\delta\leqslant\delta_0$, if $\{U_i\}_{1\leq i\leq \nu}$ is a $\delta$-interacting bubble family and $\phi$ solves the equation
\begin{align}\label{w2.4}
\begin{cases}
    \Delta \phi+p\sigma^{p-1}\phi=h,\\
    \int U_i^{p-1}\phi Z_{i}^{a}=0,\quad i=1,\cdots, \nu;\ a=1,\cdots, n+1,
\end{cases}
\end{align}
then
\begin{align}\label{prior-est}
    \|\phi\|_{*}\leq C\|h\|_{**}.
\end{align}
\end{lemma}

\begin{proof}\label{blow-up}
We use blow-up arguments to prove  \eqref{prior-est}. Suppose there are $k\to\infty$, $\frac{1}{k}$-interacting bubble families $\left\{U^{(k)}_i:=U[z^{(k)}_i,\l^{(k)}_i]\right\}_{1\leq i\leq\nu}$, $h=h_k$ with $\|h_k\|_{**}\to0$, and $\phi=\phi_k$ with $\|\phi_k\|_{*}=1$ solving the equation
\begin{equation}\label{eq-phi-k}
\begin{cases}
\Delta \phi_k +p\sigma_k^{p-1}\phi_k=h_k,\quad&\text{in}\ \mathbb{R}^n,\\
\int U_i^{p-1}Z_i^a\phi_k=0,\quad &i=1,\cdots, \nu;\, 1\leq a\leq n+1.
\end{cases}
\end{equation}
For simplicity, here and after we denote $U_i:=U^{(k)}_i$ and $\sigma_k:=\sum_{i=1}^{\nu}U^{(k)}_i$.

In the beginning, let us introduce some notations. After finite steps of choosing subsequences,  we can make the following assumptions. For each $i\in I=\{1,\cdots,\nu\}$, let $z^{(k)}_{ij}:=\lambda^{(k)}_i(z^{(k)}_j-z^{(k)}_i)$, $j\in I\setminus\{i\}$, we can assume that $\lim_{k\to\infty}z^{(k)}_{ij}$ exists or $\lim_{k\to\infty}z^{(k)}_{ij}=\infty$. Then we divide indices $I\setminus\{i\}$ into two groups:
\begin{align}\label{i-2index}
    \begin{split}
        I_{i,1}&:=\{j\in I\setminus\{i\}\ |\  \lim_{k\to\infty}z^{(k)}_{ij}\ \text{exists and}\  \lim_{k\to\infty}|z^{(k)}_{ij}|<\infty\},\\
        I_{i,2}&:=\{j\in I\setminus\{i\}\ |\  \lim_{k\to\infty}z^{(k)}_{ij}=\infty\}.
    \end{split}
\end{align}
Furthermore, we divide all bubbles into four groups: ones are higher (lower) than $U_i$ in tower relationship, ones are much higher (otherwise) than $U_i$ in cluster relationship.
\begin{align}\label{i-4index}
    \begin{split}
        T^{+}_{i}&:=\{j\in I\setminus\{i\}\ |\  \lambda^{(k)}_j\geq \lambda^{(k)}_j, R^{(k)}_{ij}=\sqrt{\lambda^{(k)}_j/\lambda^{(k)}_i}\},\\
        T^{-}_{i}&:=\{j\in I\setminus\{i\}\ |\  \lambda^{(k)}_j< \lambda^{(k)}_i, R^{(k)}_{ij}=\sqrt{\lambda^{(k)}_i/\lambda^{(k)}_j}\},\\
         C^{+}_{i}&:=\{j\in I\setminus\{i\}\ |\  \lim_{k\to\infty}\lambda^{(k)}_j/ \lambda^{(k)}_i=\infty, R^{(k)}_{ij}=\sqrt{\lambda^{(k)}_i\lambda^{(k)}_j}|z^{(k)}_i-z^{(k)}_j|\},\\
          C^{-}_{i}&:=\{j\in I\setminus\{i\}\ |\  \lim_{k\to\infty}\lambda^{(k)}_j/ \lambda^{(k)}_i<\infty, R^{(k)}_{ij}=\sqrt{\lambda^{(k)}_i\lambda^{(k)}_j}|z^{(k)}_i-z^{(k)}_j|\}.
    \end{split}
\end{align}
Recall that  $y^{(k)}_i=\lambda^{(k)}_i(x-z^{(k)}_i)$. Let us define
\begin{align}\label{Omega-Omegai}
    \begin{split}
        \Omega^{(k)}&:=\bigcup_{i\in I}\{|y^{(k)}_i|\leq L\},\\
        \Omega^{(k)}_i&:=\{|y^{(k)}_i|\leq L\}\bigcap\left(\bigcap_{j\in I, j\neq i}\{|y^{(k)}_i-z^{(k)}_{ij}|\geq \epsilon\}\right),
     \end{split}\end{align}
 with a large constant $L=L(n,\nu)$ and a small constant $\epsilon=\epsilon(n,\nu)$ to be determined later (see Figure \ref{fig:blow-up} for an illustration of $\Omega_1$).
It is easy to see that, for large $k$, we have
\begin{align}\label{Omegai-k}
    \begin{split}
        \Omega^{(k)}_i&:=\{|y^{(k)}_i|\leq L\}\bigcap\left(\bigcap_{j\in T^{+}_{i}\cup C^{+}_{i}}\{|y^{(k)}_i-z^{(k)}_{ij}|\geq \epsilon\}\right).
    \end{split}
\end{align}

For simplicity, here and after we drop the superscript if there is no confusion.
It is convenient to  denote $W=\sum_{i\in I}(w_{i,1}+w_{i,2})$  and $V=\sum_{i\in I}(v_{i,1}+v_{i,2})$ with
\begin{align}\label{wi-vi}
    \begin{split}
        w_{i,1}(x)=\frac{\lambda_i^{\frac{n-2}{2}}R^{2-n}}{\langle y_i\rangle^2}\chi_{\{|y_i|\leq R\}},\quad
        w_{i,2}(x)=\frac{\lambda_i^{\frac{n-2}{2}}R^{-4}}{\langle y_i\rangle^{n-4}}\chi_{\{|y_i|\geq R/2\}},\\
        v_{i,1}(x)=\frac{\lambda_i^{\frac{n+2}{2}}R^{2-n}}{\langle y_i\rangle^4}\chi_{\{|y_i|\leq R\}},\quad
        v_{i,2}(x)=\frac{\lambda_i^{\frac{n+2}{2}}R^{-4}}{\langle y_i\rangle^{n-2}}\chi_{\{|y_i|\geq R/2\}},
    \end{split}
\end{align}
where $R=R^{(k)}=\frac{1}{2}\min_{i\neq j}\{R^{(k)}_{ij}\}\to\infty$ as $k\to\infty$.

Since $\|\phi_k\|_*=1$, we have $|\phi_k|(x)\leq W(x)$ and there exists a sequence of points $\{x_k\}_{k\in\mathbb{N}}$ such that
\begin{align}\label{phik-wk}
    |\phi_{k}|(x_k)=W(x_k).
\end{align}

First, going to a subsequence if necessary, we  consider the case that $\{x_k\}_{k\in\mathbb{N}}\subset \Omega$.

\emph{Case 1: $\{x_k\}_{k\in\mathbb{N}}\subset \Omega_{i_0}$ for some $i_0\in I$.} Let $i_0=1$ and  define
\begin{align}\label{phik-hk-sigmak-1}
\begin{cases}
    \tilde{\phi}_k(y_1):=W^{-1}(x_k)\phi_k(x)\quad \text{with}\ y_1=\lambda_1(x-z_1),\\
    \tilde{h}_k(y_1):= \lambda_1^{-2}W^{-1}(x_k)h_k(x),\quad
    \tilde{\sigma}_k(y_1):=\sigma_k(x).
\end{cases}
\end{align}
Then $\tilde{\phi}_k$ satisfy
\begin{align}\label{eq-phik-tilde}
\begin{cases}
 \Delta \tilde{\phi}_k(y_1) +p\lambda_1^{-2}\tilde{\sigma}_k^{p-1}(y_1)\tilde{\phi}_k(y_1)=\tilde{h}_k(y_1)\quad &\text{in}\ \mathbb{R}^n,\\
 \int U^{p-1}Z^a\tilde{\phi}_kdy_1=0,\quad &1\leq a\leq n+1.
\end{cases}
\end{align}
Here $U=U[0,1](y_1)$ and $Z^a=Z^a(y_1)=Z_1^a(x)$  defined in \eqref{eq:Z}.  Let $\tilde{z}_j=z^{(k)}_{1j}$ , $\bar{z}_j=\lim_{k\to\infty}\tilde{z}^{(k)}_{1j}$ and
\begin{align*}
 E_1:= \bigcap_{j \in T^{+}_{1}\cap I_{1,1}}\{|y_1-\bar{z}_{j}|\geq 1/M\},\quad
      E_2:= \bigcap_{j \in C^{+}_{1}\cap I_{1,1}}\{|y_1-\bar{z}_{j}|\geq |\bar{z}_j|/M\}.
\end{align*}
Let
\begin{align}\label{K-D}
    \begin{split}
        K_M:=&\ \{|y_1|\leq M\}\cap E_1\cap E_2.
    \end{split}
\end{align}
Suppose $M\geq2\max\{L, \epsilon^{-1}\}$, it is easy to see that  $\Omega^{(k)}_i\subset\subset K_M$ for $k$ large.

\begin{claim}\label{clm:blowup}
In each compact subset $K_M$, it holds that, as $k\to\infty$,
\begin{align}
    \begin{split}
     \lambda_1^{-2}\tilde{\sigma}_k^{p-1}\to\  U[0,1],\quad |\tilde{h}_k| \to\  0,
       \end{split}\end{align}
uniformly. Moreover, we have
\begin{align}
    |\tilde{\phi}_k|(y_1)\lesssim |y_1-\tilde{z}_j|^{4-n},\quad j\in T^{+}_{1}\cup C^{+}_{1}.
\end{align}
\end{claim}
We postpone the proof of Claim \ref{clm:blowup} and finish the blow-up argument in Case 1.
By the standard elliptic regularity theorem, the Claim \ref{clm:blowup} shows that there exists a subsequence  of $\tilde{\phi}_k$ uniformly converges in each $K_M$. Furthermore, by the diagonal arguments, let $M\to \infty$, we have a subsequence of $\tilde{\phi}_k$ weakly converges to $\tilde{\phi}$ with
\begin{align}\label{eq-phi-tilde}
\begin{cases}
 \Delta \tilde{\phi} +pU^{p-1}\tilde{\phi}=0,\quad &\text{in}\ \mathbb{R}^n\setminus \{\bar{z}_j \ |\ j\in (T^{+}_{1}\cup C^{+}_{1})\cap I_{1,1}\},\\
 |\tilde{\phi}|(y_1)\lesssim |y_1-\bar{z}_j|^{4-n},\quad &j\in (T^{+}_{1}\cup C^{+}_{1})\cap I_{1,1},\\
 \int U^{p-1}Z^a\tilde{\phi}dy_1=0,\quad &1\leq a\leq n+1.
\end{cases}
\end{align}
Notice that all singular $\bar{z}_j$ are removable. Together with the non-degeneracy of Talenti bubbles, we get $\tilde{\phi}\equiv0$. However, since $|Y_k|\leq L$, going to a subsequence if necessary, then $\lim_{k\to\infty}Y_k=Y_{\infty}$ and consequently $\tilde{\phi}(Y_{\infty})=1$. This is a contradiction.

\begin{figure}[ht]
\centering
\includegraphics[width=0.6\textwidth]{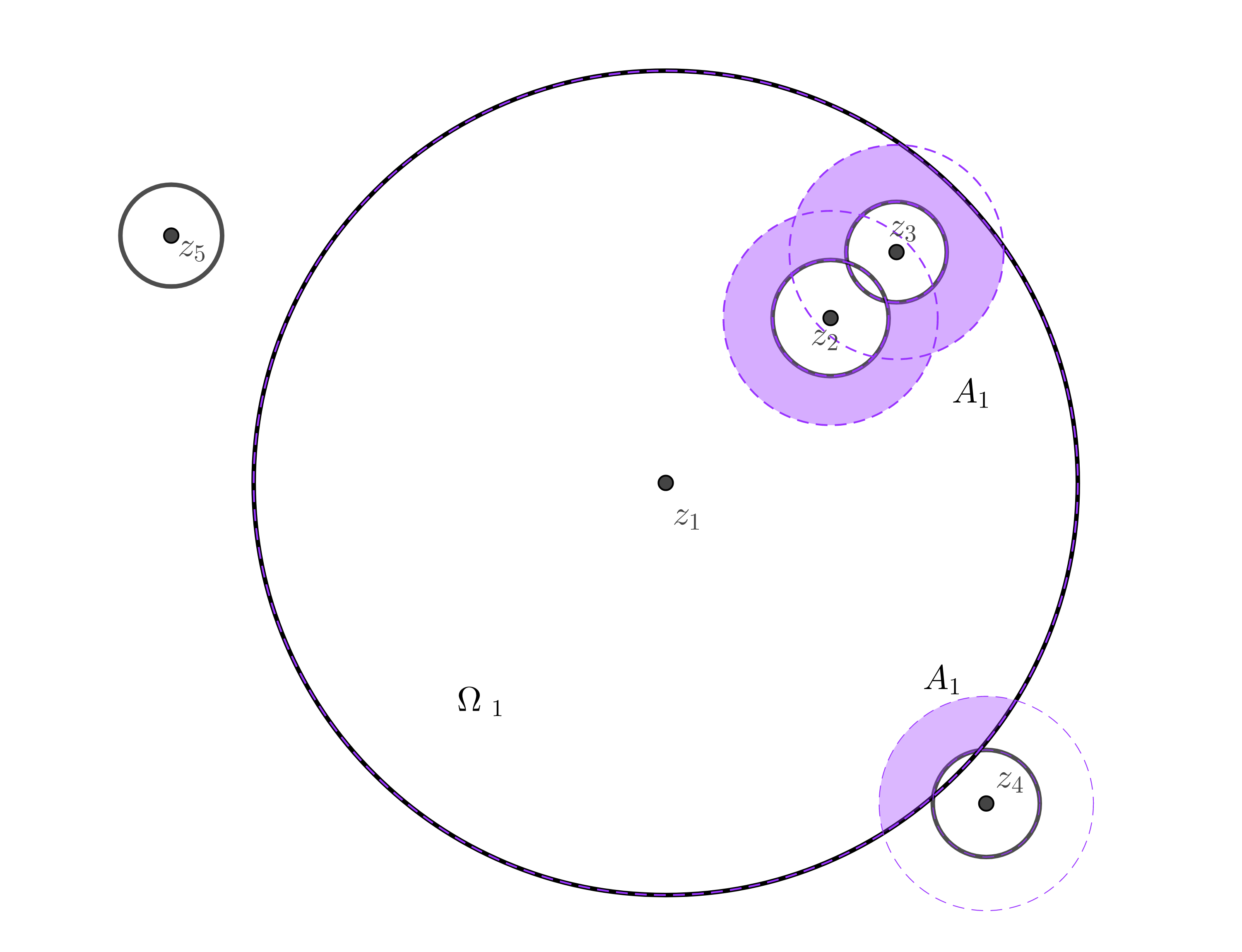}
\caption{Illustration for the blow-up regions of a simple bubble configuration. The solid circles denote $\{y_i=L\}$ for $i=1,\cdots,5$. The dashed circles mean $\{|y_1-\tilde z_{1j}|=\epsilon\}$. The shaded regions constitute $A_1$. }
\label{fig:blow-up}
\end{figure}

  \emph{Case 2: $\{x_k\}_{k\in\mathbb{N}}\subset \Omega\setminus\left(\cup_{i\in I}\Omega_i\right)$.} There exist some  $i\in I$ such that, choosing a subsequence if necessary, $\{x_k\}_{k\in\mathbb{N}}\subset A_i$ where
  \begin{align}
      A_i:=\bigcup_{j\in J_i}\left\{|y_i|\leq L,\ |y_i-\tilde{z}_{ij}|\leq \epsilon,\ |y_j|\geq L\right\},\quad J_i=T^{+}_{i}\cup C^{+}_{i}.
  \end{align}
  We can choose the \emph{smallest} $A_{i_0}$ such that, for each $A_i\supset \{x_k\}_{k\in\mathbb{N}}$, it holds that $\lambda_{i_0}\geq\lambda_i$. For simplicity, let $i_0=1$ and take notations as before. See Figure \ref{fig:blow-up} for $A_i$ in a simple case. Define
\begin{align}
    \begin{split}
        \tilde{W}(x):=&\ \sum_{j\in J_1} (w_{j,1}(x)+w_{j,2}(x))+w_{1,1}(x),\\
        \tilde{V}(x):=&\ \sum_{j\in J_1} (v_{j,1}(x)+v_{j,2}(x))+v_{1,1}(x).
    \end{split}
\end{align}
We will use the following two claims to show that $\tilde W(x)$ is a supersolution to our problem in the region $A_1$.
\begin{claim}\label{clm:out}
For any two bubbles,  say $U_i$ and $U_j$, suppose $k$ large enough, in the region $\{|y_i|\geq L, |y_j|\geq  L\}$ it holds that
\begin{align}
    \sum_{m,l\in\{i,j\}}U_m^{p-1}(w_{l,1}+w_{l,2})\leq \epsilon_1\sum_{l\in\{i,j\}}(v_{l,1}+v_{l,2}),
\end{align}
with small $\epsilon_1=\max\{L^{\frac{-4}{n-2}}, 2^nL^{-\frac12},2\epsilon^2\}$.
\end{claim}
\begin{claim}\label{clm:1-in}
Fix a bubble $U_i$, let $J_i=T^{+}_{i}\cup C^{+}_{i}$. In the region $A_i$,  we have
\begin{align}
    U_i^{p-1}\sum_{j\in J_i}(w_{j,1}+w_{j,2})\leq&\  \epsilon_1\sum_{j\in J_i}(v_{j,1}+v_{j,2}),\label{clm3-1}\\
    \sum_{j\in J_i}U_j^{p-1}w_{i,1}\leq&\ \epsilon_1\sum_{j\in J_i}(v_{j,1}+v_{j,2})+\epsilon_1v_{i,1},\label{clm3-2}
\end{align}
when $k$ large enough.
\end{claim}
We also postpone the proofs of Claim \ref{clm:out} and Claim \ref{clm:1-in}.
By \eqref{wi-vi}, it is easy to see $\Delta \tilde{W}\leq -2(n-4)\tilde{V}$.  In the region $A_1$, it is easy to see $\sum_{j\in T^-_1\cup C^-_1}U_j\ll U_1$ for $k$ large. Therefore
\[\sigma_k^{p-1}=\left(\sum_{j\in J_1}U_j+U_{1}+\sum_{j\in T^{-}_{1}\cup C^{-}_{1}}U_{j}\right)^{p-1}\leq \sum_{j\in J_1}U_{j}^{p-1}+ \frac{5}{4}U_1^{p-1}. \]
Consequently, it follows from Claim \ref{clm:out} and Claim \ref{clm:1-in} that
\begin{align}\label{barrier-1}
    \begin{split}
        \sigma_k^{p-1}\tilde{W}
        =&\ \sum_{i,j\in J_1}U_i^{p-1}\left(w_{j,1}+w_{j,2}\right)+\frac{5}{4}U_1^{p-1}\sum_{j\in J_1}\left(w_{j,1}+w_{j,2}\right)\\
       &\ + w_{1,1}\sum_{j\in J_1}U_j^{p-1}+\frac{5}{4}U_1^{p-1}w_{1,1}\\
        \leq&\  (4\nu^2\epsilon_1+\frac{5}{4})\tilde{V}\quad \text{in}\ A_1.
    \end{split}
\end{align}
Then, by choosing $L$ large and $\epsilon$ small such that $(4\nu^2\epsilon_1+5/4)p\leq 3$,
\begin{align}
    \Delta \tilde{W}+p\sigma_k^{p-1}\tilde{W}\leq -\tilde{V}\quad \text{in}\ A_1.
\end{align}
From Case 1, we conclude that, for $k$ large,
\begin{align}
    \label{interior-est-1}
    |\phi_k|(x)\leq C\|h_k\|_{**}W(x)\quad \text{in}\ \cup_{i\in I}\Omega_i,
\end{align}
with a constant $C=C(n,\nu)$. On the other hand, if $i\in T^{-}_{1}$, then $|y_i|=\frac{\lambda_i}{\lambda_1}|y_1-\tilde{z}_i|\leq L$. We have, for $k$ large,
\begin{align}\label{barrier-2}
\begin{split}
    \frac{w_{1,1}}{w_{i,1}}=&\ \left(\frac{\lambda_1}{\lambda_i}\right)^{\frac{n-2}{2}}\frac{\langle y_i\rangle^2}{\langle y_1\rangle^2}\geq L^{-2}\left(\frac{\lambda_1}{\lambda_i}\right)^{\frac{n-2}{2}}\gg1,\\
    \frac{v_{1,1}}{v_{i,1}}=&\ \left(\frac{\lambda_1}{\lambda_i}\right)^{\frac{n+2}{2}}\frac{\langle y_i\rangle^4}{\langle y_1\rangle^4}\geq L^{-4}\left(\frac{\lambda_1}{\lambda_i}\right)^{\frac{n+2}{2}}\gg1.
\end{split}
\end{align}
If $i\in C^{-}_{1}$, then $|y_i|=\frac{\lambda_i}{\lambda_1}|y_1-\tilde{z}_i|\geq \frac{\lambda_i |\tilde{z}_i|}{2\lambda_1}=\frac{\lambda_i|z_1-z_i|}{2}$. We have
\begin{align}\label{barrier-3}
\begin{split}
    \frac{w_{1,1}}{w_{i,2}}=&\
    R^{6-n} \left(\frac{\lambda_1}{\lambda_i}\right)^{\frac{n-2}{2}}\frac{\langle y_i\rangle^{n-4}}{\langle y_1\rangle^2}\geq 2^{4-n} L^{-2}R^2\left(\frac{\lambda_1}{\lambda_i}\right)\gg1,\\
    \frac{v_{1,1}}{v_{i,2}}=&\ R^{6-n}\left(\frac{\lambda_1}{\lambda_i}\right)^{\frac{n+2}{2}}\frac{\langle y_i\rangle^{n-2}}{\langle y_1\rangle^4}\geq 2^{2-n} L^{-4}R^{4}\left(\frac{\lambda_1}{\lambda_i}\right)^{2}\gg1.
\end{split}
\end{align}
From \eqref{interior-est-1}, \eqref{barrier-2},  and \eqref{barrier-3}, we have
\begin{align}\label{barrier-4}
\begin{split}
    W=&\ \tilde{W}+\sum_{j\in T^{-}_{1}\cup C^{-}_{1}}(w_{j,1}+w_{j,2})\leq 2\tilde{W}\quad \text{in}\ A_1,\\
     V=&\ \tilde{V}+\sum_{j\in T^{-}_{1}\cup C^{-}_{1}}(v_{j,1}+v_{j,2})\leq 2\tilde{V}\quad \text{in}\ A_1.
 \end{split}
\end{align}
By \eqref{barrier-1}, \eqref{interior-est-1} and \eqref{barrier-4}, it is easy to see that $\pm 10C\|h_k\|_{**}\tilde{W}$ is an upper (lower) barrier for $\phi_k$ in $A_1$. It follows that
\begin{align}
    |\phi_k|(x_k)\tilde{W}^{-1}(x_k)\leq 10C\|h_k\|_{**}\to0.
\end{align}
 It  contradicts to \eqref{phik-wk} since $W(x_k)\geq \tilde{W}(x_k)$.

\emph{Case 3: $\{x_k\}_{k\in\mathbb{N}}\subset  \Omega^c=\cup_{i\in I}\{|y_i|\geq L\}$.}
As before, it is easy to see that $\Delta W\leq -2(n-4)V$.  From Claim \ref{clm:out},   we also have
\begin{align}
    \begin{split}
        \sigma_k^{p-1}W\leq&\    \nu^2\epsilon_1V\quad \text{in}\ \Omega^c.
    \end{split}
\end{align}
Then,
\begin{align}
    \Delta W+p\sigma_k^{p-1}W\leq -V\quad \text{in}\ \Omega^c.
\end{align}
From Case 1 and Case 2, we conclude that, for $k$ large,
\begin{align}
    \label{interior-est-2}
    |\phi_k|(x)\leq C\|h_k\|_{**}W(x)\quad \text{in}\ \Omega.
\end{align}
In particular, we have
\begin{align}
    |\phi_k|(x)\leq C\|h_k\|_{**}W(x)\quad \text{on}\ \partial \Omega^c=\partial\Omega.
\end{align}
Thus $\pm C\|h_k\|_{**}W$ is an upper (lower) barrier for $\phi_k$ in $\Omega^c$. It follows that
\begin{align}
    |\phi_k|(x_k)W^{-1}(x_k)\leq C\|h_k\|_{**}\to0.
\end{align}
 It  contradicts to \eqref{phik-wk}.

 To complete the proof of \eqref{prior-est}, it suffices to prove the three claims above.
 \begin{proof}[Proof of Claim \ref{clm:blowup}] To prove the claim, we need to use a simple inequality
\begin{align}
    \label{ineq-1}
    \frac{\sum_{i} a_i}{\sum_{i}b_i}\leq \max_i\{\frac{a_i}{b_i}\},
\end{align}
holds for any positive numbers $a_i, b_i$. Thus by \eqref{eq-phik-tilde}
\[|\tilde \phi_k|(y_1)=\frac{|\phi_k(x)|}{W(x_k)}\leq \frac{W(x)}{W_k(x)}\leq \max_{j}\left\{\frac{w_{j,1}(x)+w_{j,2}(x)}{w_{j,1}(x_k)+w_{j,2}(x_k)}\right\},\]
\[|\tilde h_k|(x)=\frac{|h_k(x)|}{\l_1^2W(x_k)}\leq \|h_k\|_{**}\frac{V(x)}{\lambda_1^2W(x_k)}\leq \|h_k\|_{**}\max_{j}\left\{\frac{v_{j,1}(x)+v_{j,2}(x)}{\l_1^2w_{j,1}(x_k)+\l_1^2w_{j,2}(x_k)}\right\}.\]

We shall estimate the RHS of the above two equations on each compact subset $K_M$. They are divided into the following four cases.


\noindent
\emph{The case $j\in T^{+}_{1}$.} By \eqref{t:U1atU2}, we have
\begin{align}\label{T11-U}
   U_j(x)\thickapprox \frac{\lambda_j^{\frac{n-2}{2}}}{\langle y_j\rangle^{n-2}}
    \thickapprox\frac{\lambda_1^{\frac{n-2}{2}}R_{1j}^{2-n}}{(\varepsilon_{1j}^2+|y_1-\tilde{z}_j|)^{n-2}}\lesssim M^{n-2}\lambda_1^{\frac{n-2}{2}}R_{1j}^{2-n}.
\end{align}
Let $Y_k=\lambda_1(x_k-z_1)$. Since $|y_j|=R_{1j}^2|y_1-\tilde{z}_j|\geq R_{1j}^2/M\geq R$ for $k $ large, we only need to consider
\begin{align}
        \frac{w_{j,2}(x)}{w_{j,2}(x_k)}\thickapprox&\ \frac{(\varepsilon_{1j}^{2}+|Y_k-\tilde{z}_j|)^{n-4}}{(\varepsilon_{1j}^{2}+|y_1-\tilde{z}_j|)^{n-4}}\lesssim \frac{1}{|y_1-\tilde{z}_j|^{n-4}},
        \label{T11-phi}\\
        \frac{v_{j,2}(x)}{\lambda_1^2w_{j,2}(x_k)}\thickapprox&\
        \frac{(\varepsilon_{1j}^{2}+|Y_k-\tilde{z}_j|)^{n-4}}{(\varepsilon_{1j}^{2}+|y_1-\tilde{z}_j|)^{n-2}}\lesssim M^{n-2}.\label{T11-h}
\end{align}
\emph{The case $j\in T^{-}_{1}$.} By \eqref{t:U2atU1}, we have
\begin{align}\label{T12-U}
   U_j(x) \thickapprox\frac{\lambda_1^{\frac{n-2}{2}}R_{1j}^{2-n}}{(1+\varepsilon_{1j}^2|y_1-\tilde{z}_j|)^{n-2}}\lesssim \lambda_1^{\frac{n-2}{2}}R_{1j}^{2-n}.
\end{align}
It holds that $|y_j|=R_{1j}^{-2}|y_1-\tilde{z}_j|\leq MR_{1j}^{-2}\leq R$ for $k$ large. Then we have
\begin{align}
        \frac{w_{j,1}(x)}{w_{j,1}(x_k)}\thickapprox&\ \frac{(1+\varepsilon_{1j}^{2}|Y_k-\tilde{z}_j|)^{2}}{(1+\varepsilon_{1j}^{2}|y_1-\tilde{z}_j|)^{2}}\lesssim
        1, \label{T12-phi}\\
       \frac{v_{j,1}(x)}{\lambda_1^2w_{j,1}(x_k)}\thickapprox&\
         R_{1j}^{-4} \frac{(1+\varepsilon_{1j}^{2}|Y_k-\tilde{z}_j|)^{2}}{(1+\varepsilon_{1j}^{2}|y_1-\tilde{z}_j|)^{4}}\lesssim R_{1j}^{-4}.\label{T12-h}
\end{align}
\emph{The case $j\in C^{+}_{1}$.} Let $\lambda=\lambda_j/\lambda_1$, then $R_{1j}=\sqrt{\lambda}|\tilde{z}_j|$ and $|\tilde{z}_j|\geq1$. It holds that, no matter $|\bar{z}_j|<\infty$ or $|\bar{z}_j|=\infty$, we have $|y_j|=\lambda |y_1-\tilde{z}_j|\geq \sqrt{\lambda}R_{1j}/M\geq R$ for $k$ large. Then
\begin{align}\label{C11-U}
   U_j(x) \thickapprox\frac{\lambda_1^{\frac{n-2}{2}}}{(\lambda^{-1/2}+\lambda^{1/2}|y_1-\tilde{z}_j|)^{n-2}}
   \lesssim \lambda_1^{\frac{n-2}{2} }M^{n-2}R_{1j}^{2-n}.
\end{align}
Then we have
\begin{align}
        \frac{w_{j,2}(x)}{w_{j,2}(x_k)}\thickapprox&\ \frac{(\lambda^{-1/2}+\lambda^{1/2}|Y_k-\tilde{z}_j|^2)^{n-4}}{(\lambda^{-1/2}+\lambda^{1/2}|y_1-\tilde{z}_j|)^{n-4}}\lesssim
        \frac{1}{|y_1-\tilde{z}_j|^{n-4}}, \label{C11-phi}\\
       \frac{v_{j,2}(x)}{\lambda_1^2w_{j,2}(x_k)}\thickapprox&\
         \lambda R_{1j}^{-2} \frac{(\lambda^{-1/2}+\lambda^{1/2}|Y_k-\tilde{z}_j|)^{n-4}}{(\lambda^{-1/2}+\lambda^{1/2}|y_1-\tilde{z}_j|)^{n-2}}\nonumber\\
         \lesssim&\  M^{n-2}R_{1j}^{-2}|\tilde{z}_j|^{-2}\leq M^{n-2}R_{1j}^{-2}.\label{C11-h}
\end{align}
\emph{The case $j\in C^{-}_{1}$.}  Let $\lambda=\lambda_j/\lambda_1$, then $R_{1j}=\sqrt{\lambda}|\tilde{z}_j|\leq (1+\theta_j)|\tilde{z}_j|$ with $\theta_j=\lim_{k\to \infty}\lambda^{(k)}_j/\lambda^{(k)}_1$. It  holds that $|y_1|\leq  M\leq |\tilde{z}_j|/2$ for $k$ large. Then $|y_1-\tilde{z}_j|\geq |\tilde{z}_j|/2$,
\begin{align}\label{C12-U}
   U_j(x) \thickapprox\frac{\lambda_1^{\frac{n-2}{2}}}{(\lambda^{-1/2}+\lambda^{1/2}|y_1-\tilde{z}_j|)^{n-2}}
   \lesssim  \lambda_1^{\frac{n-2}{2}}R_{1j}^{2-n}.
\end{align}
 Furthermore, we have
\begin{align}
        \frac{w_{j,2}(x)}{w_{j,2}(x_k)}\thickapprox&\ \frac{(\lambda^{-1/2}+\lambda^{1/2}|Y_k-\tilde{z}_j|^2)^{n-4}}{(\lambda^{-1/2}+\lambda^{1/2}|y_1-\tilde{z}_j|)^{n-4}}\lesssim
        R_{1j}^{4-n}, \label{C12-phi}\\
       \frac{v_{j,2}(x)}{\lambda_1^2w_{j,2}(x_k)}\thickapprox&\
         \lambda R_{1j}^{-2} \frac{(\lambda^{-1/2}+\lambda^{1/2}|Y_k-\tilde{z}_j|)^{n-4}}{(\lambda^{-1/2}+\lambda^{1/2}|y_1-\tilde{z}_j|)^{n-2}}\lesssim R_{1j}^{-n}.\label{C12-h}
\end{align}
It is easy to see that
\begin{align}\label{U1-phi1-h1}
  U_1(x) = U(y_1),\quad
        \frac{w_{1,1}(x)}{w_{1,1}(x_k)}\thickapprox&\ \frac{\langle Y_k\rangle^2}{\langle y_1\rangle^2}\lesssim
        1, \quad
       \frac{v_{1,1}(x)}{\lambda_1^2w_{1,1}(x_k)}\thickapprox
         \frac{\langle Y_k\rangle^2}{\langle y_1\rangle^4}\lesssim 1.
\end{align}
From \eqref{T11-U}, \eqref{T12-U}, \eqref{C11-U}, \eqref{C12-U} and \eqref{U1-phi1-h1}, we get
\begin{align}
    \label{sigmak-lim}
    \lambda_1^{-2}\tilde{\sigma}_k^{p-1}(y_1)=U[0,1](y_1)+o(1)\to U[0,1](y_1).
\end{align}
From \eqref{T11-U}, \eqref{T12-U}, \eqref{C11-U}, \eqref{C12-U} and \eqref{U1-phi1-h1}, we get
\begin{align}\label{hk-lim}
    \begin{split}
        |\tilde{h}_k|\leq&\ \|h_k\|_{**}\frac{V(x)}{\lambda_1^2W(x_k)}
        \lesssim
       M^{n-2}\|h_k\|_{**}\to 0.
       \end{split}\end{align}
From \eqref{T11-phi}, \eqref{T12-phi}, \eqref{C11-phi}, \eqref{C12-phi} and \eqref{U1-phi1-h1}, we see that the singularities  only happen for $j\in (T^{+}_{1}\cup C^{+}_{1})\cap I_{1,1}$ with
\begin{align}\label{phik-lim}
    |\tilde{\phi}_k|(y_1)\lesssim |y_1-\tilde{z}_j|^{4-n}.
\end{align}
\end{proof}
\begin{proof}[Proof of the Claim \ref{clm:out}] Here and after, we always assume $k$ is large enough. WLOG, we assume $i=1$ and $j=2$. It is easy to see the claim holds when $m=l$,
\begin{align}
    \sum_{l=1}^2U_i^{p-1}(w_{l,1}+w_{l,2})\leq
    \epsilon_1\sum_{l=1}^2(v_{l,1}+v_{l,2}).
\end{align}
It reduces to consider the case $m\neq l$. We divide it to the following two main cases.

\noindent
(1) \emph{The case $\lambda_1\geq\lambda_2$ and $R_{12}=\sqrt{\lambda_1/\lambda_2}$.} Let $\xi_1=\lambda_2(z_1-z_2)$ and $\xi_2=\lambda_1(z_2-z_1)$. It holds that $|\xi_1|\leq1$ and $|\xi_2|\leq R_{12}^2$. Since $y_1=R_{12}^2(y_2-\xi_1)$, then $\langle y_1\rangle\geq \frac12R_{12}^2\langle y_2\rangle $ on the set $\{|y_1|\geq L,|y_2|\geq L\}$. We have
\begin{align}
\begin{split}
   U_1^{p-1}(w_{2,1}+w_{2,2})=&\ \frac{\lambda_1^2}{\langle y_1\rangle^4}(w_{2,1}+w_{2,2})
    \leq\ 16\lambda_2^2\langle y_2\rangle^{-4}(w_{2,1}+w_{2,2})\\
    \leq&\ 16\langle y_2\rangle^{-2}(v_{2,1}+v_{2,2})\leq\epsilon_1(v_{2,1}+v_{2,2}).
\end{split}
\end{align}
 Observe that $|y_1-\xi_2|=R_{12}^2|y_2|\geq LR_{12}^2$, then $|y_1|\geq (L-1)R_{12}^2$.  In the region $\{|y_1|\geq L, |y_2|\geq L\}$, we get
\begin{align}
\begin{split}
   U_2^{p-1}(w_{1,1}+w_{1,2})=&\ \frac{\lambda_2^2}{\langle y_2\rangle^4}\frac{\lambda_1^{\frac{n-2}{2}}R^{-4}}{\langle y_1\rangle^{n-4}}\chi_{\{|y_1|\geq(L-1)R_{12}^2\}}.
\end{split}
\end{align}
First, we bound the above term on the set $\{L\leq |y_2|\leq LR_{12}, |y_1|\geq (L-1)R_{12}^2\}$, that is
\begin{align}
\begin{split}
   \frac{\lambda_2^2}{\langle y_2\rangle^4}\frac{\lambda_1^{\frac{n-2}{2}}R^{-4}}{\langle y_1\rangle^{n-4}}=&\
    \frac{\lambda_2^{\frac{n+2}{2}}R^{-4}}{\langle y_2\rangle^4}\frac{R_{12}^{n-2}}{\langle y_1\rangle^{n-4}}
    \leq\ (L-1)^{4-n}\frac{\lambda_2^{\frac{n+2}{2}}R^{-4}}{\langle y_2\rangle^4}\frac{1}{R_{12}^{n-6}}\\
    \leq&\ \frac{L^{n-6}}{(L-1)^{n-4}}\left(\frac{\lambda_2^{\frac{n+2}{2}}R^{2-n}}{\langle y_2\rangle^{4}}\chi_{\{|y_2|\leq R\}}+\frac{\lambda_2^{\frac{n+2}{2}}R^{-4}}{\langle y_2\rangle^{n-2}}\chi_{\{R\leq |y_2|\leq LR_{12}\}}\right)\\
    \leq &\ \epsilon_1 (v_{2,1}+v_{2,2}).
\end{split}
\end{align}
Second, on $\{|y_2|\geq LR_{12}, |y_1|\geq (L-1)R_{12}^2\}$, we have $\langle y_1\rangle=\langle R_{12}^2(y_2-\xi_1)\rangle\geq \frac12R_{12}^2\langle y_2\rangle$,
\begin{align}
\begin{split}
   \frac{\lambda_2^2}{\langle y_2\rangle^4}\frac{\lambda_1^{\frac{n-2}{2}}R^{-4}}{\langle y_1\rangle^{n-4}}=&\
    \frac{\lambda_2^{\frac{n+2}{2}}R^{-4}}{\langle y_2\rangle^4}\frac{R_{12}^{n-2}}{\langle y_1\rangle^{n-4}}
    \leq\ v_{2,2}\frac{\langle y_2\rangle^{n-6}R_{12}^{n-2}}{\langle y_1\rangle^{n-4}}
    \leq \epsilon_1v_{2,2}.
\end{split}
\end{align}

(2) \emph{The case $\lambda_1\geq\lambda_2$ and $R_{12}=\sqrt{\lambda_1\lambda_2}|z_1-z_2|$.} Notice that $|\xi_2|=\sqrt{\l_1/\l_2}R_{12}$ and $y_2=\l_2\l_1^{-1}(y_1-\xi_2)$. Then on the set $\{L\leq |y_1|\leq \frac12 \sqrt{\l_1/\l_2}R_{12}\}$, we have $\langle y_2\rangle\geq \frac12\sqrt{\l_2/\l_1}R_{12}$. Consequently
\begin{align}\label{U2w11}
    U_2^{p-1}w_{1,1}=\frac{\l_2^2}{\langle y_2\rangle^4}
    \frac{\l_1^{\frac{n-2}{2}}R^{2-n}}{\langle y_1\rangle^2}\chi_{\{L\leq |y_1|\leq R\}}\leq 16\l_1^2R_{12}^{-4}\frac{\l_1^{\frac{n-2}{2}}R^{2-n}}{\langle y_1\rangle^2}\leq \epsilon_1 v_{1,1}.
\end{align}
Next consider $U_2^{p-1}w_{1,2}$. To bound it, we divide it into two cases. First,  suppose $\l_2|z_1-z_2|\geq {L}^{\frac14}$, then
on the set $\{R<|y_1|<\frac12\sqrt{\l_1/\l_2}R_{12}\}$
\begin{align}
    U_2^{p-1}w_{1,2}\leq 16\l_1^2R_{12}^{-4}\frac{\l_1^{\frac{n-2}{2}}R^{-4}}{\langle y_1\rangle^{n-4}}=16v_{1,2}\frac{\langle y_1\rangle^2}{R_{12}^4}<\frac{16}{\l_2^2|z_1-z_2|^2}v_{1,2}\leq \frac{16}{\sqrt{L}}v_{1,2}.
\end{align}
On the set $\{|y_1|>\frac12\sqrt{\l_1/\l_2}R_{12}\}$, then
\begin{align*}
    U_2^{p-1}w_{1,2}=\frac{\l_2^2}{\langle y_2\rangle^4}\frac{\l_1^{\frac{n-2}{2}}R^{-4}}{\langle y_1\rangle^{n-4}}=v_{2,1}\frac{(\l_1/\l_2)^{\frac{n-2}{2}}R^{n-6}}{\langle y_1\rangle^{n-4}}<\frac{2^{n-4}}{\l_2^2|z_1-z_2|^2} v_{2,1}<\frac{2^{n-4}}{\sqrt{L}}v_{2,1}.
\end{align*}
Second, suppose $|\xi_1|=\l_2|z_1-z_2|\leq L^{\frac14}$. On the set $\{L\leq |y_2|\leq R\}$, since $y_1=\l_1\l_2^{-1}(y_2-\xi_1)$, then $|y_1|\geq \frac{\l_1}{\l_2}\frac{L}{2}\geq \frac{\sqrt{L}}{2}R_{12}^2$ and
\[U_2^{p-1}w_{1,2}=\frac{\l_2^2}{\langle y_2\rangle^4}\frac{\l_1^{\frac{n-2}{2}}R^{-4}}{\langle y_1\rangle^{n-4}}\leq v_{2,1}\frac{(\l_1/\l_2)^{\frac{n-2}{2}}R^{n-6}}{(\frac12\sqrt{L}R_{12}^2)^{n-4}}\leq (\frac{1}{2}\sqrt{L})^{4-n}v_{2,1}\leq \epsilon_1v_{2,1}.\]
On the set $\{|y_2|\geq R\}$,
\begin{align*}
    U_2^{p-1}w_{1,2}=v_{2,2}\frac{(\frac{\l_1}{\l_2})^{\frac{n-2}{2}}\langle y_2\rangle^{n-6}}{\langle y_1\rangle^{n-4}}\leq v_{2,2}\frac{(\frac{\l_1}{\l_2})^{\frac{n-2}{2}}\langle y_2\rangle^{n-6}}{\langle \frac12\frac{\l_1}{\l_2}y_2\rangle^{n-4}}\leq v_{2,2}\frac{2^{n-4}(\frac{\l_1}{\l_2})^{\frac{6-n}{2}}}{\langle y_2\rangle^2}\leq \epsilon_1 v_{2,2}.
\end{align*}

Let us consider the term $U_1^{p-1}w_{2,1}$.
On the set $ \{L<|y_2|<R, |y_1|<R\}$, one has $\langle y_2\rangle\geq \frac12\sqrt{\l_2/\l_1}R_{12}$ and
\[U_1^{p-1}w_{2,1}=\frac{\l_1^2}{\langle y_1\rangle^4}\frac{\l_2^{\frac{n-2}{2}}R^{2-n}}{\langle y_2\rangle^2}\leq\frac{\l_1^2}{\langle y_1\rangle^4}\frac{\l_2^{\frac{n-2}{2}}R^{2-n}}{(\frac12\sqrt{\l_2/\l_1}R_{12})^{2}}=4(\frac{\l_2}{\l_1})^{\frac{n-4}{2}}R_{12}^{-2}v_{1,1}\leq\epsilon v_{1,1}.\]
On the set $\{L<|y_2|<R, R<|y_1|\}\cap\{L(\l_2/\l_1)^2v_{1,2}\geq v_{2,1}\}$, one has
\begin{align}\label{121-1}
    U_1^{p-1}w_{2,1}=\frac{\l_1^2}{\langle y_1\rangle^4}\frac{\l_2^{\frac{n-2}{2}}R^{2-n}}{\langle y_2\rangle^2}=v_{2,1}\frac{\langle y_2\rangle^2}{\langle y_1\rangle^4}(\frac{\l_1}{\l_2})^2\leq \frac{L\langle y_2\rangle^2}{\langle y_1\rangle^4}v_{1,2}\leq LR^{-2}v_{1,2}.
\end{align}
On the set $\{L<|y_2|<R, R<|y_1|\}\cap\{L(\l_2/\l_1)^2v_{1,2}\leq v_{2,1}\}$, one has $\langle y_1\rangle^{n-2}\geq L(\l_1/\l_2)^{\frac{n-2}{2}}R^{n-6}\langle y_2\rangle^4$, then
\begin{align}\label{121-2}
    U_1^{p-1}w_{2,1}=v_{2,1}\frac{\langle y_2\rangle^2}{\langle y_1\rangle^4}\left(\frac{\l_1}{\l_2}\right)^2\leq v_{2,1}L^{-\frac{4}{n-2}}R^{-\frac{4(n-6)}{n-2}}\langle y_2\rangle^{\frac{2n-20}{n-2}}\leq \epsilon_1 v_{2,1}
\end{align}
here we have use $ L\leq|y_2|\leq R$.


It remains to consider the term $U_1^{p-1}w_{2,2}$. First, on the set $\{|y_2|\geq R, |y_1|\leq R\}$, similar to \eqref{U2w11}, one has
\begin{align}
    U_1^{p-1}w_{2,2}=\frac{\l_1^2}{\langle y_1\rangle^4}\frac{\l_2^{\frac{n-2}{2}}R^{-4}}{\langle y_2\rangle^{n-4}}\leq \frac{\l_1^2}{\langle y_1\rangle^4}\frac{\l_2^{\frac{n-2}{2}}R^{2-n}}{\langle y_2\rangle^{2}}\leq \epsilon_1 v_{1,1}
\end{align}

Second, on the set $\{|y_2|\geq R ,|y_1|\geq R\}\cap \{v_{1,2}\geq v_{2,2}\}$, one has $\langle y_1\rangle/\langle y_2\rangle\leq (\frac{\l_1}{\l_2})^{\frac{n+2}{2(n-2)}}$,
\begin{align*}
    U_1^{p-1}w_{2,2}=\frac{\l_1^2}{\langle y_1\rangle^4}\frac{\l_2^{\frac{n-2}{2}}R^{-4}}{\langle y_2\rangle^{n-4}}\leq v_{1,2}\frac{\langle y_1\rangle^{n-6}}{\langle y_2\rangle^{n-4}}\left(\frac{\l_2}{\l_1}\right)^{\frac{n-2}{2}}\leq v_{1,2}\frac{1}{\langle y_2\rangle^2}\left(\frac{\l_2}{\l_1}\right)^{\frac{8}{n-2}}\leq \epsilon_1 v_{1,2}.
\end{align*}
On the set $\{|y_2|\geq R ,|y_1|\geq R\}\cap \{v_{1,2}\leq v_{2,2}\}$, one has $\langle y_1\rangle\geq \langle y_2\rangle (\l_1/\l_2)^{\frac{n+2}{2(n-2)}}$
\[U_1^{p-1}w_{2,2}=\frac{\l_1^2}{\langle y_1\rangle^4}\frac{\l_2^{\frac{n-2}{2}}R^{-4}}{\langle y_2\rangle^{n-4}}=v_{2,2}\left(\frac{\l_1}{\l_2}\right)^2\frac{\langle y_2\rangle^2}{\langle y_1\rangle^4}\leq v_{2,2}\frac{(\l_2/\l_1)^{\frac{8}{n-2}}}{\langle y_2\rangle^2}\leq \epsilon_1 v_{2,2}.\]

\end{proof}
\begin{proof}
[Proof of the Claim \ref{clm:1-in}]
WLOG we assume $i=1$. By the definition \eqref{i-4index}, we have $\lambda_j\gg \lambda_1$. Let $\lambda=\lambda_j/\lambda_1$, $\xi_1=\lambda_j(z_1-z_j)$ and $\xi_j=\tilde{z}_j=\lambda_1(z_j-z_1)$.
 Observe that, if $j\in C^{+}_{1}$, $|y_j|\leq R\ll \sqrt{\lambda}R_{1j}= |\xi_1|$, then $|y_j-\xi_1|\geq|\xi_1|/2$. Thus, on the set $\{
 L\leq |y_j|\leq R\}$, no matter $j\in C_1^+$ or $T_1^+$,
 \begin{align}\label{t11-c11}
     \l\langle y_1\rangle^2=\l+\l^{-1}|y_j-\xi_1|^2\geq \frac12 R_{1j}^2,\quad j\in T^{+}_1\cup C^{+}_1.
 \end{align}
It follows from \eqref{t11-c11} that, on the set $\{
 L\leq |y_j|\leq R\}$,
\begin{align}\label{claim-2-1}
    \begin{split}
        U_1^{p-1}w_{j,1}=&\ \frac{\lambda_1^2}{\langle y_1\rangle^4}\frac{\lambda_j^{\frac{n-2}{2}}R^{2-n}}{\langle y_j\rangle^{2}}
        \leq\ 16 R_{1j}^{-4}\frac{\lambda_j^{\frac{n+2}{2}}R^{2-n}}{\langle y_j\rangle^{2}}\leq \epsilon_1v_{j,1}.
    \end{split}
\end{align}
Moreover, on the set $\{|y_j|\geq R, |y_1-\xi_j|\leq \epsilon\}$, it holds that
$\langle y_j\rangle^2=1+\l^2|y_1-\xi_j|^2\leq 2\l^2\epsilon^2$ and
\begin{align}\label{claim-2-2}
    \begin{split}
        U_1^{p-1}w_{j,2}=\ \frac{\lambda_1^2}{\langle y_1\rangle^4}\frac{\lambda_j^{\frac{n-2}{2}}R^{-4}}{\langle y_j\rangle^{n-4}}=\frac{\langle y_j\rangle^2}{\l^2\langle y_1\rangle^4}v_{j,2}\leq 2\epsilon^2 v_{j,2}\leq \epsilon_1 v_{j,2}.
    \end{split}
\end{align}
Together with \eqref{claim-2-1} and \eqref{claim-2-2}, we prove the \eqref{clm3-1}.

If $L\leq |y_j|\leq R$, it is easy to see that
\begin{align}\label{claim-2-3}
    \begin{split}
        U_j^{p-1}w_{1,1}=\frac{\lambda_j^{2}}{\langle y_j\rangle^4}\frac{\lambda_1^{\frac{n-2}{2}}R^{2-n}}{\langle y_1\rangle^2}=\ \frac{\lambda^{\frac{2-n}{2}}}{\langle y_1\rangle^2}v_{j,1}\leq \epsilon_1v_{j,1}.
    \end{split}
\end{align}
For $|y_j|>R$, one can use the same type of estimates \eqref{121-1} and \eqref{121-2} to prove
\[U_j^{p-1}w_{1,1}\leq \epsilon_1(v_{j,2}+v_{1,1}).\]
\end{proof}
Combining the above claims, we complete the proof.
\end{proof}

Next, we estimate the coefficients $c_b^j$ in \eqref{w2.3}.
\begin{lemma}\label{lem:c_jb}
Suppose $\sigma$ is the sum of a family of $\delta$-interacting bubbles. If $\phi$, $h$ and $c_b^j$ satisfy \eqref{w2.3}, then
\[|c_{b}^j|\lesssim Q\|h\|_{**}+Q^p\|\phi\|_{*},\quad 1\leq j\leq \nu,\ 1\leq b\leq n+1.\]
\end{lemma}
\begin{proof}
We shall multiply \eqref{w2.3} by $Z_j^b$.
By the Lemma \ref{lem:kernel}, for $a,b\leq n+1$, there exist some constants $\gamma^b$ such that
\[\sum_{i,a}\int c_{a}^{i}U_i^{p-1}Z_{i}^aZ_{j}^b=c_{b}^j\gamma^b+\sum_{i\neq j}\sum_{a=1}^{n+1}c_{a}^i O(q_{ij}).\]
Fix $1\leq j\leq \nu$, $1\leq b \leq n+1$. By the orthogonal condition in \eqref{w2.3} and $|Z_j^b|\lesssim U_j$
\begin{align}
\begin{split}\label{sigma(p-1)-z-phi}
    \left|\int p\sigma^{p-1}\phi Z_j^{b}\right|=&\
        \left|\int p\left(\sigma^{p-1}-U_j^{p-1}\right) Z_j^{b}\phi\right|
        \\
        \lesssim&\ \|\phi\|_{*}\int\left(\sigma^{p-1}- U_j^{p-1}\right)U_jWdx.
    \end{split}
\end{align}
Now we use the fact that, $(\sigma^{p-1}-U_i^{p-1})U_i\geq0$ for each $i$,
\[(\sigma^{p-1}-U_j^{p-1})U_j\leq \sum_i(\sigma^{p-1}-U_i^{p-1})U_i=\sigma^p-\sum_iU_i^p.\]
Then by Proposition \eqref{prop:on-weight},\eqref{sigma(p-1)-z-phi} can be bounded by
\begin{align*}
    \begin{split}
        \left|\int p\sigma^{p-1}\phi Z_j^{b}\right|
        \lesssim&\ \|\phi\|_{*}\int(\sigma^p-\sum_{i=1}^\nu U_i^p)Wdx\lesssim \|\phi\|_{*}\int V(x)W(x)dx \lesssim\ Q^p\|\phi\|_*.
    \end{split}
\end{align*}
Here we have used Lemma \ref{lem:s-U-rho} in the last inequality (see similar estimates \eqref{intVW}) and \eqref{rho0-L2-5}).
Similarly, by Lemma \ref{lem:V-U}  we have
\begin{align}
\begin{split}
   \left| \int h Z_{j}^b\right|\lesssim&\ \|h\|_{**}\sum_{i=1}^\nu \int VU_jdx \lesssim\  Q\|h\|_{**}.
\end{split}
\end{align}
Multiplying \eqref{w2.3} by $Z_j^b$ and  integrating, we see that $\{c_b^j\}$ satisfy the linear system
\begin{align*}
    c_b^j\gamma^b+\sum\limits_{i\neq j}\sum\limits_{a=1}^{n+1}c_a^iO(q_{ij})=\int p\sigma^{p-1}\phi Z_j^{b}+\int h Z_{j}^b,\quad &1\leq j\leq \nu,\ 1\leq b\leq n+1.
\end{align*}
Since $q_{ij}\leq Q\leq \delta$ the system is solvable and we prove the claim.
\end{proof}
From Lemma \ref{lem:prior-est} and \ref{lem:c_jb}, using the same argument as in the proof of Proposition 4.1 in \cite{delPino2003}, we can prove the following result:
\begin{proposition}\label{prop:existence}
There exists positive $\delta_{0}$ and a constant $C>0$, independent of $\delta$, such that for all $\delta \leqslant \delta_{0}$ and all $h$ with $\|h\|_{**}<\infty$, problem \eqref{w2.3} has a unique solution $\phi \equiv L_{\delta}(h) .$ Besides,
\[
\left\|L_{\delta}(h)\right\|_{*} \leqslant C\|h\|_{* *}, \quad\left|c_{a}^i\right| \leqslant C\delta\|h\|_{* *} .
\]
\end{proposition}

\begin{proposition}\label{prop:rho0-exist}
Suppose that $\delta$ is small enough.  There exists a solution $\rho_0$ and a family of scalars $(c_a^i)$ which solve
\begin{align}\label{rho0-non}
    \Delta \phi+[(\sigma+\phi)^p-\sigma^p]+\sigma^p-\sum_{i=1}^\nu U_i^p=\sum_{i,a} c_a^iU_i^pZ_i^a
\end{align}
with
\begin{equation}\label{rho0-pointwise}
|\rho_0|(x)\leq CW(x).
\end{equation}
\end{proposition}
\begin{proof}
Let
\begin{align}
   N_1(\phi)=&(\sigma+\phi)^p-\sigma^p-p\sigma^{p-1}\phi,\\
N_2=&\sigma^p-\sum_{i=1}^\nu U_i^p.
\end{align}
Then, \eqref{rho0-non} is equivalent to
\begin{align}
    \phi=A(\phi)=: -L_\delta(N_1(\phi))-L_\delta(N_2),
\end{align}
where $L_\delta$ is defined in Proposition \ref{prop:existence}. We will show that $A$ is a contraction mapping.

First, it follows from Proposition  \ref{prop:on-weight} that there exists $C_2=C_2(n,\nu)$
\begin{align}
  \|N_2\|_{**}\leq C_2.
\end{align}

Next,  we claim that $\|N_1(\phi)\|_{**}\leq C_1R^{-4(p-1)}\|\phi\|_*$. In fact, since  $|N_1(\phi)|\leq C|\phi|^p\leq C\|\phi\|_{*}^pW^p$,
then
\begin{align}\label{n1phi}
    \begin{split}
        \|N_1(\phi)\|_{**}\leq C\|\phi\|_{*}\sup_{\mathbb{R}^n} W^p(x)V^{-1}(x).
    \end{split}
\end{align}
Notice, for the first type weight function, say $|y_i|\leq R$, we have
\begin{align}
\left(\frac{\lambda_i^{\frac{n-2}{2}}R^{2-n}}{\langle y_i\rangle^2}\right)^p\left(\frac{\langle y_i\rangle^4}{\lambda_i^{\frac{n+2}{2}}R^{2-n}} \right)
=R^{-4}\langle y_i\rangle^{4-2p}\leq R^{-4}.
\end{align}
For the second type weight, say $R/2\leq |y_i|$,
we have
\begin{align}
    \begin{split}
        \left( \frac{\lambda_i^{\frac{n-2}{2}}R^{-4}}{\langle y_i\rangle^{n-4}}\right)^p\left(\frac{\langle y_i\rangle^{n-2}}{\lambda_i^{\frac{n+2}{2}}R^{-4}}\right)
        = R^{-4(p-1)}\langle y_i\rangle^{n-2-p(n-4)}
        \leq
        R^{-4(p-1)},
    \end{split}
\end{align}
since $n-2-p(n-4)\leq 0$. By the inequality \eqref{ineq-1}, we get
\begin{align}
    W^pV^{-1}\leq R^{-4(p-1)}.
\end{align}
 Thus, there exists $C_1=C_1(n,\nu)$ such that
\begin{align}
    \begin{split}
        \|N_1(\phi)\|_{**}\leq C_1R^{-4(p-1)}\|\phi\|_{*}.
    \end{split}
\end{align}
Making $C_1$ possibly larger so that we also have $\|L_\delta(h)\|_{*}\leq C_1\|h\|_{**}$ in Proposition \ref{prop:existence}.

Now define the space
\begin{align}
    E=\{u: \ u\in C(\mathbb{R}^n)\cap D^{1,2}(\mathbb{R}^n),
    \|u\|_{*}\leq C_1C_2+1\}.
\end{align}
We claim that $A$ is a contraction mapping from $E$ to $E$.
In fact, choosing $\delta$ small, then $R$ large such that $R^{-4(p-1)} C_1^2(C_1C_2+1)\leq 1$,  we have
\begin{align}
    \begin{split}
        \|A(\phi)\|_{*}\leq&\ C_1\|N_1(\phi)\|_{**}+C_1\|N_2\|_{**}\\
        \leq&\ R^{-4(p-1)} C^2_1(C_1C_2+1)+C_1C_2\leq C_1C_2+1.
    \end{split}
\end{align}
Thus, $A(E)\subset E$. Furthermore,
\begin{align}
    \begin{split}
        \|A(\phi_1)-A(\phi_2)\|_{*}\leq&\
        \|L_\delta(N_1(\phi_1))-L_\delta(N_1(\phi_2))\|_{**}\\
        \leq&\
        C_1\|N_1(\phi_1)-N_1(\phi_2)\|_{**}.
    \end{split}
\end{align}
If $n\geq6$, then $N_1'(t)\leq C|t|^{p-1}$. As a result,
\begin{align}
    \begin{split}
        |N_1(\phi_1)-N_1(\phi_2)|\leq&\
        C\left(|\phi_1|^{p-1}+|\phi_2|^{p-1}\right)|\phi_1-\phi_2|\\
        \leq&\
        C\left(\|\phi_1\|_{*}^{p-1}+\|\phi_2\|_{*}^{p-1}\right)\|\phi_1-\phi_2\|_{*}W^{p}.
        \end{split}
\end{align}
Since $ W^pV^{-1}\leq R^{-4(p-1)}\ll 1$ if $\delta$ small, we get
\begin{align}
    \|A(\phi_1)-A(\phi_2)\|_{*}\leq \frac{1}{2}\|\phi_1-\phi_2\|_{*}.
\end{align}
Thus, $A$ is a contraction. It follows from the contraction mapping theorem that there is a unique $\phi \in E$, such that $\phi =A(\phi)$. Moreover, it follows Proposition \ref{prop:existence} that
$ \|\phi\|_{*}\leq C$.
\end{proof}
\subsection{Some estimates of the error} In this subsection, we will establish $L^2$ estimates for the $\nabla\rho$, based on the point-wise estimates from the previous subsection.
\begin{proposition}\label{prop:L2-rho0}
Suppose $\delta$ is small enough. We have the gradient estimates
\begin{align}\label{rho0-L2}
\|\nabla\rho_0\|_{L^2}\lesssim \begin{cases}Q|\log Q|^{\frac12}, &\text{if }n=6,\\ Q^{\frac{p}{2}},&\text{if }n\geq 7.\end{cases}
\end{align}
\end{proposition}
\begin{proof}
Multiplying $\rho_0$ to \eqref{rho0-non} to get
\begin{align}\label{rho0-L2-1}
    \int |\nabla \rho_0|^2\lesssim \int \sigma^{p-1}\rho_0^2+\int|\rho_0|^{p+1}+\int (\sigma^p-\sum U_i^p)\rho_0.
\end{align}
It follows from Proposition \ref{prop:rho0-exist} that $|\rho_0(x)|\lesssim W(x)$.
By the Lemma \ref{lem:sigma(p-1)rho2}, we have
\begin{align}\label{rho0-L2-2}
    \begin{split}
        \int \sigma^{p-1}\rho_0^2&\lesssim\
        \sum_{i=1}^\nu \int U_i^{p-1}\rho_0^2\\
        \lesssim&\ \sum_{i=1}^\nu \left(\int_{|y_i|\leq R}\frac{\lambda_i^2}{\langle y_i\rangle^4}\frac{\lambda_i^{n-2}R^{4-2n}}{\langle y_i\rangle^{4}}dx
        +\int_{|y_i|\geq R/2}\frac{\lambda_i^2}{\langle y_i\rangle^4}\frac{\lambda_i^{n-2}R^{-8}}{\langle y_i\rangle^{2n-8}}dx\right)\\
        &\ +\sum_{i\neq j} \left(\int_{|y_i|\leq R}\frac{\lambda_j^2}{\langle y_j\rangle^4}\frac{\lambda_i^{n-2}R^{4-2n}}{\langle y_i\rangle^{4}}dx
        +\int_{|y_i|\geq R/2}\frac{\lambda_j^2}{\langle y_j\rangle^4}\frac{\lambda_i^{n-2}R^{-8}}{\langle y_i\rangle^{2n-8}}dx\right)\\
        \lesssim&\ R^{-n-2}\thickapprox Q^p.
    \end{split}
\end{align}
By Sobolev inequality, we have
\begin{align}\label{rho0-L2-3}
    \int |\rho_0|^{p+1}\lesssim \|\nabla\rho_0\|_{L^2}^{p+1}.
\end{align}
Now consider the last term  in \eqref{rho0-L2-1}.  Proposition \ref{prop:on-weight} says that $|\sigma^p-\sum_i U_i^p|\lesssim V(x)$.

$\bullet$ If $n\geq 7$, using \eqref{define-V} and \eqref{define-W}, we obtain
\begin{align}\label{intVW}
\begin{split}
   &\left| \int (\sigma^p-\sum U_i^p)\rho_0  \right|\lesssim \int V(x)W(x)\\
   &\leq\sum_{i,j=1}^\nu \int_{\substack{|y_i|\leq R\\|y_j|\leq R}} \frac{\lambda_i^{\frac{n+2}{2}}R^{2-n}}{\langle y_i\rangle^{4}}\frac{\lambda_j^{\frac{n-2}{2}}R^{2-n}}{\langle y_j\rangle^{2}}+\int_{\substack{|y_i|\leq R\\|y_j|\geq R/2}}\frac{\lambda_i^{\frac{n+2}{2}}R^{2-n}}{\langle y_i\rangle^{4}}\frac{\lambda_j^{\frac{n-2}{2}}R^{-4}}{\langle y_j\rangle^{n-4}}\\
   &+\sum_{i,j=1}^\nu \int_{\substack{|y_i|\geq R/2\\ |y_j|\leq R}}\frac{\lambda_i^{\frac{n+2}{2}}R^{-4}}{\langle y_i\rangle^{n-2}}\frac{\lambda_j^{\frac{n-2}{2}}R^{2-n}}{\langle y_j\rangle^{2}}+\int_{\substack{|y_i|\geq R/2\\|y_j|\geq R/2}}\frac{\lambda_i^{\frac{n+2}{2}}R^{-4}}{\langle y_i\rangle^{n-2}}\frac{\lambda_j^{\frac{n-2}{2}}R^{-4}}{\langle y_j\rangle^{n-4}}.
   \end{split}
\end{align}
By some tedious computations, Lemma \ref{lem:s-U-rho} shows
\begin{align}\label{rho0-L2-4}
    \int (\sigma^p-\sum U_i^p)\rho_0\lesssim R^{-n-2}\approx Q^p.
\end{align}

$\bullet$ If $n=6$, then $p=2$. By Lemma \ref{lem:6-dim}, we have
\begin{align}\label{rho0-L2-5}
\left|\int (\sigma^2-\sum U_i^2)\rho_0\right|\leq\sum_{i\neq j} \left|\int U_iU_j\rho_0\right|\lesssim R^{-8}\log R\thickapprox Q^2|\log Q|.
\end{align}
Plugging \eqref{rho0-L2-2}, \eqref{rho0-L2-3}, \eqref{rho0-L2-4}, \eqref{rho0-L2-5}  into \eqref{rho0-L2-1}, we get \eqref{rho0-L2}.
\end{proof}

Now consider $\rho_1=\rho-\rho_0$. Recall $\rho$ satisfies \eqref{rho-1} and $\rho_0$ satisfies \eqref{rho0-non}. Thus $\rho_1$ solves
\begin{align}\label{rho1-eq}
    \begin{cases}\Delta \rho_{1}+\left[\left(\sigma+\rho_{0}+\rho_{1}\right)^{p}-\left(\sigma+\rho_{0}\right)^{p}\right]+\sum_{i, a} c_{a}^j U_{j}^{p-1} Z_{j}^a+f=0.\\
    \int \nabla \rho_1 \cdot\nabla Z_i^a=0 \quad i=1,\cdots, \nu;\ a=1,\cdots, n+1.
    \end{cases}
\end{align}
Now we decompose
\begin{align}\label{rho1-decom}
\rho_1=\sum_{i=1}^\nu\beta^i U_i+\sum_{i=1}^\nu\sum_{a=1}^{n+1}\beta_a^iZ_{i}^a+\rho_2,
\end{align}
such that for all $i=1,\cdots, \nu$ and $a=1,\cdots, n+1$
\begin{align}\label{rho2-orth}
    \int \nabla \rho_2\cdot \nabla U_i=0=\int \nabla \rho_2\cdot \nabla Z_{i}^a.
\end{align}

\begin{lemma}\label{lem:rho2}
If $\delta$ is small enough, then $\rho_2$ satisfies
\begin{align}
     \|\nabla\rho_2\|_{L^2}\lesssim \sum_{i=1}^\nu|\beta^i|+\sum_{i=1}^\nu\sum_{a=1}^{n+1}|\beta_a^i|+\|f\|_{H^{-1}}.
\end{align}
\end{lemma}
\begin{proof}
Multiply the \eqref{rho1-eq} by $\rho_2$ and integrating, we get
\begin{align*}
    \int |\nabla \rho_2|^2=&\int [(\sigma+\rho_0+\rho_1)^p-(\sigma+\rho_0)^p]\rho_2+\int |\rho_2f|
\end{align*}
Since the elementary inequality
\begin{align}\label{s-rho01}
    |(\sigma+\rho_0+\rho_1)^p-(\sigma+\rho_0)^p-p(\sigma+\rho_0)^{p-1}\rho_1|\lesssim |\rho_1|^p
\end{align}
then
\begin{align}\label{nrho2}
&\int |\nabla \rho_2|^2\leq p\int |\sigma+\rho_0|^{p-1}|\rho_1\rho_2|+\int |\rho_1|^p|\rho_2|+\int |\rho_2 f|.
\end{align}
Notice the decomposition of $\rho_1$ in \eqref{rho1-decom},
\[p\int |\sigma+\rho_0|^{p-1}|\rho_1\rho_2|\leq p\int |\sigma+\rho_0|^{p-1}\rho_2^2+\mathcal{B}\int |\sigma+\rho_0|^{p-1}U_i|\rho_2|\]
where $\mathcal{B}=\sum_i |\beta^i|+\sum_{i,a}|\beta_a^i|$.
By H\"older inequality
\begin{align}\begin{split}
    &\int |\sigma+\rho_0|^{p-1}U_i|\rho_2| \lesssim \|\rho_2\|_{L^{2^*}}\||\sigma+\rho_0|^{p-1}U_i\|_{L^{(2^*)'}}\lesssim \|\nabla\rho_2\|_{L^2},\\
    &\int |\rho_1|^p|\rho_2|\leq \|\rho_1\|_{L^{2^*}}^p\|\rho_2\|_{L^{2^*}}\lesssim  (\mathcal{B}+\|\nabla\rho_2\|_{L^2})^p\|\nabla\rho_2\|_{L^2}\label{rho1prho2}.\end{split}
\end{align}
It follows from the second variation estimate (for instance, see \cite[Prop 3.1]{bahri1988nonlinear} and \cite[Prop 3.10]{figalli2020sharp}) and the orthogonal condition of $\rho_2$ that there exists constant $\tilde c<1$ such that
\begin{align}
    p\int \sigma^{p-1}\rho_2^2\leq \tilde c\int |\nabla \rho_2|^2.
\end{align}
Recall a simple inequality that if $x>0$ and $p\in (1,2]$ then  $\left\|x+y|^{p-1}-|x|^{p-1}\right|\leq C|y|^{p-1}$ for any $y$. Consequently
\begin{align*}
    p\int |\sigma+\rho_0|^{p-1}\rho_2^2\leq \tilde c\int |\nabla \rho_2|^2+C\int |\rho_0|^{p-1}\rho_2^2\leq (\tilde c+C\|\nabla\rho_0\|_{L^2}^{p-1}) \|\nabla\rho_2\|_{L^2}^2.
\end{align*}
Therefore,
\begin{align}\label{eq:psp}
    p\int |\sigma+\rho_0|^{p-1}|\rho_1\rho_2|\leq (\tilde c+C\|\nabla\rho_0\|_{L^2}^{p-1}) \|\nabla\rho_2\|_{L^2}^2+C\mathcal{B}\|\nabla\rho_2\|_{L^2}.
\end{align}
By Proposition \ref{prop:L2-rho0}, we can make $\|\nabla\rho_0\|_{L^2}\ll 1$. Plugging in \eqref{rho1prho2} and \eqref{eq:psp} to \eqref{nrho2},
\begin{align*}
    \|\nabla\rho_2\|_{L^2}^2\lesssim \mathcal{B}\|\nabla\rho_2\|_{L^2}+(\mathcal{B}+\|\nabla\rho_2\|_{L^2})^p\|\nabla\rho_2\|_{L^2}+\|f\|_{H^{-1}}\|\nabla\rho_2\|_{L^2}.
\end{align*}
Choosing $\delta$ small such that $\|\nabla\rho_2\|_{L^2}<1$ and $\mathcal{B}<1$, then
\begin{align}\label{nrho2Q}
    \|\nabla\rho_2\|_{L^2}\lesssim \mathcal{B}+\|f\|_{H^{-1}}.
\end{align}
\end{proof}
\begin{lemma}\label{lem:beta-i}
If $\delta$ is small, then
\[|\beta^i|+|\beta_a^i|\lesssim Q^2+\|f\|_{H^{-1}}.\]
\end{lemma}
\begin{proof}
We shall multiply \eqref{rho1-eq} by $U_k$ and integrate it. Before that, let us make some preparations. It follows from  \eqref{s-rho01} and $|(\sigma+\rho_0)^{p-1}-U_k^{p-1}|\lesssim \sum_{i\neq k}U_i^{p-1}+|\rho_0|^{p-1}$ that
\begin{align}
\begin{split}
    &\left|\int[(\sigma+\rho_0+\rho_1)^p-(\sigma+\rho_0)^p]U_k-p\int U_k^{p}\rho_1\right|\\
    &\hspace{1cm}\lesssim
    \sum_{i\neq k}\int U_i^{p-1}U_k|\rho_1|+ \int |\rho_1|^p U_k+\int |\rho_0|^{p-1}\rho_1U_k.
\end{split}
\end{align}
By H\"older inequality and Sobolev inequality,
\begin{align*}
    \int |\rho_1|^p U_k&\lesssim \mathcal{B}^p+\int\rho_2^pU_k\lesssim \mathcal{B}^p+\|f\|_{H^{-1}}^p,\\
    \int|\rho_0|^{p-1}\rho_1U_k&\lesssim \|\nabla\rho_0\|_{L^2}^{p-1}\|\nabla\rho_1\|_{L^2}\|U_k\|_{L^{2^*}}\lesssim o(1)\left(\mathcal{B}+\|f\|_{H^{-1}}\right),\\
    \int U_i^{p-1}|\rho_1|U_k&\leq \|\rho_1\|_{L^{2^*}}\|U_i^{p-1}U_k\|_{L^{(2^*)'}}\lesssim o(1)\left(\mathcal{B}+\|f\|_{H^{-1}}\right),\quad i\neq k.
\end{align*}
Here $o(1)$ denotes a quantity goes to $0$ when $\delta \to 0$.
Multiplying \eqref{rho1-eq} by $U_k$ and integrate it. The above estimates give
\begin{align*}
    -\int \nabla \rho_1\nabla U_k+p\int U_k^p\rho_1\lesssim& o(1)\left(\mathcal{B}+\|f\|_{H^{-1}}\right)+\sum_{j,a}|c_a^j| \left|\int U_j^{p-1}Z_j^a U_k\right|+\int |fU_k|.
\end{align*}
We know that $\int U_j^{p-1}Z_j^a U_k=0$ if $j=k$ and $\int U_j^{p-1}Z_j^a U_k\lesssim \int U_i^{p}U_k\approx Q$ if $j\neq k$ by Lemma \ref{lem:2int-est}. It follows from Proposition \ref{prop:on-weight}, Lemma \ref{lem:c_jb} that $|c_a^j|\lesssim Q$. Inserting these estimates to the above equation,
\begin{align*}
    (p-1)\left(\sum_i \beta^i\int U_i^p U_k+\sum_{i,a} \beta_a^i\int U_k^{p}Z_i^a \right)\lesssim o(1)\mathcal{B}+Q^2+\|f\|_{H^{-1}},
\end{align*}
that is
\begin{align}\label{beta-1}
    (p-1)\beta^k+\sum_{i\neq k}\beta^iO(q_{ik})+\sum_{i,a}\beta^i_aO(q_{ik})\lesssim o(1)\mathcal{B}+Q^2+\|f\|_{H^{-1}}.
\end{align}

Since $\int\nabla \rho_1\cdot \nabla Z^k_b=0$, then \eqref{rho1-decom} and \eqref{rho2-orth} imply
\[\sum_{i=1}^\nu\beta^i\int \nabla U_i\cdot\nabla Z^b_k+\sum_{i=1}^\nu\sum_{a=1}^{n+1}\beta_a^i\int \nabla Z_i^a\cdot \nabla Z^b_k=0.\]
Using Lemma \ref{lem:kernel}, it is equivalent to
\begin{align}\label{beta-2}
    \sum_{i\neq k} \beta^iO(q_{ik})+\beta^k_b+\sum_{i\neq k,a\neq b}\beta^i_aO(q_{ik}) =0.
\end{align}
Combining \eqref{beta-1} and \eqref{beta-2}, we can get the conclusion.


\end{proof}
\begin{proposition}\label{prop:rho1-gradient}
Suppose $\delta$ is small enough. We have
\[\|\nabla\rho_1\|_{L^2}\lesssim Q^2+\|f\|_{H^{-1}}.\]
\end{proposition}
\begin{proof}
This just follows from Lemma \ref{lem:beta-i} and Lemma \ref{lem:rho2} with
\[\|\nabla\rho_1\|_{L^2}\lesssim \sum_i |\beta^i|+\sum_{i,a}|\beta_a^i|+\|\nabla\rho_2\|_{L^2}.\]
\end{proof}


Finally, we can prove the estimates which are used in the proof of the main theorem.
\begin{lemma}\label{lem:srU-est}
Suppose $\delta$ is small enough. We have
\begin{align*}
    \left|\int \sigma^{p-1} \rho Z_k^{n+1}\right|=o(Q)+\|f\|_{H^{-1}},\quad \left|\int |\rho|^{p} Z_k^{n+1}\right|=o(Q)+\|f\|_{H^{-1}}.
\end{align*}
\end{lemma}
\begin{proof}
Notice $\rho=\rho_0+ \rho_1$. Then
\begin{align}\label{srZ}
    \int \sigma^{p-1} \rho Z_k^{n+1}=\int \sigma^{p-1} \rho_0 Z_k^{n+1}+\int \sigma^{p-1} \rho_1 Z_k^{n+1}.
\end{align}
By H\"older inequality and Sobolev inequality
\[\left|\int \sigma^{p-1}\rho_1 Z_k^{n+1}\right|\leq \|\rho_1\|_{L^{2^*}}\|\sigma^{p-1}Z_k^{n+1}\|_{L^{(2^*)'}}\lesssim \|\nabla\rho_1\|_{L^2}\lesssim Q^2+\|f\|_{H^{-1}}.\]
It remains to consider the first term on the RHS of \eqref{srZ}. By the orthogonality condition of $\rho_0$, similar to \eqref{sigma(p-1)-z-phi}, one has
\begin{align}
\left|    \int \sigma^{p-1}\rho_0Z_k^{n+1}\right|=\left|\int (\sigma^{p-1}-U_k^{p-1})\rho_0Z_k^{n+1}\right|= o(Q).
\end{align}
Finally, by H\"older inequality and Sobolev inequality
\begin{align}
    \left|\int |\rho|^pZ_k^{n+1}\right|\leq \|\rho\|_{L^{2^*}}^p\lesssim \|\nabla \rho\|_{L^2}^p\lesssim \|\nabla\rho_0\|_{L^2}^p+\|\nabla\rho_1\|_{L^2}^p\lesssim o(Q)+\|f\|_{H^{-1}}^p.
\end{align}
\end{proof}

\section{Sharp example}
Let us consider the two functions $\tilde U_{1}:=U\left[-R e_{1} ; 1\right], \tilde U_{2}:=U\left[{Re}_{1} ; 1\right]$ where $R\gg 1$. One can define $\tilde \sigma=\tilde U_1+\tilde U_2$ and construct norms $\|\cdot\|_{\tilde *}$ and $\|\cdot\|_{\widetilde{**}}$ as \eqref{def:**}.

By Proposition \ref{prop:rho0-exist}, we can find a unique solution $\tilde \rho$ and a family of scalars $(c_a^i)$ such that\footnote{In this special two bubbles $\tilde U_1,\tilde U_2$ setting, one can use potential estimate to prove directly as in \cite{wei2010infinitely}.}
\begin{align}\label{t-rho-eq}
\begin{cases}
\Delta\left(\tilde U_{1}+\tilde U_{2}+\tilde \rho\right)+\left(\tilde U_{1}+\tilde U_{2}+\tilde \rho\right)^{p}+\sum_{j,a} \tilde c_{a}^j \tilde U_{j}^{p-1} \tilde Z_{j}^a=0, \\
\int \nabla \tilde Z_{j}^a\cdot \nabla \tilde\rho=0,\quad j=1,2, \, a=1,\cdots, n+1.
\end{cases}
\end{align}
Here $\tilde Z^a_j$ are the corresponding ones in \eqref{eq:Z} for $\tilde U_1$ and $\tilde U_2$. It follows from Lemma \ref{lem:c_jb}, Proposition \ref{prop:rho0-exist} and Proposition \ref{prop:L2-rho0} that
\begin{align}\label{ex-t-rho}
    \sum_{j, a}\left|\tilde c_{a}^j\right| \lesssim Q\approx  R^{2-n},\quad \|\tilde \rho\|_{\tilde *}\leq C(n,\nu),\quad  \left\|\nabla \tilde \rho\right\|_{L^{2}} \lesssim R^{-\frac{p}{2}(n-2)}.
\end{align}
Now let $u:=\tilde U_{1}+\tilde U_{2}+\tilde \rho$. Then
\[
\Delta u+|u|^{p-1} u=-\sum_{a,j} \tilde c_{a}^j \tilde U_{j}^{p-1} \tilde Z_{j}^a=f.
\]
It is easy to see that
\begin{align}\label{2.5}
\|f\|_{H^{-1}} \approx \sum_{j, k}\left|\tilde c_{a }^j\right| \lesssim R^{2-n}.
\end{align}
Consider
\begin{align}\label{min-opt}
\inf_{\substack{z_1,z_2\in\mathbb{R}^n\\ \l_1>0,\l_2>0}} \left\|\nabla\left(u-\sum_{j=1,2} U\left[z_{j} ; \lambda_{j}\right]\right)\right\|_{L^{2}}
\end{align}
which can be attained by some
\[
U_1:=U\left[z_{1} ; \lambda_{1}\right],\quad  U_2:=U\left[z_{2} ; \lambda_{2}\right].
\]
Necessary we should have
\[
\left\|\nabla\left(u- U_1-U_2\right)\right\|_{L^{2}} \leq\left\|\nabla(u-\tilde U_{1}-\tilde U_{2})\right\|_{L^{2}} \lesssim R^{-\frac{p}{2}(n-2)}.
\]
Hence
\begin{align}\label{tU-U}
    \left\|\nabla\left(\tilde U_{1}+\tilde U_{2}-U_1-U_2\right)\right\|_{L^{2}} \lesssim R^{-\frac{p}{2}(n-2)}.
\end{align}
This implies that: up to some reordering
\begin{align}\label{ljz1}
\lambda_{j}=1+o_R(1),\quad  z_{1}=-(R+o_R(1)) e_{1}, \quad z_{2}=(R+o_R(1)) e_{1}.
\end{align}
Here $o_R(1)$ means a quantity goes to $0$ when $R\to \infty$.
Denote $\rho=u-U_1-U_2$, then \eqref{t-rho-eq} means that $\rho$ satisfies
\begin{align}\label{opt-rho}
\begin{cases}
\Delta\rho+\left({U}_{1}+{U}_{2}+{\rho}\right)^{p}-U_1^p-U_2^p+\sum_{j, a} \tilde c_{a }^j \tilde U_{j}^{p-1} \tilde Z_{j}^a=0 \\
\int {U}_{j}^{p-1} {Z}_{j}^a {\rho}=0,\quad j=1,2,\, a=1,\cdots,n+1.
\end{cases}
\end{align}
Denote $\sigma=U_1+U_2$,  $y_1=x-z_1$
and $y_2=x-z_2$.
Because of \eqref{ljz1}, the norm $\|\cdot\|_{\tilde *}\approx \|\cdot\|_{*}$ and $\|\cdot\|_{\widetilde{**}}\approx\|\cdot\|_{**}$ with $\|\cdot\|_{*}$  and $\|\cdot\|_{**}$ defined in \eqref{def:**}. Then \eqref{ex-t-rho} implies \[\|\rho\|_{*}\leq C(n,\nu),\quad \|\nabla \rho\|_{L^2}\lesssim R^{-\frac{p}{2}(n-2)}.\]

In order to get a lower bound of  $\|\nabla\rho\|_{L^2}$, we need to obtain the precise first order term of $\rho$. The idea is to match the inner error with the outer error up to leading order.

Suppose $\phi_{1}(x)$ satisfies the following equation
\begin{align}\label{v1-opt}
    \begin{cases}
    \Delta \phi+pU_1^{p-1}\phi+R^{n-2}{U_1^{p-1}}U_2\chi_{\{|x-z_1|<R\}}+R^{n-2}\sum_{i,a}\hat c_a^iU_i^{p-1}Z_i^a=0,\\
    \int U_i^{p-1}Z_i^a \phi=0, \quad i=1,2;\, a=1,\cdots,n+1.
    \end{cases}
\end{align}
Similarly let $\phi_{2}(x)$ satisfy the following equation
\begin{align}\label{v2-opt}
    \begin{cases}
    \Delta \phi+pU^{p-1}_2\phi+R^{n-2}{U_2^{p-1}}U_1\chi_{\{|x-z_2|<R\}}+R^{n-2}\sum_{a}\bar c_a^iU^{p-1}_iZ_i^a=0,\\
    \int U_i^{p-1}Z_i^a \phi=0,\quad i=1,2;\, a=1,\cdots,n+1.
    \end{cases}
\end{align}
Since $R^{n-2}U_1^{p-1}U_2=\alpha_n ^pR^{n-2}\langle2R\rangle^{2-n}\langle x-z_1\rangle^{-4}$, then it is easy to show that
\begin{align}
    \phi_{1}(x)=&A_0|x-z_1|^{-2}+o_R(|x-z_1|^{-2}),\quad 1\ll |x-z_1|\leq R, \\
    \phi_2(x)=&A_0|y- z_2|^{-2}+o_R(|x-z_2|^{-2}),\quad 1\ll |x-z_2|\leq R,
\end{align}
where
\[A_0= \frac{2^{1-n}\alpha_n^{p}}{n-4} , \quad \alpha_n=[n(n-2)]^{\frac{n-2}{4}}.\]
Suppose $\phi_3$ satisfies
\begin{align}\label{v3-opt}
    \Delta \phi +\left(\frac{\al_n}{ |x-\hat{z}_1|^{n-2}} +\frac{\al_n}{|x-\hat{z}_2|^{n-2}}\right)^{p} - \frac{\al_n^p}{ |x-\hat{z}_1 |^{p(n-2)}} -\frac{\al_n^p}{ |x-\hat{z}_2 |^{p(n-2)}} =0
\end{align}
with $\phi(\infty)=0$. Here $\hat z_1=z_1/R$ and $\hat z_2=z_2/R$. Then it is easy to know
\begin{align}
\label{vinfty}
    \phi_3=& A_0 |x-\hat{z}_1|^{-2}+o_R(|x-\hat z_1|^{-2}),\quad  \text{near } \hat z_1,\\
    \phi_3=& A_0 |x-\hat{z}_2|^{-2}+o_R(|x-\hat z_2|^{-2}),\quad  \text{near } \hat z_2.
\end{align}

Fix $s>0$ small.
It is easy to see that in the region $ |y_i|\sim R^{1-s}$,   the two functions $ R^{2-n} \phi_{i} (y_i) $ and $ R^{-n} \phi_{3} (\frac{x}{R})$ differ by
\begin{align}\label{match}
    R^{2-n} \phi_{i} ( y_i)- R^{-n} \phi_{3} (\frac{x}{R})=R^{2-n}o_R(\langle y_i\rangle^{-2}),
\end{align}
for $i=1,2$.
Now we can prove the Theorem \ref{thm:7counter} in the introduction.
\begin{proof}[Proof of Theorem \ref{thm:7counter}]
Let $  \eta(t)=1$ for $ t<R^{1-s}$ and $ \eta(t)=0$ for $ t>2R^{1-s}$. Denote $\eta_1=\eta(|y_1|)$, $\eta_2=\eta(|y_2|)$ and
\begin{align}\label{phi-f}
    \phi(x)=R^{2-n}[\phi_{1}(y_1)\eta_1+\phi_{2}(y_2)\eta_2]+R^{-n}\phi_{3}(\frac{x}{R})(1-\eta_1-\eta_2)
\end{align}
where $y_1=x-z_1$ and $y_2=x-z_2$.
Then one can compute
\begin{align*}
    \Delta \phi+p\sigma^{p-1}\phi+\sigma^p-U_1^p-U_2^p=-\sum_{a,i}(\eta_1\hat c_a^i+\eta_2\bar c_a^i)U_i^{p-1}Z_i^a+J_1+J_2+J_3+J_4,
\end{align*}
where
\begin{align*}
    J_1=&p(\sigma^{p-1}-U_1^{p-1})R^{2-n}\phi_{1}(y_1)\eta_1+p(\sigma^{p-1}-U_2^{p-1})R^{2-n}\phi_{2}(y_2)\eta_2,\\
    J_2=&[\sigma^p-U_1^p-U_2^p]-U_1^{p-1}U_2\eta_1-U_2^{p-1}U_1\eta_2,\\
    &-(1-\eta_1-\eta_2)[(\frac{\alpha_n}{|y_1|^{n-2}}+\frac{\alpha_n}{|y_2|^{n-2}})^p-\frac{\alpha_n^p}{|y_1|^{p(n-2)}}-\frac{\alpha_n^p}{|y_2|^{p(n-2)}}]\\
    J_3=&\nabla\eta_1\cdot\nabla(R^{2-n}\phi_{1}(y_1)-R^{-n}\phi_{3}(\frac{x}{R}))+\Delta \eta_1(R^{2-n}\phi_{1}(y_1)-R^{-n}\phi_{3}(\frac{x}{R})),\\
    J_4=&\nabla\eta_2\cdot \nabla(R^{2-n}\phi_{2}(y_2)-R^{-n}\phi_{3}(\frac{x}{R}))+\Delta \eta_2 (R^{2-n}\phi_{2}(y_2)-R^{-n}\phi_{3}(\frac{x}{R})).
\end{align*}
Then combine with \eqref{opt-rho},
\begin{align}\label{rho-phi-1}
   \begin{cases} \Delta (\rho-\phi)+p\sigma^{p-1}(\rho-\phi)=h+\sum_{a,i}(\eta_1\hat c_a^i+\eta_2\bar c_a^i-\tilde c_a^i)U_i^{p-1}Z_i^a,\\
   \int U_i^{p-1}Z_i^a\phi=0,\quad i=1,2,\, a=1,\cdots,n+1,
   \end{cases}
\end{align}
where
\[h=[-(\sigma+\rho)^p+\sigma^p+p\sigma^{p-1}\rho]-\sum_{i=1}^4J_i+\sum_{i,a}\tilde c^i_a(U_i^{p-1}Z^a_i-\tilde U_i^{p-1}\tilde Z^a_i).\]
Now let us show $\|h\|_{**}=o_R(1)$. The first term can be controlled similar to \eqref{n1phi} to be
\begin{align}
    \|(\sigma+\rho)^p-\sigma^p-p\sigma^{p-1}\rho\|_{**}\lesssim \||\rho|^{p}\|_{**}\lesssim o_R(1).
\end{align}
Now it is easy to see
\begin{align*}
    \|J_1\|_{**}=\sup_{y_1,y_2\in \mathbb{R}^n} \sum_{i=1,2} |\sigma^{p-1}-U_i^{p-1}|\langle y_i\rangle^4|\phi_{i}(y_i)|\eta_i\lesssim R^{-2}.
\end{align*}
To control $J_2$, we notice $U_2\lesssim R^{-s} U_1$ on $\eta_1>0$ and $U_1\lesssim R^{-s}U_2$ on $\eta_2>0$. Then
\begin{align*}
    \|J_2\|_{**}
    \lesssim&\sup_{y_1} R^{n-2}\langle y_1\rangle^4[\sigma^p-U_1^p-U_2^p-U_1^{p-1}U_2]\eta_1\\
    &+\sup_{y_2} R^{n-2}\langle y_2\rangle^4[\sigma^p-U_1^p-U_2^p-U_2^{p-1}U_1]\eta_2\\
    &+\sup R^4(|y_1|^{n-2}+|y_2|^{n-2})\left[\sigma^{p}-\left(\frac{\alpha_n}{|y_1|^{n-2}}+\frac{\alpha_n}{|y_2|^{n-2}}\right)^p\right.\\
    &\hspace{3cm}\left.-U_1^p-U_2^p+\frac{\alpha_n^p}{|y_1|^{p(n-2)}}+\frac{\alpha_n^p}{|y_2|^{p(n-2)}}\right](1-\eta_1-\eta_2)\\
    \lesssim &R^{-4s}+o_R(1).
\end{align*}
For $J_3$ and $J_4$, we need to use the fact in \eqref{match}
\begin{align*}
    \|J_3\|_{**}=o_R(1),\quad \|J_4\|_{**}=o_R(1).
\end{align*}
Because of \eqref{ex-t-rho} and \eqref{tU-U}, then
\begin{align}\label{c-diff}
    \left\| \sum_{i,a}\tilde c^i_a(U_i^{p-1}Z^a_i-\tilde U_i^{p-1}\tilde Z^a_i)\right\|_{**}\lesssim o_R(1)R^{2-n}.
\end{align}
Combine the previous estimates, we get $\|h\|_{**}=o_R(1)$.
Therefore by Lemma \ref{lem:prior-est},
$$ \| \rho -\phi\|_{*} \leq o_R(1). $$
This gives a lower bound on $\rho$, with error $o_R(1)$. From this we can multiply \eqref{opt-rho} by $\rho$ and integrate
\begin{align}\label{ex-nrho}
    \int |\nabla \rho|^2 =& p \int \sigma^{p-1} \rho^{2} + \int (\sigma^p -U_1^p-U_2^p) \rho  + \sum_{i,a}\tilde c_a^j\int \tilde U_j^{p-1}\tilde Z^a_j\rho.
\end{align}
It follows from \eqref{ljz1} and \eqref{opt-rho} that
\begin{align*}
    \left|\int\tilde U^{p-1}_j\tilde Z^a_j\rho\right|=\left|\int (\tilde U^{p-1}_j\tilde Z^a_j- U^{p-1}_j Z^a_j)\rho\right|\lesssim& o_R(1)\int WV
    \lesssim \begin{cases}o(R^{-n-2}), &n\geq 7,\\ o(R^{-8}\log R), &n=6,\end{cases}
\end{align*}
where we have used the estimates in \eqref{rho0-L2-3} and \eqref{rho0-L2-5}. Similarly using $\|\rho-\phi\|_{*}= o_R(1)$, we can derive
\begin{align*}
    \left|\int (\sigma^p-U_1^p-U_2^p)(\rho-\phi)\right|\lesssim o_R(1)\int VW\lesssim \begin{cases}o(R^{-n-2}), &n\geq 7,\\ o(R^{-8}\log R),& n=6.\end{cases}
\end{align*}
However,  direct computations shows, using \eqref{phi-f},
\begin{align*}
    \int (\sigma^p-U_1^p-U_2^p)\phi\gtrsim \int_{R^{1-s}\leq |y_1|\leq R}U_1^pU_2\phi=&R^{4-2n}\int_{R^{1-s}\leq |y_1|\leq R}\frac{dx}{\langle y_1\rangle^6}\\
    \gtrsim&\begin{cases}
    R^{-n-2}\approx Q^p,& n\geq 7,\\ R^{-8}\log R\approx Q^2|\log Q|,&n=6.
    \end{cases}
\end{align*}
Plugging in the above facts to \eqref{ex-nrho} to get
\[\|\nabla \rho\|_{L^2}^2\geq \begin{cases}
     Q^p,& n\geq 7,\\ Q^2|\log Q|,&n=6.
    \end{cases}\]
Since $\|f\|_{H^{-1}}\lesssim R^{2-n}\approx Q$, the proof is done.
\end{proof}

\section*{Acknowledgement} The research of  B. Deng is supported by China Scholar Council and Natural Science Foundation of China (No. 11721101). The research of L. Sun and J. Wei is partially supported by NSERC of Canada.

\appendix
\section{Some useful estimates}
This appendix contains some useful estimates involved Talenti bubbles and their derivatives.
\begin{lemma}\label{lem:B-inf}
Let $\alpha>\beta>1$ and $\alpha+\beta=2^*$,
\[\int U_i^\beta\inf (U_i^\alpha,U_j^\alpha)=O(q_{ij}^{\frac{n}{n-2}}|\log q_{ij}|)\]
\end{lemma}
\begin{proof}
See the proof in \cite[E4]{bahri1989critical}.
\end{proof}
\begin{lemma}\label{lem:U12-est}
For any two bubbles, there exists $C_n$ such that
\[\int_{\Rn}U[z_i,\l_i]^{p}\l_j\partial_{\l_j}U[z_j,\l_j]=-C_n \left(q_{ij}^{-\frac{2}{n-2}}-2\frac{\l_i}{\l_j} \right)q_{ij}^{\frac{n}{n-2}}+O(q_{ij}^{\frac{n}{n-2}}\log q_{ij}^{-1}) \]
\end{lemma}
\begin{proof}
See the proof in \cite[F16]{bahri1989critical}. Moreover, if $\l_i\leq \l_j$ then RHS $\approx -q_{ij}$ when $q_{ij}\ll 1$.
\end{proof}


\begin{lemma}\label{lem:2int-est}
Given $n\geq3$, let $U_i=U[z_i,\lambda_i]$, $i=1,2$, be two bubbles. Then, for any fixed $\varepsilon>0$ and any non-negative exponents such that $\alpha+\beta=2^*$, it holds
\begin{align}
    \int_{\mathbb{R}^n}U_1^{\alpha}U_2^{\beta}\thickapprox_{n,\varepsilon}
    \begin{cases}
    q_{12}^{\min\left(\alpha,\beta\right)}\quad& \text{if}\quad  |\alpha-\beta|\geq\varepsilon,\\
    q_{12}^{\frac{n}{n-2}}\log\left(\frac{1}{q_{12}}\right) \quad& \text{if}\quad \alpha=\beta.
    \end{cases}
\end{align}
\end{lemma}
\begin{proof}
See the proof of proposition B.2 in \cite{figalli2020sharp}.
\end{proof}

\begin{lemma}\label{lem:3int-est}
Given $n\geq6$, let $U_i=U[z_i,\lambda_i]$, $i=1,2,3$, be three bubbles with $\delta$-interaction, that is $ Q:=\max\{q_{12}, q_{13}, q_{23}\}<\delta$. Suppose $\delta$ is small enough. Then
\begin{enumerate}
    \item For $n=6$, we have
    \begin{align}\label{3int-n=6}
        \int_{\mathbb{R}^n}U_1U_2U_3\lesssim
        Q^{\frac{3}{2}}\log\left(\frac{1}{Q}\right).
    \end{align}
    \item For $n\geq7$, we have
    \begin{align}\label{3int-ngeq7}
        \int_{\mathbb{R}^n}U_1^{p-1}U_2U_3\lesssim
        Q^{\frac{n-1}{n-2}}|\log Q|^{\frac{n-5}{n}}.
    \end{align}
\end{enumerate}
\end{lemma}


\begin{proof}
For $n=6$, $2^*=3$, by  the H\"older inequality, we get
\begin{align}\label{6-tri}
    \int_{\mathbb{R}^n}U_1U_2U_3\leq\left(\int_{\mathbb{R}^n}U_1^{\frac{3}{2}}U_2^{\frac{3}{2}}\right)^{\frac{1}{3}}
    \left(\int_{\mathbb{R}^n}U_1^{\frac{3}{2}}U_3^{\frac{3}{2}}\right)^{\frac{1}{3}}
    \left(\int_{\mathbb{R}^n}U_2^{\frac{3}{2}}U_3^{\frac{3}{2}}\right)^{\frac{1}{3}}.
\end{align}
By the Lemma \ref{lem:2int-est}, we have
\begin{align}
    \begin{split}
        \int_{\mathbb{R}^n}U_1U_2U_3\lesssim
        q_{12}^{\frac{1}{2}}|\log q_{12}|^{\frac{1}{3}}
        q_{13}^{\frac{1}{2}}|\log q_{13}|^{\frac{1}{3}}
        q_{23}^{\frac{1}{2}}|\log q_{23}|^{\frac{1}{3}}.
    \end{split}
\end{align}
Since the function $x^{\frac{1}{2}}|\log x|^{\frac{1}{3}}$ increasing near zero, choosing $\delta$ small, we get \eqref{3int-n=6}

For $n\geq7$, $2^*=\frac{2n}{n-2}$, let $\alpha=\frac{4n}{5(n-2)}, \beta= \frac{6n}{5(n-2)}$, $s_1=\frac{5}{2n}$ and $s_2=\frac{n-5}{n}$. By  the H\"older inequality and the Lemma \ref{lem:2int-est}, we get
\begin{align}
    \begin{split}
        \int_{\mathbb{R}^n}U_1U_2U_3\leq&\
        \left(U_1^{\alpha} U_2^{\beta} \right)^{s_1}
        \left(U_1^{\alpha} U_3^{\beta} \right)^{s_1}
        \left(U_2^{\frac{2^*}{2}} U_3^{\frac{2^*}{2}} \right)^{s_2}\\
        \lesssim&\
        q_{12}^{\frac{2}{n-2}} q_{13}^{\frac{2}{n-2}} q_{23}^{\frac{n-5}{n-2}}|\log q_{23}|^{\frac{n-5}{n}}.
        \end{split}
\end{align}
Since the function $x^{\frac{n-5}{n-2}}|\log x|^{\frac{n-5}{n}}$ increasing near zero, choosing $\delta$ small, we get \eqref{3int-ngeq7}
\end{proof}

\begin{lemma}\label{lem:kernel}
For the $Z_i^a$ defined in \eqref{eq:Z}, there exist some constants $\gamma^a=\gamma^a(n)>0$ such that
\begin{align}
\int U_i^{p-1}Z_{i}^aZ_{i}^b=\begin{cases}0\quad &\text{if }\ a\neq b, \\
\gamma^a \quad &\text{if}\ \ 1\leq a=b\leq n+1.\\
\end{cases}
\end{align}
If $i\neq j$ and $1\leq a,b\leq n+1$, we have
\begin{align}
    \left|\int U_i^{p-1}Z_{i}^aZ_{j}^b\right|\lesssim q_{ij}.
\end{align}
\begin{proof}
See the proof in \cite[F1-F6]{bahri1989critical}. Moreover, it is known that $\gamma^1=\gamma^2=\cdots=\gamma^n$.
\end{proof}
\begin{lemma}\label{lem:aip}
Suppose $p\in (1,2]$ and $a_i\geq0$, then
\begin{align}\label{aip}
    \left(\sum_{i=1}^\nu a_i\right)^p-\sum_{i=1}^\nu a_i^p\leq \sum_{i\neq j} [(a_i+a_j)^p-a_i^p-a_j^p]
\end{align}
the equality holds when at most one of $a_i$ is non-zero.
\end{lemma}
\begin{proof}
It is equivalent to prove
\[f(a_1,a_2,\cdots,a_\nu)=\left(\sum_{i=1}^\nu a_i\right)^p+(\nu-2)\sum_{i=1}^\nu a_i^p- \sum_{i\neq j} (a_i+a_j)^p\leq 0\]
Denote $a_1+a_2=s$. Define $g(x)=f(x,s-x,a_3,\cdots, a_\nu)$. It is easy to see
\begin{align*}
    \frac{g''(x)}{p(p-1)}=(\nu-2)[x^{p-2}+(s-x)^{p-2}]-\sum_{i=3}^\nu [(x+a_i)^{p-2}+(s-x+a_i)^{p-2}]
\end{align*}
Since $p-2<0$ and $a_i\geq 0$, then $g''(x)\geq 0$ for $x\in [0,s]$. Since $g(0)=g(s)$, we must have $g$ achieves the maximum at $x=0$ or $s$. Therefore $f(a_1,a_2,\cdots,a_\nu)\leq f(0,a_1+a_2,\cdots, a_\nu)$. Repeating the above process for any pairs, we obtain $f\leq 0$.
\end{proof}
\end{lemma}



\section{Integrals required in section 3}
This appendix is devoted to the computations of integral quantities mainly required in section 3. Precisely, they are used to bound the integral quantities $\int \sigma^{p-1}\rho_0^2$ and $\int (\sigma^p-\sum U_i^p)\rho_0$. The strategy is that: if the integrals essentially much less than $R^{-n-2}$, we can simply use H\"older inequality and Lemma \ref{lem:2int-est} to give them bounds no greater than $R^{-n-2}$. Otherwise, we need compute the integrals in four cases: bubble tower with $U_1$ higher (lower) than $U_2$, bubble cluster with $U_1$ higher (lower) than $U_2$. In each case, we split the involved integrals in regions where the integrand  has a power-like behavior and then computing the integrals explicitly.

Recall that
\begin{align}
    U_i(x)=U[z_i,\lambda_i](x)\thickapprox\frac{\lambda_i^{\frac{n-2}{2}}}{(1+\lambda_i^2|x-z_i|^2)^{\frac{n-2}{2}}}=\frac{\lambda_i^{\frac{n-2}{2}}}{\langle y_i\rangle^{n-2}},\quad i=1,2,3,
\end{align}
where $y_i=\lambda_i(x-z_i)$ and  $\langle y\rangle=\sqrt{1+|y|^2}$. For $i\neq j$,
\begin{align}
    R_{ij}:=\max\left\{\sqrt{\lambda_i/\lambda_j}, \sqrt{\lambda_j/\lambda_i}, \sqrt{\lambda_i\lambda_j}|z_i-z_j|\right\}\thickapprox q_{ij}^{-\frac{1}{n-2}}.
\end{align}



\begin{lemma}\label{lem:sigma(p-1)rho2}
Suppose $n\geq 6$ and $1\ll R\leq R_{12}/2$, we have
\begin{align}
    \int_{|y_1|\leq R} \frac{\lambda_1^2}{\langle y_1\rangle^4}\frac{\lambda_1^{n-2}R^{4-2n}}{\langle y_1\rangle^{4}}dx
    \thickapprox&\
    \begin{cases} R^{4-2n},\quad & n=6,7\\
    R^{-12}\log R,\quad &n=8,\\
    R^{-n-4},\quad &n>8,\end{cases}\label{sigma(p-1)rho2-1}\\
    \int_{|y_1|\geq R} \frac{\lambda_1^2}{\langle y_1\rangle^4}\frac{\lambda_1^{n-2}R^{-8}}{\langle y_1\rangle^{2n-8}}dx\thickapprox&\ \  R^{-n-4},\quad\quad\quad\quad \ \ n\geq 6,\label{sigma(p-1)rho2-2}\\
    \int_{|y_1|\leq R} \frac{\lambda_2^2}{\langle y_2\rangle^4}\frac{\lambda_1^{n-2}R^{4-2n}}{\langle y_1\rangle^{4}}dx\lesssim&\ \  R^{-n-4}|\log R|^{\frac{4}{n}},\quad n\geq 6,\label{sigma(p-1)rho2-3}\\
    \int_{|y_1|\geq R} \frac{\lambda_2^2}{\langle y_2\rangle^4}\frac{\lambda_1^{n-2}R^{-8}}{\langle y_1\rangle^{2n-8}}dx\lesssim&\
    \begin{cases}
        R^{-n-2},\quad &n=6,7,8,\\
        R^{-n-4},\quad &n\geq9.
    \end{cases}
    \label{sigma(p-1)rho2-4}
\end{align}

\end{lemma}
\begin{proof}
(1) The first integral in \eqref{sigma(p-1)rho2-1} is equivalent to
\begin{align}
\begin{split}
    \lambda_1^{n}R^{4-2n}\int_{|y_1|\leq R}\frac{dx}{\langle y_1\rangle^{8}}=
    R^{4-2n}\int_{|y_1|\leq R}\frac{dy_1}{\langle y_1\rangle^{8}}
    \thickapprox\begin{cases} R^{4-2n},\quad & n=6,7,\\
    R^{-12}\log R,\quad &n=8,\\
    R^{-4-n},\quad &n\geq 9.\end{cases}
\end{split}
\end{align}

\noindent
(2) The second integral in \eqref{sigma(p-1)rho2-2} is equivalent to
\begin{align}
    R^{-8}\int_{|y_1|\geq R} \frac{dy_1}{\langle y_1\rangle^{2n-4}}\thickapprox R^{-n-4}.
\end{align}

\noindent
(3) The third integral in \eqref{sigma(p-1)rho2-3} is equivalent to
\begin{align}\nonumber
    \begin{split}
        \lambda_1^{n-4}R^{4-2n}\int_{|y_1|\leq R}&U_2^{\frac{4}{n-2}}(x) U_1^{\frac{4}{n-2}}(x)dx\\
        \leq&\
        \lambda_1^{n-4}R^{4-2n}\left(\int_{|y_1|\leq R}\frac{dy_1}{\lambda_1^n}\right)^{\frac{n-4}{n}}
        \left(\int_{|y_1|\leq R}U_2^{\frac{n}{n-2}}U_1^{\frac{n}{n-2}}dx\right)^{\frac{4}{n}}
        \\
        \lesssim&\ R^{-n} q_{12}^{\frac{4}{n-2}}|\log q_{12}|^{\frac{4}{n}}.
    \end{split}
\end{align}
Here we have used the H\"older inequality and the Lemma \ref{lem:2int-est}.
Recall that $R\leq R_{12}\thickapprox q_{12}^{-\frac{1}{n-2}}$ and the function $x^{4}|\log x|^{\frac{4}{n}}$ is increasing near zero, then we get
\begin{align}
    \int_{|y_1|\leq R} U_2^{\frac{4}{n-2}}(x)\frac{\lambda_1^{n-2}R^{4-2n}}{\langle y_1\rangle^{4}}dx\lesssim R^{-n-4}|\log R|^{\frac{4}{n}}.
\end{align}
(4)
Set $z=\lambda_1(z_2-z_1)$ and $\lambda=\lambda_2/\lambda_1$, the integral in \eqref{sigma(p-1)rho2-4} is equivalent to
\begin{align}\label{4-change}
    \begin{split}
       \int_{|y_1|\geq R} \frac{\lambda_2^2}{\langle y_2\rangle^4}\frac{\lambda_1^{n-2}R^{-8}}{\langle y_1\rangle^{2n-8}}dx=\  R^{-8}\int_{|y_1|\geq R}\frac{1}{(\lambda^{-1}+\lambda|y_1-z|^2)^{2}}\frac{dy_1}{\langle y_1\rangle^{2n-8}}.
    \end{split}
\end{align}
\emph{The case $\lambda_1\geq \lambda_2$ and $R_{12}=\sqrt{\lambda_1/\lambda_2}$}. We have $|z|\leq \lambda^{-1}=R_{12}^2$, then it holds that $|y_1-z|\geq |y_1|/2$ when $|y_1|\geq 2\lambda^{-1}$. Thus, from \eqref{4-change} we have
\begin{align}
       \int_{|y_1|\geq R} \frac{\lambda_2^2}{\langle y_2\rangle^4}\frac{\lambda_1^{n-2}R^{-8}}{\langle y_1\rangle^{2n-8}}dx
       \lesssim&\ R^{-8}\left(\lambda^{2}\int_{R\leq |y_1|\leq 2\lambda^{-1}}\frac{dy_1}{\langle y_1\rangle^{2n-8}} +\lambda^{-2}\int_{|y_1|\geq 2\lambda^{-1}} \frac{dy_1}{\langle y_1\rangle^{2n-4}}\right)\notag\\
       \lesssim&\
       \begin{cases}
           R^{-n-4},\quad  &n\geq 9,\\
           R^{-12}\log R,\quad &n=8,\\
        R^{4-2n}, \quad &n=6,7.
       \end{cases}\label{4-change-1}
\end{align}
\emph{The case $\lambda_1\leq \lambda_2$ and $R_{12}=\sqrt{\lambda_2/\lambda_1}$}. We have $|z|\leq 1$, then it holds that $|y_1-z|\geq |y_1|/2$ when $|y_1|\geq R\gg1$. Thus, from \eqref{4-change} we have
\begin{align}\label{4-change-2}
    \begin{split}
      \int_{|y_1|\geq R} \frac{\lambda_2^2}{\langle y_2\rangle^4}\frac{\lambda_1^{n-2}R^{-8}}{\langle y_1\rangle^{2n-8}}dx
       \lesssim&\ R^{-8}\lambda^{-2}\int_{|y_1|\geq R}\frac{1}{(\lambda^{-2}+|y_1-z|^2)^{2}}\frac{dy_1}{\langle y_1\rangle^{2n-8}}\\
       \lesssim&\ R^{-8}R_{12}^{-4}\int_{|y_1|\geq R}\frac{dy_1}{\langle y_1\rangle^{2n-4}}\lesssim R^{-8-n}.
    \end{split}
\end{align}
\emph{The case $\lambda_1\geq \lambda_2$ and $R_{12}=\sqrt{\lambda_1\lambda_2}|z_1-z_2|$}. We have $R_{12}=\sqrt{\lambda}|z|\leq |z|$ and $\lambda|z|\geq1$. Set $\hat{y}=y_1/|z|$ and $e=z/|z|$, we have
\begin{align}\label{4-change-3}
    \begin{split}
       &R^{-8}\int_{|y_1|\geq R}\frac{1}{(\lambda^{-1}+\lambda|y_1-z|^2)^{2}}\frac{dy_1}{\langle y_1\rangle^{2n-8}} \\
   \thickapprox&\  R^{-8}\lambda^{-2}\int_{|\hat{y}|\geq R/|z|}\frac{|z|^{4-n}}{(\lambda^{-1}|z|^{-1}+|\hat{y}-e|)^{4}}\frac{d\hat{y}}{(|z|^{-1}+|\hat{y}|)^{2n-8}}\\
       \lesssim&\ R^{-8}\lambda^{-2}|z|^{4-n}\left( \int_{R/|z|\leq |\hat{y}|\leq 1/2}\frac{d\hat{y}}{|\hat{y}|^{2n-8}}+\int_{1/2\leq |\hat{y}|\leq 2}\frac{d\hat{y}}{|\hat{y}-e|^{4}}+\int_{|\hat{y}|\geq 2}\frac{d\hat{y}}{|\hat{y}|^{2n-4}}\right)\\
       \lesssim&\
       \begin{cases}
           R^{-8}(\lambda|z|)^{-2}|z|^{6-n}\lesssim R^{-n-2},\quad &n=6,7,\\
           R^{-8}\lambda^{-2}|z|^{-4}\log \left( \frac{|z|}{R}\right)\lesssim R^{-10}\lambda|\log\left(\frac{\lambda}{2}\right)|\lesssim R^{-10},\quad &n=8,\\
           R^{-n}(\sqrt{\lambda}|z|)^{-4}\lesssim R^{-n-4},\quad &n\geq 9.
       \end{cases}
    \end{split}
\end{align}
\emph{The case $\lambda_1\leq \lambda_2$ and $R_{12}=\sqrt{\lambda_1\lambda_2}|z_1-z_2|$}. Similar to the case above, we have
\begin{align}\label{4-change-4}
    \begin{split}
       R^{-8}\int_{|y_1|\geq R}\frac{1}{(\lambda^{-1}+\lambda|y_1-z|^2)^{2}}\frac{dy_1}{\langle y_1\rangle^{2n-8}}
       \lesssim\
       \begin{cases}
           R^{-n-2},\quad &n=6,7,8,\\
           R^{-n-4},\quad &n\geq 9.
       \end{cases}
    \end{split}
\end{align}
Together with \eqref{4-change-1}-\eqref{4-change-4}, we get \eqref{sigma(p-1)rho2-4}.
\end{proof}

\begin{lemma}\label{lem:s-U-rho}
Suppose $n\geq 6$ and $1\ll R\leq R_{12}/2$, we have
\begin{align}
\int_{|y_1|\leq R} \frac{\l_1^{\frac{n+2}{2}}R^{2-n}}{\langle y_1\rangle^4}\frac{\lambda_1^{\frac{n-2}{2}}R^{2-n}}{\langle y_1\rangle^{2}}dx\thickapprox&\
\begin{cases}
  R^{-8}\log R,\quad &n=6,\\
  R^{-n-2},\quad &n\geq 7,\\
\end{cases}\label{s-c-1}\\
\int_{|y_1|\geq R} \frac{\l_1^{\frac{n+2}{2}}R^{-4}}{\langle y_1\rangle^{n-2}}\frac{\lambda_1^{\frac{n-2}{2}}R^{-4}}{\langle y_1\rangle^{n-4}}dx\thickapprox&\ \
  R^{-n-2},\quad\quad\quad\quad n\geq 7,
\label{s-c-2}\\
 \int_{|y_1|\leq R,|y_2|\leq R} \frac{\l_1^{\frac{n+2}{2}}R^{2-n}}{\langle y_1\rangle^4}\frac{\lambda_2^{\frac{n-2}{2}}R^{2-n}}{\langle y_2\rangle^2}dx
 \lesssim&\ \  R^{-n-2},\quad \quad \quad \quad  n\geq 6,\label{s-c-3}\\
    \int_{|y_1|\leq R,|y_2|\geq R} \frac{\l_1^{\frac{n+2}{2}}R^{2-n}}{\langle y_1\rangle^4}\frac{\lambda_2^{\frac{n-2}{2}}R^{-4}}{\langle y_2\rangle^{n-4}}dx\lesssim&\
    \begin{cases}
            R^{-8}\log R,\quad &n=6,\\
            R^{-n-2},\quad &n\geq 7,
    \end{cases}
    \label{s-c-4}\\
     \int_{|y_1|\geq R,|y_2|\leq R} \frac{\l_1^{\frac{n+2}{2}}R^{-4}}{\langle y_1\rangle^{n-2}}\frac{\lambda_2^{\frac{n-2}{2}}R^{2-n}}{\langle y_2\rangle^{2}}dx\lesssim&\
     \begin{cases}
            R^{-8}\log R,\quad &n=6,\\
            R^{-n-2},\quad &n\geq 7,
        \end{cases}\label{s-c-5}\\
     \int_{|y_1|\geq R,|y_2|\geq R} \frac{\l_1^{\frac{n+2}{2}}R^{-4}}{\langle y_1\rangle^{n-2}}\frac{\lambda_2^{\frac{n-2}{2}}R^{-4}}{\langle y_2\rangle^{n-4}}dx\lesssim&\ \ R^{-n-2},\quad  \quad\quad\quad n\geq 7.
     \label{s-c-6}
\end{align}
\end{lemma}
\begin{proof}
(1) The integral in \eqref{s-c-1} is,
\begin{align}
\begin{split}
    R^{4-2n}\int_{|y_1|\leq R}\frac{\lambda_1^{n}dx}{\langle y_1\rangle^{6}}=&\
    R^{4-2n}\int_{|y_1|\leq R}\frac{dy_1}{\langle y_1\rangle^{6}}\thickapprox
    \begin{cases}
  R^{-8}\log R,\quad &n=6,\\
  R^{-n-2},\quad &n\geq 7,\\
\end{cases}
\end{split}
\end{align}

\noindent
(2) The integral in \eqref{s-c-2} is
\begin{align}
\begin{split}
    R^{-8}\int_{|y_1|\geq R}\frac{\lambda_1^{n}dx}{\langle y_1\rangle^{2n-6}}=&\
    R^{-8}\int_{|y_1|\geq R}\frac{dy_1}{\langle y_1\rangle^{2n-6}}
    \thickapprox R^{-n-2},\quad n\geq7.
\end{split}
\end{align}
(3)
\emph{The case $\lambda_1\geq \lambda_2$ and $R_{12}=\sqrt{\lambda_1/\lambda_2}$.} Set $\Omega_1=\{|y_1|\leq R, |y_1-z|\leq R/\lambda\}$, $z=\lambda_1(z_2-z_1)$ and $\lambda=\lambda_2/\lambda_1$. The integral in \eqref{s-c-3} is
\begin{align}\label{s-c-3-1}
    \begin{split}
        \int_{\Omega_1} \frac{\l_1^{\frac{n+2}{2}}R^{2-n}}{\langle y_1\rangle^4}\frac{\lambda_2^{\frac{n-2}{2}}R^{2-n}}{\langle y_2\rangle^2}dx=&\
        R^{4-2n}\left(\frac{\lambda_2}{\lambda_1}\right)^{\frac{n-4}{2}}\int_{\Omega_1}\frac{1}{\langle y_1\rangle^4}\frac{dy_1}{\l^{-1}+\l|y_1-z|^2}\\
        \lesssim&\
        R^{4-2n}R_{12}^{4-n}\int_{|y_1|\leq  R}\frac{dy_1}{\langle y_1\rangle^4}
        \lesssim R^{4-2n}.
    \end{split}
\end{align}
\emph{The case $\lambda_1\leq \lambda_2$ and $R_{12}=\sqrt{\lambda_2/\lambda_1}$.} Set $z=\lambda_2(z_1-z_2)$ and $\lambda=\lambda_1/\lambda_2$. The integral in \eqref{s-c-3} is, seeing that $\Omega_1=\{|y_2|\leq R, |y_2-z|\leq R/\lambda\}$,
\begin{align}\label{s-c-3-2}
    \begin{split}
        \int_{\Omega_1} \frac{\l_1^{\frac{n+2}{2}}R^{2-n}}{\langle y_1\rangle^4}\frac{\lambda_2^{\frac{n-2}{2}}R^{2-n}}{\langle y_2\rangle^2}dx=&\
        R^{4-2n}\left(\frac{\lambda_1}{\lambda_2}\right)^{\frac{n-2}{2}}\int_{\Omega_1}\frac{1}{\langle y_2\rangle^2}\frac{dy_2}{(\l^{-1}+\lambda|y_2-z|^2)^2}\\
        \lesssim&\
        R^{4-2n}R_{12}^{2-n}\int_{|y_2|\leq  R}\frac{dy_2}{\langle y_2\rangle^2}
        \lesssim R^{4-2n}.
    \end{split}
\end{align}
\emph{The case $\lambda_1\geq \lambda_2$ and $R_{12}=\sqrt{\lambda_1\lambda_2}|z_1-z_2|$.} Set $z=\lambda_1(z_2-z_1)$ and $\lambda=\lambda_2/\lambda_1$. In this case, if $\Omega_1\neq \emptyset$, then one must have $|x-z_1|\geq \frac12|z_1-z_2|$ (see Figure \ref{fig:bubble-cluster}), that is $|y_1-z|\geq \frac12|z|$ and $\l |y_1-z|^2\approx R_{12}^2$ in $\Omega_1$. Consequently, the integral in  \eqref{s-c-3} is
\begin{align}\label{s-c-3-3}
    \begin{split}
        \int_{\Omega_1} \frac{\l_1^{\frac{n+2}{2}}R^{2-n}}{\langle y_1\rangle^4}\frac{\lambda_2^{\frac{n-2}{2}}R^{2-n}}{\langle y_2\rangle^2}dx=&\
        R^{4-2n}\left(\frac{\lambda_2}{\lambda_1}\right)^{\frac{n-4}{2}}\int_{\Omega_1}\frac{1}{\langle y_1\rangle^4}\frac{dy_1}{\lambda^{-1}+\lambda|y_1-z|^2}\\
        \lesssim&\
        R^{4-2n}R_{12}^{-2}\int_{|y_1|\leq R}\frac{dy_1}{\langle y_1\rangle^4}
        \lesssim R^{-n-2}.
    \end{split}
\end{align}
\emph{The case $\lambda_1\leq \lambda_2$ and $R_{12}=\sqrt{\lambda_1\lambda_2}|z_1-z_2|$.} Set $z=\lambda_2(z_1-z_2)$ and $\lambda=\lambda_1/\lambda_2$. The integral in  \eqref{s-c-3} is
\begin{align}\label{s-c-3-4}
    \begin{split}
        \int_{\Omega_1} \frac{\l_1^{\frac{n+2}{2}}R^{2-n}}{\langle y_1\rangle^4}\frac{\lambda_2^{\frac{n-2}{2}}R^{2-n}}{\langle y_2\rangle^2}dx=&\
        R^{4-2n}\left(\frac{\lambda_1}{\lambda_2}\right)^{\frac{n-2}{2}}\int_{\Omega_1}\frac{1}{\langle y_2\rangle^2}\frac{dy_2}{(\l^{-1}+\l|y_2-z|^2)^2}\\
        \lesssim&\
        R^{4-2n}R_{12}^{-4}\int_{|y_2|\lesssim R}\frac{dy_2}{\langle y_2\rangle^2}
        \lesssim R^{-n-2}.
    \end{split}
\end{align}
From \eqref{s-c-3-1}-\eqref{s-c-3-4},  we get \eqref{s-c-3}.

\noindent
(4) To prove \eqref{s-c-4}, we shall rescale the LHS of it around $z_1$. For the following four cases, we set $\Omega_2=\{|y_1|\leq R,|y_2|\geq R\}$,  $z=\lambda_1(z_2-z_1)$ and $\lambda=\lambda_2/\lambda_1$.

\noindent
\emph{The case $\lambda_1\geq \lambda_2$ and $R_{12}=\sqrt{\lambda_1/\lambda_2}$.}
By the definition of $R_{12}$, we have $\lambda |z|\leq 1$. Since $y_2=\lambda (y_1-z)$, then $|y_2|\leq \lambda(|y_1|+|z|)\leq \lambda R+1$. By the assumption $R\gg1$, it holds that $\Omega_2=\emptyset$. Thus, the integral in \eqref{s-c-4} is zero.

\noindent
\emph{The case $\lambda_1\leq \lambda_2$ and $R_{12}=\sqrt{\lambda_2/\lambda_1}$.}  Remember that $|z|\leq1$ and $\l=R^2_{12}\gg1$.  The integral in \eqref{s-c-4} is, see that $\Omega_2=\{x: \lambda^{-1}R\leq |y_1-z|,\ |y_1|\leq R\}$,
\begin{align}\label{s-c-4-1}
    \begin{split}
        \int_{\Omega_2} \frac{\l_1^{\frac{n+2}{2}}R^{2-n}}{\langle y_1\rangle^4}&\frac{\lambda_2^{\frac{n-2}{2}}R^{-4}}{\langle y_2\rangle^{n-4}}dx=\ R^{-n-2}R_{12}^{6-n} \int_{\Omega_2}\frac{1}{\langle y_1\rangle^4}\frac{dy_1}{(\lambda^{-2}+|y_1-z|^2)^{\frac{n-4}{2}}}\\
     \thickapprox&\
     R^{-n-2}R_{12}^{6-n}\left(\int_{R/\l}^2t^3dt+\int_{2}^{R}\frac{dt}{t}\right)
     \lesssim R^{4-2n}\log R.
    \end{split}
\end{align}
\emph{The case $\lambda_1\geq \lambda_2$ and $R_{12}=\sqrt{\lambda_1\lambda_2}|z_1-z_2|$.} In $\Omega_2$, we have $|y_1|\leq |y_2|=\lambda|y_1-z|\leq |y_1-z|$. It follows that $|y_1-z|\geq|z|/2$ and $\lambda|y_1-z|^2\thickapprox R_{12}^2$ in $\Omega_2$. Then
\begin{align}\label{s-c-4-2}
    \begin{split}
        \int_{\Omega_2} \frac{\l_1^{\frac{n+2}{2}}R^{2-n}}{\langle y_1\rangle^4}\frac{\lambda_2^{\frac{n-2}{2}}R^{-4}}{\langle y_2\rangle^{n-4}}dx=&\
    R^{-n-2}\frac{\lambda_2}{\lambda_1} \int_{\Omega_2}\frac{1}{\langle y_1\rangle^4}\frac{dy_1}{(\lambda^{-1}+\l|y_1-z|^2)^{\frac{n-4}{2}}}\\
    \lesssim&\
    R^{-n-2}R_{12}^{4-n}\int_{|y_1|\leq R}\frac{dy_1}{\langle y_1\rangle^4}\lesssim R^{-n-2}.
    \end{split}
\end{align}
\emph{The case $\lambda_1\leq \lambda_2$ and $R_{12}=\sqrt{\lambda_1\lambda_2}|z_1-z_2|$.} Let $\hat{y}=y_1/|z|$ and $e=z/|z|$, then $\Omega_2=\{|\hat{y}|\leq \frac{R}{|z|}, |\hat{y}-e|\geq \frac{R}{\lambda |z|}\}$. Then we get
\begin{align}\label{s-c-4-3}
    \begin{split}
        \int_{\Omega_2} \frac{\l_1^{\frac{n+2}{2}}R^{2-n}}{\langle y_1\rangle^4}\frac{\lambda_2^{\frac{n-2}{2}}R^{-4}}{\langle y_2\rangle^{n-4}}dx=&\ R^{-n-2}\left(\frac{\lambda_2}{\lambda_1}\right)^{\frac{6-n}{2}} \int_{\Omega_2}\frac{1}{\langle y_1\rangle^4}\frac{dy_1}{(\lambda^{-2}+|y_1-z|^2)^{\frac{n-4}{2}}}\\
        \lesssim&\ R^{-n-2}\left(\frac{\lambda_2}{\lambda_1}\right)^{\frac{6-n}{2}}\left(\int_0^1t^{n-5}dt+\int_0^1t^{3}dt+\int_{1}^{R/|z|}\frac{dt}{t}\right)\\
        \thickapprox&\ R^{-n-2}\left(\frac{\lambda_2}{\lambda_1}\right)^{\frac{6-n}{2}}\log \left(\frac{R}{|z|}\right)\\
        \lesssim&\
        \begin{cases}
            R^{-n-2}\log R,\quad &n=6,\\
            R^{-n-2},\quad &n\geq 7.
        \end{cases}
    \end{split}
\end{align}
Together with \eqref{s-c-4-1}-\eqref{s-c-4-3}, we get \eqref{s-c-4}.

\noindent
(5) To prove \eqref{s-c-5}, we shall rescale the LHS of it around $z_2$. For the following four cases, we set $\Omega_3=\{x: \lambda^{-1}R\leq|y_2-z|, |y_2|\leq R\}$, $z=\lambda_2(z_1-z_2)$ and $\lambda=\lambda_1/\lambda_2$. 

\noindent
\emph{The case $\lambda_1\geq \lambda_2$ and $R_{12}=\sqrt{\lambda_1/\lambda_2}$.} Then  the integral in \eqref{s-c-5} is
\begin{align}\label{s-c-5-1}
    \begin{split}
        \int_{\Omega_3} \frac{\l_1^{\frac{n+2}{2}}R^{-4}}{\langle y_1\rangle^{n-2}}&\frac{\lambda_2^{\frac{n-2}{2}}R^{2-n}}{\langle y_2\rangle^{2}}dx=\ R^{-n-2}\left(\frac{\l_2}{\l_1}\right)^{\frac{6-n}{2}}\int_{\Omega_3}\frac{1}{(\lambda^{-2}+|y_2-z|^2)^{\frac{n-2}{2}}}\frac{dy_2}{\langle y_2\rangle^2}\\
         \thickapprox&\ R^{-n-2}R_{12}^{6-n}\left(\int_{R/\l}^2tdt+\int_{2}^{R}\frac{dt}{t}\right)
     \lesssim R^{4-2n}\log R.
    \end{split}
\end{align}
\emph{The case $\lambda_1\leq \lambda_2$ and $R_{12}=\sqrt{\lambda_2/\lambda_1}$.} By the definition of $R_{12}$, we have $\lambda |z|\leq 1$. Since $y_1=\lambda(y_2-z)$, we get $|y_1|\leq \lambda|y_2|+\lambda |z|\leq \l R+1$. By the assumption $R\gg1$, it holds that $\Omega_3=\emptyset$. Thus, the integral in \eqref{s-c-4} is zero.

\noindent
\emph{The case $\lambda_1\geq \lambda_2$ and $R_{12}=\sqrt{\lambda_1\lambda_2}|z_1-z_2|$.}  Let $\hat{y}=y_2/|z|$ and $e=z/|z|$, then $\Omega_3=\{|\hat{y}|\leq \frac{R}{|z|}, |\hat{y}-e|\geq \frac{R}{\lambda |z|}\}$. Notice that  $|z|\geq 1$
\begin{align}\label{s-c-5-2}
    \begin{split}
        \int_{\Omega_3} \frac{\l_1^{\frac{n+2}{2}}R^{-4}}{\langle y_1\rangle^{n-2}}\frac{\lambda_2^{\frac{n-2}{2}}R^{2-n}}{\langle y_2\rangle^{2}}dx=&\ R^{-n-2}\left(\frac{\lambda_1}{\l_2}\right)^\frac{6-n}{2}
        \int_{\Omega_3}\frac{1}{(\lambda^{-2}+|y_2-z|^2)^{\frac{n-2}{2}}}\frac{dy_2}{\langle y_2\rangle^2} \\
        \lesssim&\ R^{-n-2}\left(\frac{\lambda_1}{\lambda_2}\right)^{\frac{6-n}{2}}\left(\int_0^1tdt+\int_0^1t^{n-3}dt+\int_{1}^{R/|z|}\frac{dt}{t}\right)\\
        \thickapprox&\ R^{-n-2}\left(\frac{\lambda_1}{\lambda_2}\right)^{\frac{6-n}{2}}\log \left(\frac{R}{|z|}\right)\\
        \lesssim&\
        \begin{cases}
            R^{-n-2}\log R,\quad &n=6,\\
            R^{-n-2},\quad &n\geq 7.
        \end{cases}
    \end{split}
\end{align}
\emph{The case $\lambda_1\leq \lambda_2$ and $R_{12}=\sqrt{\lambda_1\lambda_2}|z_1-z_2|$.} In $\Omega_3$, we have $|y_2|\leq |y_1|=\lambda|y_2-z|\leq |y_2-z|$. It follows that $|y_2-z|\geq|z|/2$ and $\lambda|y_2-z|^2\thickapprox R_{12}^2$ in $\Omega_3$. Then
\begin{align}\label{s-c-5-3}
    \begin{split}
        \int_{\Omega_3} \frac{\l_1^{\frac{n+2}{2}}R^{-4}}{\langle y_1\rangle^{n-2}}\frac{\lambda_2^{\frac{n-2}{2}}R^{2-n}}{\langle y_2\rangle^{2}}dx=&\ R^{-n-2}\left(\frac{\lambda_1}{\l_2}\right)^2\int_{\Omega_3}\frac{1}{(\lambda^{-1}+\l|y_2-z|^2)^{\frac{n-2}{2}}}\frac{dy_2}{\langle y_2\rangle^2} \\
    \lesssim&\
    R^{-n-2}R_{12}^{2-n}\int_{|y_1|\leq R}\frac{dy_1}{\langle y_1\rangle^2}\lesssim R^{-n-2}.
    \end{split}
\end{align}
Together with \eqref{s-c-1}-\eqref{s-c-5-3}, we get \eqref{s-c-5}.

\noindent
(6) To prove \eqref{s-c-6}. We shall do it only for $n\geq 7$. Denote $\Omega_4=\{|y_1|\geq R, |y_2|\geq R\}$.

\noindent
\emph{The case $\lambda_1\geq\lambda_2$ and $R_{12}=\sqrt{\lambda_1/\lambda_2}$.} Set $z=\lambda_2(z_1-z_2)$ and $\lambda=\lambda_1/\lambda_2$.    Since $|z|\leq 1\ll R$, we have $|y_2-z|\geq |y_2|/2$ in $\Omega_4$. The integral in \eqref{s-c-6} is
\begin{align}\label{s-c-6-1}
    \begin{split}
         \int_{\Omega_4} \frac{\l_1^{\frac{n+2}{2}}R^{-4}}{\langle y_1\rangle^{n-2}}\frac{\lambda_2^{\frac{n-2}{2}}R^{-4}}{\langle y_2\rangle^{n-4}}dx=&\ R^{-8}R_{12}^{6-n}\int_{\Omega_4}\frac{1}{(\l^{-2}+|y_2-z|^2)^{\frac{n-2}{2}}}\frac{dy_2}{\langle y_2\rangle^{n-4}}\\
         \thickapprox&\
         R^{-8}R_{12}^{6-n}\int_{R}^{\infty}\frac{dt}{t^{n-5}}\lesssim R^{4-2n},\quad n\geq 7.
    \end{split}
\end{align}
\emph{The case $\lambda_1\leq\lambda_2$ and $R_{12}=\sqrt{\lambda_2/\lambda_1}$.} Set $z=\lambda_1(z_2-z_1)$ and $\lambda=\lambda_2/\lambda_1$. As before, we have
\begin{align}\label{s-c-6-2}
    \begin{split}
         \int_{\Omega_4} \frac{\l_1^{\frac{n+2}{2}}R^{-4}}{\langle y_1\rangle^{n-2}}\frac{\lambda_2^{\frac{n-2}{2}}R^{-4}}{\langle y_2\rangle^{n-4}}dx=&\ R^{-8}R_{12}^{6-n}\int_{\Omega_4}\frac{1}{\langle y_1\rangle^{n-2}}\frac{dy_1}{(\l^{-2}+|y_1-z|^2)^{\frac{n-4}{2}}}\\
         \thickapprox&\
         R^{-8}R_{12}^{6-n}\int_{R}^{\infty}\frac{dt}{t^{n-5}}\lesssim R^{4-2n},\quad n\geq 7.
    \end{split}
\end{align}
\emph{The case $\lambda_1\geq\lambda_2$ and $R_{12}=\sqrt{\lambda_1\lambda_2}|z_1-z_2|$.}  Set $z=\lambda_1(z_2-z_1)$ and $\lambda=\lambda_2/\lambda_1$. We also set $\Omega_4=\Omega_{41}\cup\Omega_{42}\cup\Omega_{43}$ where $\Omega_{41}=\{R\leq |y_1|\leq \frac{|z|}{2} \}$, $\Omega_{42}=\{R/\lambda\leq |y_1-z|\leq \frac{|z|}{2} \}$ and $\Omega_{43}=\{ |y_1|\geq \frac{|z|}{2}, |y_1-z|\geq \frac{|z|}{2} \}$. Remember that $R_{12}=\sqrt{\lambda}|z|$ and $|z|\geq \lambda^{-1}$. We have
\begin{align}\label{s-c-6-3}
    \begin{split}
         \int_{\Omega_4}& \frac{\l_1^{\frac{n+2}{2}}R^{-4}}{\langle y_1\rangle^{n-2}}\frac{\lambda_2^{\frac{n-2}{2}}R^{-4}}{\langle y_2\rangle^{n-4}}dx=\ R^{-8}\lambda^{\frac{6-n}{2}}\int_{\Omega_4}\frac{1}{\langle y_1\rangle^{n-2}}\frac{dy_1}{(\l^{-2}+|y_1-z|^2)^{\frac{n-4}{2}}}\\
         \lesssim&\ R^{-8}\lambda^{\frac{6-n}{2}}\left(|z|^{4-n}\int_{\Omega_{41}}\frac{dy_1}{\langle y_1\rangle^{n-2}}+|z|^{2-n}\int_{\Omega_{42}}\frac{dy_1}{\langle y_1-z\rangle^{n-4}}+\int_{|y_1|\geq |z|}\frac{dy_1}{|y_1|^{2n-6}}\right)\\
         \thickapprox&\ R^{-8}R_{12}^{6-n}\lesssim R^{-n-2},\quad n\geq 7.
    \end{split}
\end{align}
In the fourth line of \eqref{s-c-6-3}, we use the fact that, for each $y_1\in\Omega_{43}$, that $|y_1-z|\geq \frac{|z|}{2}\geq \frac{|y_1|}{4}$ if $|y_1|\leq 2|z|$ while it holds that  $|y_1-z|\geq |y_1|-|z|\geq \frac{|y_1|}{2}$ if $|y_1|\geq 2|z|$.

\noindent
\emph{The case $\lambda_1\leq\lambda_2$ and $R_{12}=\sqrt{\lambda_1\lambda_2}|z_1-z_2|$.}  Set $z=\lambda_2(z_1-z_2)$ and $\lambda=\lambda_1/\lambda_2$. We also set $\Omega_4=\Omega_{41}\cup\Omega_{42}\cup\Omega_{43}$ where $\Omega_{41}=\{R\leq |y_2|\leq \frac{|z|}{2} \}$, $\Omega_{42}=\{R/\lambda\leq |y_2-z|\leq \frac{|z|}{2} \}$ and $\Omega_{43}=\{ |y_2|\geq \frac{|z|}{2}, |y_2-z|\geq \frac{|z|}{2} \}$. As before,  we have
\begin{align}\label{s-c-6-4}
    \begin{split}
         \int_{\Omega_4}& \frac{\l_1^{\frac{n+2}{2}}R^{-4}}{\langle y_1\rangle^{n-2}}\frac{\lambda_2^{\frac{n-2}{2}}R^{-4}}{\langle y_2\rangle^{n-4}}dx=\  R^{-8}\lambda^{\frac{6-n}{2}}\int_{\Omega_4}\frac{1}{\langle y_2\rangle^{n-4}}\frac{dy_2}{(\l^{-2}+|y_2-z|^2)^{\frac{n-2}{2}}}\\
         \lesssim&\ R^{-8}\lambda^{\frac{6-n}{2}}\left(|z|^{2-n}\int_{\Omega_{41}}\frac{dy_1}{\langle y_2\rangle^{n-4}}+|z|^{4-n}\int_{\Omega_{42}}\frac{dy_1}{\langle y_2-z\rangle^{n-2}} +\int_{\Omega_{43}}\frac{dy_1}{|y_2|^{2n-6}}\right)\\
         \thickapprox&\ R^{-8}R_{12}^{6-n}\lesssim R^{-n-2},\quad n\geq 7.
    \end{split}
\end{align}
Together with \eqref{s-c-6-1}-\eqref{s-c-6-4}, we get \eqref{s-c-6}.
\end{proof}
The inequality \eqref{s-c-2} and  \eqref{s-c-6} shows that we should keep the decay of error term $(\sum U_i)^2-\sum U_i^2=\sum_{i\neq j}U_iU_j$ near infinity in dimension six. In the following lemma, we directly  estimate $\int U_1U_2\rho_0$.
\begin{lemma}\label{lem:6-dim}
Suppose $n=6$ and $1\ll R\leq \frac{1}{2}\min\{R_{12}, R_{13},R_{23}\}$,   then
\begin{align}
        \sum_{j=1}^3\int U_1(x)U_2(x)\frac{\l_j^{2}R^{-4}}{\langle y_j\rangle^2}dx\lesssim&\ R^{-8}\log R.
        \label{6-dim-1}
\end{align}
\end{lemma}
\begin{proof}
From the proof of Proposition  \ref{prop:on-weight} (in particular \eqref{h-t-4} and \eqref{h-c-4}), we see that, for $n=6$,
\begin{align}
    \begin{split}
        U_1(x)U_2(x)\lesssim&\
        \sum_{i=1}^2 \left(\frac{\lambda_i^{4}R_{12}^{-4}}{\langle y_i\rangle^4}\chi_{\{|y_i|\leq  R_{12}^2\}}+\frac{\lambda_i^4R_{12}^{-4}}{\langle y_i
        \rangle^6}\chi_{\{|y_i|\geq R_{12}^2\}}\right).
    \end{split}
\end{align}
Then
\begin{align*}
\begin{split}
     &\sum_{j=1}^3\int U_1U_2\frac{\l_j^{2}R^{-4}}{\langle y_j\rangle^2}=\
  \sum_{i=1}^2\sum_{j=1}^3  \left(\int_{ |y_i|\leq R_{12}^2}\frac{\lambda_i^{4}R_{12}^{-4}}{\langle y_i\rangle^4}\frac{\lambda_j^{2}R^{-4}}{\langle y_j\rangle^2}+  \int_{|y_i|\geq R_{12}^2}\frac{\lambda_i^{4}R_{12}^{-4}}{\langle y_i\rangle^6}\frac{\lambda_j^{2}R^{-4}}{\langle y_j\rangle^2}\right).
   \end{split}
\end{align*}
Consider the first integral on the RHS of the above equation.
By \eqref{s-c-3} and \eqref{s-c-4} of Lemma \ref{lem:s-U-rho}, it is already bounded above by $R^{-8}\log R$ when integrating on the region $\{|y_i|<R\}$ or $\{|y_j|<R\}$. The same thing happens for the second integral on $\{|y_j|<R\}$. Therefore, to establish \eqref{6-dim-1}, it reduces to prove the following
\begin{align}\label{U1U2rho}
  \sum_{i=1}^2\sum_{j=1}^3 \left(\int_{\substack{R\leq |y_i|\leq R_{12}^2\\
  |y_j|\geq R}}\frac{\lambda_i^{4}R_{12}^{-4}}{\langle y_i\rangle^4}\frac{\lambda_j^{2}R^{-4}}{\langle y_j\rangle^2}+\int_{\substack{|y_i|\geq R_{12}^2\\ |y_j|\geq R}}\frac{\lambda_i^{4}R_{12}^{-4}}{\langle y_i\rangle^6}\frac{\lambda_j^{2}R^{-4}}{\langle y_j\rangle^2}\right)\lesssim R^{-8}\log R.
\end{align}
It suffices to compute the following four types of integral.

\noindent
(1) Suppose $i=j$ for the first integral in \eqref{U1U2rho}. Since $n=6$, we have
\begin{align}\label{6d-1-1}
    \begin{split}
        \int_{R\leq |y_i|\leq R_{12}^2}\frac{\lambda_i^4R_{12}^{-4}}{\langle y_i\rangle^4}\frac{\lambda_i^2R^{-4}}{\langle y_i\rangle^2}dx= R_{12}^{-4}R^{-4}\int_{R\leq |y_i|\leq R_{12}^2}\frac{dy_1}{\langle y_i\rangle^6}\lesssim R^{-8}\log R.
    \end{split}
\end{align}

\noindent
(2)  Suppose $i=j$ for the second integral in \eqref{U1U2rho}. Since $n=6$, we have
\begin{align}\label{6d-1-2}
    \begin{split}
        \int_{ |y_i|\geq R_{12}^2}\frac{\lambda_i^4R_{12}^{-4}}{\langle y_i\rangle^6}\frac{\lambda_i^2R^{-4}}{\langle y_i\rangle^2}dx= R_{12}^{-4}R^{-4}\int_{ |y_i|\geq R_{12}^2}\frac{dy_i}{\langle y_i\rangle^8}\lesssim R^{-12}.
    \end{split}
\end{align}
(3) Suppose $i\neq j$ for the first integral in \eqref{U1U2rho}. Denote $\Omega_1=\{x: R\leq |y_i|\leq R_{12}^2, |y_j|\geq R\}$, we need to consider three cases.

\noindent
\emph{The case $\lambda_i\geq\lambda_j$ and $R_{ij}=\sqrt{\lambda_i/\lambda_j}$.} Set  $z=\lambda_j(z_i-z_j)$ and $\lambda=\lambda_i/\lambda_j$. Since $|z|\leq 1$, we have $|y_j-z|/2\leq |y_j|\leq 2|y_j-z|$.  Then $\Omega_1\subset\{R/2\lambda\leq|y_j|\leq 2R_{12}^2/\lambda\}$. Similar to \eqref{s-c-6-1}, we have
\begin{align}\label{6d-1-3-1}
    \begin{split}
        \int_{\Omega_1}\frac{\lambda_i^{4}R_{12}^{-4}}{\langle y_i\rangle^4}\frac{\lambda_j^{2}R^{-4}}{\langle y_j\rangle^2}dx
   =&\
   R_{12}^{-4}R^{-4}\int_{\Omega_1} \frac{1}{(\lambda^{-2}+|y_j-z|^2)^2}\frac{dy_j}{\langle y_j\rangle^2}\\
   \lesssim&\  R_{12}^{-4}R^{-4}\int_{R/2\l}^{2R_{12}^2/\l}t^{-1}dt\lesssim R^{-8}\log R.
    \end{split}
\end{align}
\emph{The case $\lambda_i\leq\lambda_j$ and $R_{ij}=\sqrt{\lambda_j/\lambda_i}$.} Set $z=\lambda_i(z_j-z_i)$ and $\lambda=\lambda_j/\lambda_i$. As before, we have
\begin{align}\label{6d-1-3-2}
    \begin{split}
        \int_{\Omega_1}\frac{\lambda_i^{4}R_{12}^{-4}}{\langle y_i\rangle^4}\frac{\lambda_j^{2}R^{-4}}{\langle y_j\rangle^2}dx
   =&\
   R_{12}^{-4}R^{-4}\int_{\Omega_1} \frac{1}{\langle y_i\rangle^4}\frac{dy_i}{\lambda^{-2}+ |y_i-z|^2}\\
   \lesssim&\  R_{12}^{-4}R^{-4}\int_{R}^{R_{12}^2}t^{-1}dt\lesssim R^{-8}\log R.
    \end{split}
\end{align}
\emph{The case  $R_{ij}=\sqrt{\lambda_i\lambda_j}|z_i-z_j|$.}  Set $z=\lambda_i(z_j-z_i)$ and $\lambda=\lambda_j/\lambda_i$. We also set $\Omega_1=\Omega_{11}\cup\Omega_{12}\cup\Omega_{13}$ where $\Omega_{11}=\{R\leq |y_i|\leq \frac{|z|}{2} \}$, $\Omega_{12}=\{R/\l\leq |y_i-z|\leq \frac{|z|}{2} \}$ and $\Omega_{13}=\{  \frac{|z|}{2}\leq |y_i|\leq R_{12}^2, |y_i-z|\geq \frac{|z|}{2} \}$.   Then, similar to \eqref{s-c-6-3}, we have
\begin{align}\label{6d-1-3-3}
    \begin{split}
         \int_{\Omega_1}&\frac{\lambda_i^{4}R_{12}^{-4}}{\langle y_i\rangle^4}\frac{\lambda_j^{2}R^{-4}}{\langle y_j\rangle^2}dx
   =\  R_{12}^{-4}R^{-4}\int_{\Omega_1}\frac{1}{\langle y_i\rangle^{4}}\frac{dy_i}{(\l^{-2}+|y_i-z|^2)}\\
         \lesssim&\ R_{12}^{-4}R^{-4}\left(|z|^{-2}\int_{\Omega_{11}}\frac{dy_i}{\langle y_i\rangle^{4}}+|z|^{-4}\int_{\Omega_{12}}\frac{dy_i}{\langle y_i-z\rangle^{2}} +\int_{\Omega_{13}}\frac{dy_i}{|y_i|^{6}}\right)\\
         \lesssim&\
         \begin{cases}
         R^{-8}\quad &\text{if}\ 2R_{12}^2\leq |z|,\\
         R^{-8}\log R\quad &\text{if}\ 2R_{12}^2\geq |z|.
         \end{cases}
    \end{split}
\end{align}
Together with \eqref{6d-1-3-1}, \eqref{6d-1-3-2} and \eqref{6d-1-3-3}, we have
\begin{align}\label{6d-1-3}
    \begin{split}
        \int_{\substack{R\leq |y_i|\leq R_{12}^2\\
   |y_j|\geq R}}\frac{\lambda_i^{4}R_{12}^{-4}}{\langle y_i\rangle^4}\frac{\lambda_j^{2}R^{-4}}{\langle y_j\rangle^2}dx
  \lesssim R^{-8}\log R.
    \end{split}
\end{align}
(4)
Suppose $i\neq j$ for the second integral in \eqref{U1U2rho}.
Denote $\Omega_2=\{x:  |y_i|\geq R_{12}^2, |y_j|\geq R\}$. we need to consider three cases.

\noindent
\emph{The case $\lambda_i\geq \lambda_j$.}
 Let $z=\lambda_i(z_j-z_i)$, $\lambda=\lambda_j/\lambda_i\leq1$, $U=U(y_i)=U[0,1](y_i)$ and $\tilde{U}=\tilde{U}(y_i)=U[z,\lambda](y_i)$.
 Then
\begin{align}\label{6d-1-4-1-holder}
    \begin{split}
          \int_{\Omega_2}\frac{\lambda_i^{4}R_{12}^{-4}}{\langle y_i\rangle^6}&\frac{\lambda_j^{2}R^{-4}}{\langle y_j\rangle^2}dx =\   R_{12}^{-4}R^{-4} \lambda \int_{\Omega_2}U^{\frac{3}{2}}\tilde{U}^{\frac{1}{2}}dy_i\\
  \leq&\
  R^{-8}\left(\int_{|y_i|\geq R_{12}^{2}}\frac{dy_i}{\langle y_i\rangle^{\frac{20}{3}}}\right)^{\frac{3}{4}}
  \left(\int U\tilde{U}^{2}dy_i\right)^{\frac{1}{4}}
  \lesssim R^{-10}.
    \end{split}
\end{align}
\emph{The case $\lambda_i\leq\lambda_j$ and $R_{ij}=\sqrt{\lambda_j/\lambda_i}$.} Set $z=\lambda_i(z_j-z_i)$ and $\lambda=\lambda_j/\lambda_i$. Similar to \eqref{s-c-6-1}, we have
\begin{align}\label{6d-1-4-2}
    \begin{split}
        \int_{\Omega_2}\frac{\lambda_i^{4}R_{12}^{-4}}{\langle y_i\rangle^6}\frac{\lambda_j^{2}R^{-4}}{\langle y_j\rangle^2}dx
  =&\
  R_{12}^{-4}R^{-4}\int_{\Omega_2} \frac{1}{\langle y_i\rangle^6}\frac{dy_i}{\lambda^{-2}+ |y_i-z|^2}\\
  \lesssim&\  R^{-8}\int_{R_{12}^2}^{\infty}t^{-3}dt\lesssim R^{-12}.
    \end{split}
\end{align}
\emph{The case $\lambda_i\leq\lambda_j$ and $R_{ij}=\sqrt{\lambda_i\lambda_j}|z_i-z_j|$.}  Set $z=\lambda_j(z_i-z_j)$ and $\lambda=\lambda_i/\lambda_j$. We also set $\Omega_2=\Omega_{21}\cup\Omega_{22}\cup\Omega_{23}$ where $\Omega_{21}=\{R\leq |y_j|\leq \frac{|z|}{2} \}$, $\Omega_{22}=\{R/\lambda\leq |y_j-z|\leq \frac{|z|}{2} \}$ and $\Omega_{23}=\{ |y_j|\geq \max\{ \frac{|z|}{2}, R_{12}^2\}, |y_j-z|\geq \frac{|z|}{2} \}$. Then, noticing $|z|\geq R$, we have
\begin{align}\label{6d-1-4-4}
    \begin{split}
      \int_{\Omega_2}\frac{\lambda_i^{4}R_{12}^{-4}}{\langle y_i\rangle^5}&\frac{\lambda_j^{2}R^{-4}}{\langle y_j\rangle^2}dx
  =\   R_{12}^{-4}R^{-4}\int_{\Omega_2}\frac{1}{(\l^{-2}+|y_j-z|^2)^{6}}\frac{dy_j}{\langle y_j\rangle^{2}}\\
         \lesssim&\ R^{-8}\left(\frac{1}{|z|^6}\int_{\Omega_{21}}\frac{dy_j}{\langle y_j\rangle^{2}}+\frac{1}{|z|^2}\int_{\Omega_{22}}\frac{dy_j}{| y_j-z|^{6}}+\int_{\Omega_{23}}\frac{dy_j}{|y_j|^8}\right)
         \lesssim
         R^{-10}.
    \end{split}
\end{align}
Together with \eqref{6d-1-4-1-holder}-\eqref{6d-1-4-4}, we have
        \begin{align}\label{6d-1-4}
    \begin{split}
      \int_{\substack{|y_i|\geq R_{12}^2\\ |y_j|\geq R}}\frac{\lambda_i^{4}R_{12}^{-4}}{\langle y_i\rangle^6}\frac{\lambda_j^{2}R^{-4}}{\langle y_j\rangle^2}dx
  \lesssim R^{-10}.
    \end{split}
\end{align}
\end{proof}

\begin{lemma}\label{lem:V-U}
Suppose $n\geq 6$ and $1\ll R\leq R_{12}/2$, we have
\begin{align}
\int_{|y_1|\leq R} \frac{\l_1^{\frac{n+2}{2}}R^{2-n}}{\langle y_1\rangle^4}\frac{\lambda_1^{\frac{n-2}{2}}}{\langle y_1\rangle^{n-2}}dx\thickapprox&\ R^{2-n},
\label{V-U-1}\\
\int_{|y_1|\geq R} \frac{\l_1^{\frac{n+2}{2}}R^{-4}}{\langle y_1\rangle^{n-2}}\frac{\lambda_1^{\frac{n-2}{2}}}{\langle y_1\rangle^{n-2}}dx\thickapprox&\ R^{-n},
\label{v-U-2}\\
 \int_{|y_1|\leq R} \frac{\l_1^{\frac{n+2}{2}}R^{2-n}}{\langle y_1\rangle^4}\frac{\lambda_2^{\frac{n-2}{2}}}{\langle y_2\rangle^{n-2}}dx
 \lesssim&\ R^{-n},\label{V-U-3}\\
    \int_{|y_1|\geq R} \frac{\l_1^{\frac{n+2}{2}}R^{-4}}{\langle y_1\rangle^{n-2}}\frac{\lambda_2^{\frac{n-2}{2}}}{\langle y_2\rangle^{n-2}}dx\lesssim&\ R^{2-n}.
    \label{V-U-4}
\end{align}
\end{lemma}
\begin{proof}
(1) The \eqref{V-U-1} and \eqref{v-U-2} are derived by  simple  computations similar to \eqref{s-c-1} and \eqref{s-c-2}.

\noindent
(2) Let $z=\lambda_1(z_2-z_1)$ and $\lambda=\lambda_2/\lambda_1$.

\noindent
\emph{The case $\lambda_1\geq\lambda_2$ and $R_{12}=\sqrt{\lambda_1/\lambda_2}$.} We have
\begin{align}\label{V-U-3-1}
    \begin{split}
        \int_{|y_1|\leq R} \frac{\l_1^{\frac{n+2}{2}}R^{2-n}}{\langle y_1\rangle^4}\frac{\lambda_2^{\frac{n-2}{2}}}{\langle y_2\rangle^{n-2}}dx
        =&\
        R^{2-n}R_{12}^{2-n}\int_{|y_1|\leq R}\frac{1}{\langle y_1\rangle^4}\frac{dy_1}{(1+\lambda^2|y_1-z|^2)^{\frac{n-2}{2}}}\\
        \lesssim&\ R^{4-2n}\int_{|y_1|\leq R}\frac{dy_1}{\langle y_1\rangle^4}\thickapprox R^{-n}.
    \end{split}
\end{align}
\emph{The case $\lambda_1\leq\lambda_2$ and $R_{12}=\sqrt{\lambda_2/\lambda_1}$.}
 We have $|z|\leq1$ . Then
\begin{align}\label{V-U-3-2}
    \begin{split}
        \int_{|y_1|\leq R} \frac{\l_1^{\frac{n+2}{2}}R^{2-n}}{\langle y_1\rangle^4}\frac{\lambda_2^{\frac{n-2}{2}}}{\langle y_2\rangle^{n-2}}dx
        =&\
        R^{2-n}R_{12}^{2-n}\int_{|y_1|\leq R}\frac{1}{\langle y_1\rangle^4}\frac{dy_1}{(\lambda^{-2}+|y_1-z|^2)^{\frac{n-2}{2}}}\\
        \lesssim&\ R^{4-2n} \left(\int_{2}^Rt^{n-5}dt+\int_0^2tdt\right)\thickapprox R^{-n}.
    \end{split}
\end{align}
\emph{The case $\lambda_1\geq\lambda_2$ and $R_{12}=\sqrt{\lambda_1\lambda_2}|z_1-z_2|$.}
 Let $\hat{y}=y_1/|z|$ and $e=z/|z|$, we have $\lambda\leq1$ and $\sqrt{\lambda}|z|\geq 2R$, then
\begin{align}\label{V-U-3-3}
    \begin{split}
        \int_{|y_1|\leq R} \frac{\l_1^{\frac{n+2}{2}}R^{2-n}}{\langle y_1\rangle^4}\frac{\lambda_2^{\frac{n-2}{2}}}{\langle y_2\rangle^{n-2}}dx
        \thickapprox&\
        \int_{|\hat{y}|\leq R/|z|}\frac{R^{2-n}\lambda^{\frac{2-n}{2}}|z|^{-2}}{(|z|^{-1}+|\hat{y}|)^4}\frac{d\hat{y}}{(\lambda^{-1}|z|^{-1}+|\hat{y}-e|)^{n-2}}\\
        \lesssim&\ R^{2-n}|z|^{-2}\lambda^{\frac{2-n}{2}} \left(\int_{0}^{\sqrt{\lambda}}t^{n-5}dt\right)\\
        \thickapprox&\ R^{2-n}(\sqrt{\lambda}|z|)^{-2}\leq R^{-n}.
    \end{split}
\end{align}
\emph{The case $\lambda_1\leq\lambda_2$ and $R_{12}=\sqrt{\lambda_1\lambda_2}|z_1-z_2|$.} As before, we have $\lambda\geq1$ and $\sqrt{\lambda}|z|\geq2R$, then
\begin{align}\label{V-U-3-4}
    \begin{split}
        \int_{|y_1|\leq R} \frac{\l_1^{\frac{n+2}{2}}R^{2-n}}{\langle y_1\rangle^4}&\frac{\lambda_2^{\frac{n-2}{2}}}{\langle y_2\rangle^{n-2}}dx
        \thickapprox\
        \int_{|\hat{y}|\leq R/|z|}\frac{R^{2-n}\lambda^{\frac{2-n}{2}}|z|^{-2}}{(|z|^{-1}+|\hat{y}|)^4}\frac{dy_1}{(\lambda^{-1}|z|^{-1}+|\hat{y}-e|)^{n-2}}\\
        \lesssim&\ R^{2-n}|z|^{-2}\lambda^{\frac{2-n}{2}} \left(\int_{2}^{\sqrt{\lambda}}t^{n-5}dt+\int_{0}^2tdt\right)
        \thickapprox  R^{-n}.
    \end{split}
\end{align}
Together with \eqref{V-U-3-1}-\eqref{V-U-3-4}, we get \eqref{V-U-3}.

(3) Let $z=\lambda_2(z_1-z_2)$ and $\lambda=\lambda_1/\lambda_2$.

\noindent
\emph{The case $\lambda_1\geq \lambda_2$ and $R_{12}=\sqrt{\lambda_1/\lambda_2}$.} We have $|z|\leq1$, then
\begin{align}\label{V-U-4-1}
    \begin{split}
        \int_{|y_1|\geq R} \frac{\lambda_1^{\frac{n+2}{2}}R^{-4}}{\langle y_1\rangle^{n-2}}\frac{\lambda_2^{\frac{n-2}{2}}}{\langle y_2\rangle^{n-2}}dx\thickapprox&\ \int_{|y_2-z|\geq R/\lambda}\frac{R^{-4}\lambda^{\frac{2-n}{2}}}{(\lambda^{-1}+|y_2-z|)^{n-2}}\frac{dy_2}{\langle y_2\rangle^{n-2}}\\
        \lesssim&\ R^{-2-n}\left(\int_{R/\lambda}^2tdt+\int_{2}^\infty t^{4-2n}dt\right)\thickapprox R^{-n-2}.
    \end{split}
\end{align}
\emph{The case $\lambda_1\leq \lambda_2$ and $R_{12}=\sqrt{\lambda_2/\lambda_1}$.} We have $|z|\leq\lambda^{-1}$, then $|y_2|=|y_2-z|-|z|\geq |y_2-z|/2 $,
\begin{align}\label{V-U-4-2}
    \begin{split}
        \int_{|y_1|\geq R} \frac{\lambda_1^{\frac{n+2}{2}}R^{-4}}{\langle y_1\rangle^{n-2}}\frac{\lambda_2^{\frac{n-2}{2}}}{\langle y_2\rangle^{n-2}}dx\thickapprox&\ \int_{|y_2-z|\geq R/\lambda}\frac{R^{-4}\lambda^{\frac{2-n}{2}}}{(\lambda^{-1}+|y_2-z|)^{n-2}}\frac{dy_2}{\langle y_2\rangle^{n-2}}\\
        \lesssim&\ R^{-4}\lambda^{\frac{2-n}{2}}\int_{R/\lambda}^{\infty}t^{3-n}dt\lesssim R^{-n}.
    \end{split}
\end{align}
\emph{The case $\lambda_1\geq \lambda_2$ and $R_{12}=\sqrt{\lambda_1\lambda_2}|z_1-z_2|$.} We have $\sqrt{\lambda}|z|\geq2R$ and $\lambda\geq1$, then
\begin{align}\label{V-U-4-3}
    \begin{split}
        \int_{|y_1|\geq R}& \frac{\lambda_1^{\frac{n+2}{2}}R^{-4}}{\langle y_1\rangle^{n-2}}\frac{\lambda_2^{\frac{n-2}{2}}}{\langle y_2\rangle^{n-2}}dx\thickapprox\ \int_{|y_2-z|\geq R/\lambda}\frac{R^{-4}\lambda^{\frac{2-n}{2}}}{(\lambda^{-1}+|y_2-z|)^{n-2}}\frac{dy_2}{\langle y_2\rangle^{n-2}}\\
        \lesssim&\ R^{-4}\lambda^{\frac{2-n}{2}}|z|^{2-n}\left(\int_{R/\lambda}^{|z|/2}tdt+\int_{0}^{|z|/2}tdt\right) +R^{-4}\lambda^{\frac{2-n}{2}}\int_{|z|/2}^{\infty}t^{3-n}dt\\
        \lesssim&\ R^{-n}.
    \end{split}
\end{align}
\emph{The case $\lambda_1\leq \lambda_2$ and $R_{12}=\sqrt{\lambda_1\lambda_2}|z_1-z_2|$.} We have $\sqrt{\lambda}|z|\geq2R$ and $\lambda|z|\geq1$, then
\begin{align}\label{V-U-4-4}
    \begin{split}
        \int_{|y_1|\geq R}& \frac{\lambda_1^{\frac{n+2}{2}}R^{-4}}{\langle y_1\rangle^{n-2}}\frac{\lambda_2^{\frac{n-2}{2}}}{\langle y_2\rangle^{n-2}}dx\thickapprox\ \int_{|y_2-z|\geq R/\lambda}\frac{R^{-4}\lambda^{\frac{2-n}{2}}}{(\lambda^{-1}+|y_2-z|)^{n-2}}\frac{dy_2}{\langle y_2\rangle^{n-2}}\\
        \lesssim&\ R^{-4}\lambda^{\frac{2-n}{2}}|z|^{2-n}\left(\int_{R/\lambda}^{|z|/2}tdt+\int_{0}^{|z|/2}tdt\right) +R^{-4}\lambda^{\frac{2-n}{2}}\int_{|z|/2}^{\infty}t^{4-2n}dt\\
        \lesssim&\ R^{-4}(\sqrt{\lambda}|z|)^{6-n}(\lambda|z|)^{-2}\leq R^{2-n}.
    \end{split}
\end{align}
Together with \eqref{V-U-4-1}-\eqref{V-U-4-4}, we get \eqref{V-U-4}.
\end{proof}

\bibliographystyle{plainnat}
\bibliography{ref}

\end{document}